%% file: main.tex
\documentclass[english]{amsart}
\usepackage{amsmath} 
\usepackage{amssymb} 
\usepackage{amsthm}
\numberwithin{equation}{section}
\usepackage{graphicx}
\usepackage{geometry}
\usepackage{xcolor}
\usepackage{esint}
\usepackage{bm}
\usepackage{tikz}
\usepackage{comment}
\usepackage[nocompress]{cite}
\usepackage[colorlinks=true, allcolors=blue]{hyperref}
\usepackage{cleveref}

\newcommand{\OmegaT}{\Omega_T}
\newcommand{\uvec}{\mathbf{u}}

\newcommand{\vvec}{\mathbf{v}}
\newcommand{\psivec}{\mathbf{\Psi}}
\newcommand{\varphivec}{\bm{\varphi}}
\newcommand{\phivec}{\bm{\phi}}
\newcommand{\RR}{\mathbb{R}}
\newcommand{\NN}{\mathbb{N}}
\renewcommand{\div}{\mathrm{div}_x}
\newcommand{\nablax}{\nabla_x}
\newcommand{\Svisc}{\mathbb{S}}
\newcommand{\Deltax}{\Delta_x}
\newcommand{\dd}{\mathrm{d}}
\newcommand{\weak}{\rightharpoonup}
\newcommand{\weakstar}{\overset{\ast}{\weak}}
\newcommand{\rhod}{\rho_{\delta}}
\newcommand{\uvecd}{\uvec_{\delta}}
\newcommand{\cd}{c_\delta}
\newcommand{\tgamma}{{\tilde{\gamma}}}
\newcommand{\dt}{\mathrm{d}t}
\newcommand{\dx}{\mathrm{d}x}
\newcommand{\Cw}{C_{\mathrm{w}}} 
\newcommand{\shearvisc}{\mu} 
\newcommand{\bulkvisc}{\eta} 
\newcommand{\coup}{\alpha} 
\newcommand{\paracoup}{\beta} 
\newcommand{\artgrwth}{\Gamma} 
\newcommand{\artprs}{\delta} 
\newcommand{\artvisc}{\varepsilon} 
\newcommand{\proj}{\mathbb{P}_n} 
\newcommand{\nonlin}{\mathcal{N}} 
\newcommand{\Tloc}{T_{\mathrm{loc}}} 
\newcommand{\rhoepsdelt}{\rho_{\artprs,\artvisc}} 
\newcommand{\uvecepsdelt}{\uvec_{\artprs,\artvisc}} 
\newcommand{\cepsdelt}{c_{\artprs,\artvisc}} 
\newcommand{\Peff}{p_\coup} 
\newcommand{\eps}{\varepsilon} 
\newcommand{\Bog}{\mathcal{B}} 
\newcommand{\dimension}{d} 
\newcommand{\qvec}{\mathbf{q}} 
\newcommand{\Fvec}{\mathbf{F}} 
\newcommand{\Gvec}{\mathbf{G}} 
\newcommand{\data}{{\mathtt{data}}}
\newcommand{\Cdata}{C(\data)}
\newcommand{\constD}{D_0}

\newtheorem{theorem}{Theorem}[section]
\newtheorem*{theorem*}{Theorem}
\newtheorem{proposition}{Proposition}[section]

\newtheorem{lemma}{Lemma}[section]
\theoremstyle{remark}
\newtheorem{remark}{Remark}[section]
\allowdisplaybreaks
\theoremstyle{definition}
\newtheorem{definition}{Definition}[section]

\title[Existence of weak solutions to a relaxed Navier--Stokes--Korteweg model]{Global-in-time existence of finite energy weak solutions to a relaxed Navier--Stokes--Korteweg model}

\author{Florian Oschmann}
\address{Faculty of Mathematics and Physics of the Charles University, Sokolovsk\'a 49/83, CZ-186 00 Praha, Czech Republic}
\email{\href{mailto:florian.oschmann@matfyz.cuni.cz}{florian.oschmann@matfyz.cuni.cz}}

\author{Florian Wendt}
\address{Institute of Applied Analysis and Numerical Simulation, University of Stuttgart, Pfaffenwaldring~57, D-70569 Stuttgart, Germany}
\email{\href{mailto:florian.wendt@mathematik.uni-stuttgart.de}{florian.wendt@mathematik.uni-stuttgart.de}}

\date{\today}
\keywords{Navier--Stokes--Korteweg equations; relaxation system; two-phase flow; finite energy weak solution}
\subjclass[2020]{35Q35; 76N10; 76T10; 35Q30}

\begin{document}

\begin{abstract}
    We consider a parabolic relaxation formulation of the compressible Navier--Stokes--Korteweg system and prove the global-in-time existence of finite energy weak solutions to the associated initial-boundary-value-problem.
    Our proof is based on a three-level approximation scheme, a weak compactness property of the effective viscous flux, and parabolic regularity estimates.
    Our result holds for a broad variety of non-monotone pressure functions and generalizes the corresponding results known for the compressible Navier--Stokes equations to the relaxation system.
\end{abstract}

\maketitle

\section{Introduction}\label{sec:Intro}
\input{Introduction}

\section{Admissible pressure functions, weak solutions, and main result}\label{sec:wkSol}
\input{WkSol}

\section{Global-in-time existence for the regularized system}\label{sec:exSolreg}
\input{exSol}

\section{Vanishing artificial viscosity limit}\label{sec:VanArtVisc}
\input{VanArtVisc}

\section{Vanishing artificial pressure limit}\label{sec:VanArtPres}
\input{VanArtPres}

\section{Conclusions and future directions}\label{sec:conclusions}
\input{conclusions}

\section*{Acknowledgments}
{\it F. O. has been supported by the Primus grant PRIMUS 26/SCI/026. F. W. acknowledges funding by Deutsche Forschungsgemeinschaft (DFG, German Research Foundation) under Germany's Excellence Strategy - EXC 2075 - 390740016.}

\appendix
\section{Weak compactness of the effective viscous flux}\label{Sec:Appendix EVF}
\input{appendix}

\bibliographystyle{plain}
\bibliography{Lit-Master}
\end{document}

%% file: Introduction.tex
We consider an instance of a diffuse interface model for a homogeneous compressible viscous two-phase fluid given by the compressible Navier--Stokes--Korteweg equations (NSKE) in the isothermal framework in spatial dimension $d \in \{2,3\}$.
In this model, the dynamics of the two-phase fluid occupying some bounded domain $\Omega \subseteq \RR^d$ for some positive time $T>0$ are described by the fluid's density $\rho\colon [0,T] \times \Omega \to\RR_{\geq 0}$ and the fluid's velocity $\uvec\colon[0,T] \times  \Omega \to \RR^d$ that obey
\begin{equation}\label{NSK eqs}
    \left\{
    \begin{aligned}
        &\partial_t \rho + \div(\rho\uvec) = 0 &&\text{in } (0,T)\times\Omega,\\[0.15cm]
        &\partial_t(\rho\uvec) + \div(\rho\uvec\otimes\uvec) + \nablax p(\rho) 
        =
        \div\Svisc(\nablax\uvec) + \kappa \rho \nablax\Deltax \rho 
        &&\text{in } (0,T)\times\Omega,
    \end{aligned}
    \right.
\end{equation}
subject to the boundary conditions
\begin{equation}\label{NSK BC}
    \nablax\rho\cdot \mathbf{n}_{\partial\Omega} = 0, \qquad \uvec = 0 \quad \text{on } [0,T]\times\partial\Omega,
\end{equation}
and the initial conditions 
\begin{equation}\label{NSK IC}
    \rho(0) = \rho_0,\quad (\rho\uvec)(0)=(\rho\uvec)_0\quad \text{in } \Omega.
\end{equation}
Here, $\kappa>0$ denotes the constant capillarity coefficient and $p \colon[0,\infty) \to [0,\infty)$ denotes the pressure function.
Moreover, we have used the notation
\begin{equation*}
    \Svisc(A) 
    := 
    \shearvisc \left( A + A^{\mathrm{T}} - \frac{2}{d} (\mathrm{Tr} A) \,\mathbb{I}\right) + \bulkvisc \, (\mathrm{Tr} A)\, \mathbb{I}
    \qquad \text{for any } A\in \RR^{d\times d},
\end{equation*}
with $\shearvisc >0$ and $\bulkvisc \geq 0$ denoting the constant shear and bulk viscosity coefficient, respectively.
The last term on the right-hand side of the momentum equation $\eqref{NSK eqs}_2$ corresponds to a quadratic contribution of the density's gradient in the corresponding energy functional modeling capillary effects between the two phases.
To distinguish the different phases, a pressure function of Van-der-Waals-type is frequently used, characterized by a function $p\in C^1([0,\infty))$ with $p(0)=0$ that admits some constants $0<a<b<\infty$ such that $p$ is monotonically increasing on $[0,a]\cup[b,\infty)$ and monotonically decreasing on $(a,b)$.
For such a pressure function, we distinguish the phases by calling the fluid's state vapor, spinodal, and liquid, if the density's value lies in the interval $[0,a]$, $(a,b)$, and $[b,\infty)$, respectively (see~Figure~\ref{fig:Plot p VdW}).
For more details concerning the modeling, we refer to \cite{Korteweg1901,AndersonMcFaddenWheeler1998,DunnSerrin1985}.
\par
From both the numerical as well as the analytical point of view, the NSKE in a two-phase setting exhibit two main difficulties that are not present for the compressible Navier--Stokes equations (NSE) in a single-phase setting.
First, one has to treat a third-order differential operator acting on the density in the momentum equation.
In a numerical context, this asks for an appropriate discretization of this higher-order differential operator, while in an analytical context, this term is quite delicate to treat in view of compactness arguments in a weak solution setting (see~\cite{AntonelliSpirito2019,AntonelliSpirito2022}).
Second, one has to treat a non-monotone pressure-density-relation.
From an analytical point of view, the non-monotonicity of the pressure function renders compactness arguments for the density in a weak solution framework delicate (see~\cite{Feireisl2002,BreschJabin2018}). 
From a numerical standpoint, we face the problem that the Jacobian of the underlying first-order flux exhibits complex eigenvalues in the spinodal region.
This prevents the use of standard finite-volume methods that are based on hyperbolicity.
\par
To overcome the numerical difficulties, a relaxation formulation of the NSKE \eqref{NSK eqs} was recently proposed in \cite{Rohde2010} and further developed in \cite{HKMR2020}.
In this relaxation system, the third-order differential operator is approximated by introducing an additional artificial unknown $c \colon [0,T]\times\Omega \to \RR$, called the \emph{relaxation parameter}, that satisfies an additional linear equation.
The precise relaxation system that we call the \emph{relaxed Navier--Stokes--Korteweg equations (rNSKE)} reads
\begin{equation}\label{NSK(alpha,beta) system}
\left\{
\begin{aligned}
    &\partial_t \rho 
    +
    \div(\rho\uvec)
    = 0
    &&\text{in } (0,T)\times\Omega,
    \\[0.15cm]
    &\partial_t(\rho\uvec)
    +
    \div(\rho\uvec\otimes\uvec)
    +
    \nablax p(\rho)
    =
    \div\Svisc(\nablax\uvec)
    +
    \coup \rho\nablax(c-\rho)
    &&\text{in } (0,T)\times\Omega,
    \\[0.15cm]
    &\paracoup\partial_t c 
    -
    \kappa \Deltax c
    +
    \coup(c-\rho) 
    = 
    0
    &&\text{in } (0,T)\times\Omega.
\end{aligned}
\right.
\end{equation}
Here, $\coup>0$ and $\paracoup\geq 0$ are called \emph{coupling coefficients}.
In the case $\paracoup=0$, the equation defining $c$ is elliptic and the rNSKE corresponds to the one proposed in \cite{Rohde2010}, while for the case $\paracoup>0$ the equation defining $c$ is parabolic and the rNSKE corresponds to the one proposed in \cite{HKMR2020}.
\par
The rNSKE are complemented by the boundary conditions
\begin{equation}\label{NSK(alpha,beta) BC}
    \nablax c \cdot \mathbf{n}_{\partial\Omega } = 0,
    \qquad 
    \uvec = 0 \quad \text{on }[0,T]\times\partial\Omega,
\end{equation}
and the initial conditions
\begin{equation}\label{NSK(alpha,beta) IC}
    \rho(0) = \rho_0 , \qquad 
    (\rho\uvec)(0) = (\rho\uvec)_0, \qquad 
    c(0) = c_0 \qquad \text{in } \Omega,
\end{equation}
where the initial condition for $c$ is only needed in the case $\paracoup>0$.
The rNSKE have the advantage that the momentum equation $\eqref{NSK(alpha,beta) system}_2$ can be rewritten as
\begin{equation*}
    \partial_t(\rho\uvec) 
    + 
    \div(\rho \uvec\otimes \uvec) 
    +
    \nablax p_\coup(\rho) 
    =
    \div \Svisc(\nablax \uvec)
    +
    \coup \rho \nablax c
\end{equation*}
with the \emph{artificial pressure function}
\begin{align*}
    \Peff(r) := p(r) + \frac{\coup}{2} r^2 \qquad \text{for any } r \in [0,\infty).
\end{align*}
Note that for pressure functions of Van-der-Waals-type, the artificial pressure function $\Peff$ is non-decreasing, provided $\coup > 0$ is chosen large enough (see~Figure~\ref{fig:Plot p VdW}).
Being in the case where $\Peff$ is non-decreasing, this facilitates the analytical as well as the numerical treatment of the rNSKE.
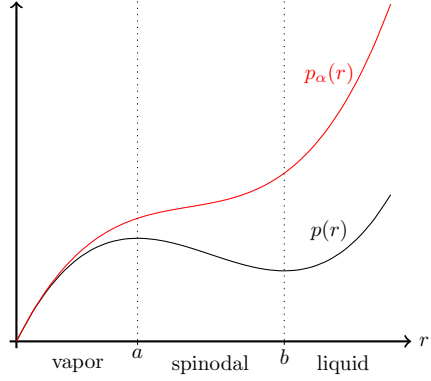
\begin{figure}[ht]\label{fig:Plot p VdW}
    \centering
        \begin{tikzpicture}[scale=0.9, every node/.style={scale=0.8}]
                \draw[->, thick](-0.1,0) -- (5.8,0) node[right] {$r$};
                \draw[->, thick](0,-0.1) -- (0,5);
                \draw(1.7787,0) -- (1.7787,-0.05);
                \draw(3.9348,0) -- (3.9348,-0.05);
                \draw[scale=3.234, domain=0:1.7, smooth, variable=\x, black] plot ({\x}, {
                (\x-0.55)^3 - (\x-0.55)^2 + 0.468875
                });
                \draw[scale=3.234, domain=0:1.7, smooth, variable=\x, red] plot ({\x}, {
                (\x-0.55)^3 - (\x-0.55)^2 + 0.468875 + 0.3*\x^2
                });
                \draw[dotted] (1.7787,5.0) -- (1.7787,0) node[below] {$a$};
                \draw[dotted] (3.9348,5.0) -- (3.9348,0) node[below] {$b$} ;
                \draw (0.9,-0.15) node[below] {vapor};
                \draw (2.85,-0.1) node[below] {spinodal};
                \draw (4.8,-0.1) node[below] {liquid};
                \draw (4.6,1.9) node[below] {\color{black}$p(r)$};
                \draw (4.6,4.2) node[below] {\color{red}$p_\coup(r)$};
        \end{tikzpicture}
        \caption{Illustration of a pressure function of Van-der-Waals-type $p$ and its corresponding artificial pressure function $\Peff$ for some fixed $\coup>0$.}
\end{figure}
\par
By using a formal asymptotic expansion, one can show that in the relaxation limit $\coup \to \infty$, the rNSKE approach the NSKE if $\paracoup=0$, or $\paracoup=\mathcal{O}(\coup^{-1})$ and $\paracoup>0$ (see~\cite{HKMR2020}).
For the case $\paracoup=0$, a rigorous convergence result for the relaxation limit has been obtained in \cite{GiesselmannLattanzioTzavaras2017} working in a smooth and periodic solution setting.
Recently, for the case $\paracoup>0$, a rigorous convergence result for the relaxation limit in the class of finite energy weak solutions for the corresponding initial-boundary-value-problem (IBVP) has been obtained in \cite{ChaudhuriRohdeWendt2025}.
The relaxation limit was also verified by numerical experiments in \cite{NeusserRohdeSchleper2015,HKMR2020}.
Moreover, we refer to the recent result \cite{RohdeWendt2025a} for a rigorous derivation of a two-phase mixture model corresponding to the rNSKE in the case $\paracoup>0$ working in a finite energy weak solution framework.
\par
Without being exhaustive, we give a brief overview on existence results for the rNSKE and related models.
For the compressible Navier--Stokes equations (NSE), the global-in-time existence and uniqueness of classical solutions for small initial data is due to \cite{MatsumuraNishida1983}.
For large initial data, the local-in-time existence and uniqueness of strong solutions was proved in \cite{Valli1983}, while the corresponding global-in-time existence is an open problem.
Concerning the weak solution framework, the global-in-time existence of finite energy weak solutions has been obtained in \cite{Lions1998,FeireislNovotnyPetzeltova2001} for an isentropic pressure law $r\mapsto r^\gamma$ in the regime $\gamma \in \big( \frac{d}{2}, \infty \big)$.
These results are based on compactness arguments that rely sensitively on the monotonicity of the underlying pressure function, thus rendering an extension to general non-monotone pressure functions delicate.
For results tackling this issue, we refer to \cite{Feireisl2002,BreschJabin2018}.
For global-in-time existence results concerning the compressible Navier--Stokes--Fourier (NSF) equations, we refer to the book \cite{FeireislNovotny2017singlim}.
Concerning the NSKE, the local-in-time existence and uniqueness of strong solutions with large initial data has been obtained in \cite{Kotschote2008}, and the local-in-time existence and uniqueness of classical solutions for the corresponding Cauchy problem is due to \cite{HattoriLi1994}.
For the weak solution framework, up to the authors' knowledge, there is no global-in-time existence result available so far.
However, the global-in-time existence of weak solutions to the NSKE in a specific density-dependent viscosity setting has been obtained in \cite{AntonelliSpirito2019,AntonelliSpirito2022,BreschDesjardinsLin2003}.
For the rNSKE in the case $\paracoup = 0$, the local-in-time existence and uniqueness of classical solutions for the corresponding Cauchy problem with large initial data was proved in \cite{Rohde2010}. For a corresponding result and a global-in-time well-posedness result for small initial-data in critical regularity spaces, we refer to \cite{Charve2014} and \cite{Charve2016}, respectively.
In the 1D framework with periodic boundary conditions and large initial data, the global-in-time existence and uniqueness of strong solutions was proved in \cite{RohdeWendt2025b}.
\par
However, there is no rigorous existence result for the IBVP corresponding to the rNSKE with $\paracoup>0$. In particular, the global-in-time existence of finite energy weak solutions in the higher-dimensional case has not been considered so far.
We emphasize that the aforementioned results in \cite{RohdeWendt2025a,ChaudhuriRohdeWendt2025} postulate the global-in-time existence of these solutions.
In the present paper we will fill this gap and prove the global-in-time existence of finite energy weak solutions for the IBVP corresponding to the rNSKE with $\paracoup>0$ in spatial dimension $d\in\{2,3\}$ complementing the results in \cite{RohdeWendt2025a,ChaudhuriRohdeWendt2025}.
\par
Our proof relies on the approximation method from \cite{FeireislNovotnyPetzeltova2001} where the global-in-time existence of finite energy weak solutions to the compressible NSE has been proved.
Compared to the compressible NSE, we have to treat a different energy structure, a possible non-monotonicity of the pressure function, and an additional nonlinear coupling between the momentum equation and the parabolic equation for the relaxation coefficient.
In particular, it is a priori unclear whether the artificial viscosity regularization for the density as well as the delicate compactness arguments from \cite{FeireislNovotnyPetzeltova2001} are compatible with the rNSKE.
It turns out that the error term in the energy inequality caused by the artificial viscosity regularization of the density can be controlled by the corresponding dissipative term while the other issues can be tackled by exploiting parabolic regularity estimates as well as using the technique from \cite{Feireisl2002}.
\\

\paragraph{\bf Organization of the paper.} In Section~\ref{sec:wkSol}, we first define the class of admissible pressure functions and discuss suitable pressure decompositions that we need for the analysis in this work.
We then introduce the concept of finite energy weak solutions, together with a formal derivation motivating why our definition is reasonable.
At the end of \Cref{sec:wkSol}, we state our main result (\Cref{main result:ex wkSol}).
In \Cref{sec:exSolreg}, we prove the global-in-time existence of solutions to a regularized rNSKE system depending on two approximation parameters.
In \Cref{sec:VanArtVisc} and \Cref{sec:VanArtPres}, we perform the corresponding approximation limits for the regularized system resulting into the proof of \Cref{main result:ex wkSol}.
In \Cref{sec:conclusions}, we close this work with some conclusions. 
In the Appendix~\ref{Sec:Appendix EVF}, we state and prove a general version of a weak compactness property to the effective viscous flux which we will apply frequently in the context of the rNSKE throughout this paper (see \Cref{sec:VanArtVisc} and \Cref{sec:VanArtPres}).\\

\paragraph{\bf Notations.} Let $D\subseteq \RR^d$, $d\in \{2,3\}$ be a bounded domain.
For $T>0$, the space-time cylinder corresponding to $D$ will be denoted by $D_T:= (0,T)\times D$.
For $s \in [1,\infty]$, Lebesgue and Sobolev spaces on $D$ will be denoted in the usual way by $L^s(D)$ and $W^{1,s}(D)$. The space of zero-trace Sobolev functions will be denoted by $W_0^{1,s}(D)$, and the space of distributions on $D$ will be denoted by $\mathcal{D}^\prime(D)$.
If no ambiguities arise, we use for vector- or matrix-valued function spaces the same notation, e.g., we write $L^s(D)$ instead of $L^s(D;\RR^d)$.
For $f \in L^1(D)$, we denote its mean value over $D$ by
\begin{equation*}
    \fint_D f := \frac{1}{|D|} \int_D f(x) \, \dd x,
\end{equation*}
where $|D|$ denotes the $d$-dimensional Lebesgue-measure of $D$.
The space $L_0^s(D)$ consists of all functions $f \in L^s(D)$ with $\fint_D f = 0$.
The space of functions on $[0,T]$ ranging into some Lebesgue space $L^s(D)$ continuously with respect to the weak topology on $L^s(D)$ will be denoted by $\Cw([0,T];L^s(D))$.
The conjugate H\"older exponent to $s$ will be denoted as $s^\prime$, i.e.,
\begin{equation*}
    s^\prime:= 
    \begin{cases}
        \frac{s}{s-1} & \text{if } s\in (1,\infty),
        \\
        \infty & \text{if } s=1,
        \\
        1 & \text{if } s=\infty.
    \end{cases}
\end{equation*}
Moreover, for $r \in (1,\infty]$, we denote by $r^-$ any number in $[1,\infty)$ that is smaller than but arbitrary close to $r$, i.e., any number in $[1,r)$.
Moreover, we use the notation
\begin{align}\label{defi Sobolev exponents}
    r^\ast :=
    \begin{cases}
        \frac{dr}{d-r} & \text{if } r< d,
        \\
        \infty & \text{if } r\geq d,
    \end{cases}
    &&
    \overline{r}:=
    \begin{cases}
        \frac{dr}{d-r} & \text{if } r <d,
        \\
        \infty^- & \text{if } r=d,
        \\
        \infty & \text{if } r>d.
    \end{cases}
\end{align}
Finally, for any $r,s \in (1,\infty)$ we denote
\begin{equation*}
\begin{aligned}
    E^{r,s}_0(D) := \overline{C^\infty_c(D;\RR^\dimension)}^{\|\cdot\|_{E^{r,s}(D)}},
    \qquad 
    \|\mathbf{f}\|_{E^{r,s}(D)}:= \|\mathbf{f}\|_{L^r(D)} + \|\div\mathbf{f}\|_{L^s(D)} \quad \forall\, \mathbf{f} \in C^\infty_c(D;\RR^d).
\end{aligned}
\end{equation*}

%% file: WkSol.tex
In this section, we introduce the concept of finite energy weak solutions for the IBVP to the rNSKE \eqref{NSK(alpha,beta) system}--\eqref{NSK(alpha,beta) IC}.
From the theory on compressible NSE, it is known that the regularity class and the global-in-time existence results for finite energy weak solutions depend crucially on the asymptotic growth rate and on the monotonicity properties of the pressure-density-relation $p$ (see e.g.~\cite{FeireislNovotnyPetzeltova2001,Feireisl2002}).
Therefore, we first have to clarify which kind of assumptions we impose on the pressure function.
To do so, we make the following definition that allows us to formulate these concisely.
\begin{definition}
    We call a pressure function $p\colon[0,\infty) \to [0,\infty)$ \emph{admissible with growth rate} $\gamma \in (1,\infty)$ if
    \begin{enumerate}
        \item $p \in C^0([0,\infty))\cap C^1((0,\infty))$, $p(0)=0$, $p\geq 0$ on $[0,\infty)$, and
        \item there exist positive constants $a_p,A_p,b_p\in (0,\infty)$ such that
        \begin{equation}\label{grwth adm p}
            p(r) \leq A_p r^\gamma + b_p \quad \forall\, r \in [0,\infty),\qquad 
            p^\prime(r) \geq a_p r^{\gamma-1} - b_p\quad \forall\, r \in (0,\infty).
        \end{equation}
    \end{enumerate}
\end{definition}
The class of admissible pressure functions accounts for a broad class of monotone and non-monotone pressures. 
To substantiate this statement, we give a few examples:
    \begin{enumerate}
        \item Any isentropic pressure law $r \mapsto r^\gamma$, where $\gamma \in (1,\infty)$, is admissible with growth rate $\gamma$.
        \item Any monotone pressure law $p\in C^0([0,\infty))\cap C^1((0,\infty))$ satisying $p(0)=0$ and the growth condition
        \begin{equation*}
            \lim\limits_{r\to\infty} \frac{p^\prime(r)}{r^{\gamma-1}} = p_\infty
        \end{equation*}
        for some constant $p_\infty>0$ and some $\gamma \in (1,\infty)$ is admissible with growth rate $\gamma$.
        \item Any pressure function of Van-der-Waals-type $p\in C^0([0,\infty))\cap C^1((0,\infty))$ (see~\Cref{sec:Intro}) satisfying the growth condition
        \begin{equation*}
            \lim\limits_{r\to\infty} \frac{p^\prime(r)}{r^{\gamma-1}} = p_{\infty}
        \end{equation*}
        for some constant $p_{\infty}>0$ and some $\gamma \in (1,\infty)$ is admissible with growth rate $\gamma$.
    \end{enumerate}
For an admissible pressure function $p$ with growth rate $\gamma \in (1,\infty)$, we fix a suitable decomposition into a monotone and a non-monotone part.
More precisely, we decompose $p$ via
\begin{equation}\label{dec prs}
    p = h + q
\end{equation}
with
\begin{equation}\label{dec prs reg}
    h\in C^0([0,\infty))\cap C^1((0,\infty)),
    \quad 
    q\in C^\infty_c([0,\infty)),
    \quad 
    h(0) = 0,
    \quad 
    q\leq 0
    \quad \text{on } [0,\infty)
\end{equation}
satisfying
\begin{equation}\label{dec prs grwth h}
    h(r)\leq A_h r^\gamma + b_h \quad \forall\, r \in [0,\infty),
    \quad 
    a_h r^{\gamma-1}\leq h^\prime(r) \quad \forall\, r \in (0,\infty),
\end{equation}
for some constants $a_h, A_h, b_h \in (0,\infty)$.
The existence of such a decomposition follows from the fact that $p$ is admissible with growth rate $\gamma$.
Indeed, we first choose $R_0>0$ large enough such that
\begin{equation*}
    a_p r^{\gamma-1} - 2b_p \geq \frac{a_p}{2}r^{\gamma-1} \qquad \forall \, r \in [R_0,\infty).
\end{equation*}
Then, we take a cut-off function $\psi \in C^\infty_c([0,\infty))$ which satisfies $\psi \geq 0$, $|\psi^\prime|\leq 1$, and
\begin{equation*}
    \psi(r) = r \quad \forall\, r \in [0,R_0],
    \qquad 
    \psi(r) = 0 \quad \forall\, r \in [3R_0,\infty).
\end{equation*}
By setting 
\begin{equation*}
    h(r):=p(r) + b_p\psi(r),
    \qquad 
    q(r) := -b_p\psi(r),
\end{equation*}
we easily check that $p,h$, and $q$ satisfy \eqref{dec prs}--\eqref{dec prs grwth h}.
\par
According to the decomposition \eqref{dec prs}, we introduce corresponding pressure potentials via
\begin{equation}\label{def press pot H Q W}
    H(r):=r\int_1^r \frac{h(z)}{z^2} \, \dd z,
    \quad 
    Q(r) :=r \int_1^r \frac{q(z)}{z^2} \, \dd z,
    \quad W(r):=H(r)+Q(r)
    \qquad \forall r \in [0,\infty),
\end{equation}
for which we readily verify that $H,Q,W \in C^0([0,\infty))\cap C^2((0,\infty))$, and
\begin{align}\label{relations potential-pressure}
        h(r) = H^\prime(r) r -H(r),
        && 
        q(r) = Q^\prime(r) r - Q(r),
        && 
        p(r)=W^\prime(r)r - W(r) 
\end{align}
for any $r \in (0,\infty)$.
Note that $H$ is convex, while $Q$ and $W$ may not be.
By virtue of \eqref{grwth adm p} and \eqref{dec prs grwth h}, we have that
\begin{equation}\label{bounds W,H}
    \begin{aligned}
        &r^\gamma \leq c_1 + c_2 H(r),
        \quad
        r^\gamma \leq c_1 + c_3 W(r),
        \quad 
        |H(r)|+ |W(r)|\leq c_1(1 + r^\gamma),
    \end{aligned}
\end{equation}
for any $r \in [0,\infty)$ and for some constants $c_1,c_2,c_3>0$ that do not depend on $r$.
We will use these decompositions and the corresponding relations frequently in \Cref{sec:exSolreg}.
\par
Let us motivate the regularity classes for a finite energy weak solution to the IBVP \eqref{NSK(alpha,beta) system}--\eqref{NSK(alpha,beta) IC} with $\paracoup>0$ by the following formal calculations.
For $T>0$, we denote by $(\rho,\uvec,c)$ a hypothetical smooth solution of \eqref{NSK(alpha,beta) system}--\eqref{NSK(alpha,beta) IC} with $\rho>0$ that exists on $[0,T]\times \Omega$ and emanates from smooth initial conditions $(\rho_0,(\rho\uvec)_0,c_0)$ with $\rho_0>0$.
Taking the scalar product of the momentum equation $\eqref{NSK(alpha,beta) system}_2$ with $\uvec$, using the continuity equation $\eqref{NSK(alpha,beta) system}_1$, the parabolic equation $\eqref{NSK(alpha,beta) system}_3$, the boundary conditions \eqref{NSK(alpha,beta) BC}, and integration by parts leads to
\begin{equation}\label{energy eq smth}
    \frac{\dd}{\dd t} E[\rho,\rho\uvec,c] + \int_\Omega \Svisc(\nablax \uvec) : \nablax\uvec + \paracoup|\partial_t c|^2 \, \dd x 
    = 0 
    \quad \text{on } (0,T),
\end{equation}
where
\begin{equation}\label{def E(t)}
    E[\rho,\rho\uvec,c]:=
    \int_\Omega 
    \frac{|\rho\uvec|^2}{2\rho}
    +
    W(\rho)
    +
    \frac{\coup}{2}|\rho-c|^2
    +
    \frac{\kappa}{2} |\nablax c|^2 
    \, \dd x
\end{equation}
for any $t \in (0,T)$.
\par
To ease the notation, we introduce the following convention which shall hold throughout the rest of this paper.
By $C(\cdot)$ we denote a generic positive constant that only depends on 
\begin{equation}\label{data - fixed params}
    \coup,\paracoup,\kappa,\shearvisc,\bulkvisc,T,\gamma,\Omega, \max\limits_{r\in [0,\infty)}|q(r)|,
\end{equation}
and on its arguments, but might change its value from line to line.
Moreover, we introduce
\begin{equation}\label{data}
    \data:=\Bigg(\|\rho_0\|_{L^{\tgamma}(\Omega)},\|(\rho\uvec)_0\|_{L^{\frac{2\tgamma}{\tgamma+1}}(\Omega)}, \Bigg\|\frac{|(\rho\uvec)_0|^2}{\rho_0}\Bigg\|_{L^1(\Omega)},\|c_0\|_{W^{1,2}(\Omega)}\Bigg),
\end{equation}
where $\tgamma:=\max\{2,\gamma\}$.
\par
By H\"older's inequality and \eqref{bounds W,H}, we infer that
\begin{equation*}
    E_0:= E[\rho_0,(\rho\uvec)_0,c_0] \leq \Cdata.
\end{equation*}
Integrating \eqref{energy eq smth} in time and using \eqref{bounds W,H} as well as H\"older's and Poincar\'e's inequality leads to
\begin{equation}\label{smth energy bds}
\begin{aligned}
    &\left\| \sqrt{\rho} \uvec \right\|_{L^\infty(0,T;L^2(\Omega))}
    +
    \|\rho\|_{L^\infty(0,T;L^\gamma(\Omega))}
    +
    \|\rho-c\|_{L^\infty(0,T;L^2(\Omega))}
    +
    \|\nablax c\|_{L^\infty(0,T;L^2(\Omega))}
    \\
    &\quad +
    \|\uvec\|_{L^2(0,T;W^{1,2}(\Omega))}
    +
    \|\partial_t c\|_{L^2(0,T;L^2(\Omega))}
    \leq 
    \Cdata.
\end{aligned}
\end{equation}
Integrating the parabolic equation $\eqref{NSK(alpha,beta) system}_3$ in space only, using the conservation of mass and the boundary conditions \eqref{NSK(alpha,beta) BC} yields
\begin{equation*}
    \frac{\dd}{\dd t}\int_\Omega c\, \dd x
    =
    -\frac{\coup}{\paracoup} 
    \bigg(
    \int_\Omega c \, \dd x
    +
    \int_\Omega \rho_0\, \dd x
    \bigg)
    \quad \text{on } (0,T),
\end{equation*}
and thus, Gr\"onwall's inequality implies
\begin{equation*}
    \left\| \int_\Omega c\, \dd x \right\|_{L^\infty((0,T))} \leq \Cdata.
\end{equation*}
In particular, by \eqref{smth energy bds} and the Poincar\'e--Wirtinger inequality,
\begin{equation}\label{smth LinftyL2 bd c}
    \|c\|_{L^\infty(0,T;L^2(\Omega))} \leq \Cdata.
\end{equation}
Combining \eqref{smth energy bds} and \eqref{smth LinftyL2 bd c} leads to
\begin{equation}\label{smth LinftryLtgamma bd rho}
    \|\rho\|_{L^\infty(0,T;L^{\tgamma}(\Omega))} \leq \Cdata.
\end{equation}
Finally, by using standard parabolic regularity estimates (see e.g.~\cite[Theorem~11.29]{FeireislNovotny2017singlim}), we obtain
\begin{equation}\label{smth L2W22 bd c}
    \|c\|_{L^2(0,T;W^{2,2}(\Omega))} 
    \leq 
    \Cdata \left( \|\rho\|_{L^2(0,T;L^2(\Omega))} + \|c_0\|_{W^{1,2}(\Omega)} \right)
    \leq \Cdata.
\end{equation}
With relation \eqref{energy eq smth} and the bounds \eqref{smth energy bds}--\eqref{smth L2W22 bd c} that only depend on $\data$, we anticipate the following definition for finite energy weak solutions to the IBVP \eqref{NSK(alpha,beta) system}--\eqref{NSK(alpha,beta) IC}.
This definition corresponds to the one that is known from the theory on the compressible NSE (see e.g.~\cite{FeireislNovotnyPetzeltova2001}).
\begin{definition}
    Let $T,\coup,\paracoup,\kappa,\shearvisc>0$, $\bulkvisc\geq 0$, $d\in \{2,3\}$, and let $\Omega \subseteq \RR^d$ be a bounded domain. 
    Let $p$ be an admissible pressure function with growth rate $\gamma \in (1,\infty)$ and denote $\tgamma:=\max\{2,\gamma\}$.
    Suppose that initial conditions $\rho_0 \in L^{\tgamma}(\Omega)$, $(\rho\uvec)_0\in L^{\frac{2\tgamma}{\tgamma+1}}(\Omega;\RR^3)$, $c_0 \in W^{1,2}(\Omega)$ with
    \begin{equation}\label{wkSol IC}
        \rho_0 \geq 0 \quad \text{a.e.}, \quad
        (\rho\uvec)_0 = 0\quad \text{on } \{\rho_0=0\},\quad 
        \frac{|(\rho\uvec)_0|^2}{\rho_0} \in L^1(\Omega),
    \end{equation}
    are given.
    Then we call the triplet $(\rho,\uvec,c)$ a \emph{finite energy weak solution} to the IBVP \eqref{NSK(alpha,beta) system}--\eqref{NSK(alpha,beta) IC} on $\OmegaT$ emanating from the initial conditions $(\rho_0,(\rho\uvec)_0,c_0)$ if the following holds:
    \begin{enumerate}
        \item We have the regularity
        \begin{equation}\label{def wsol reg}
            \begin{aligned}
                &\rho \in \Cw([0,T];L^{\tgamma}(\Omega)),\quad 
                \rho\geq 0\,\,\,\text{a.e.},\\
                &\uvec\in L^2(0,T;W^{1,2}_0(\Omega;\RR^d)),\quad
                (\rho\uvec)\in \Cw([0,T];L^{\frac{2\tgamma}{\tgamma+1}}(\Omega;\RR^d)),\\
                &c\in C([0,T];W^{1,2}(\Omega))\cap L^2(0,T;W^{2,2}(\Omega)),\quad 
                \partial_t c \in L^2(\OmegaT);
            \end{aligned}
        \end{equation}
        \item the equations \eqref{NSK(alpha,beta) system} are satisfied in the weak sense, that is,
        \begin{align}
                &\int_0^T \int_\Omega 
                \rho \partial_t \varphi 
                +
                \rho\uvec \cdot \nablax \varphi 
                \, \dd x \, \dd t 
                = 0 \quad \forall \, \varphi \in C^\infty_c(\RR^d_T),\label{def wsol cont}
                \\
                &\int_0^T\int_\Omega 
                \rho\uvec \cdot \partial_t\varphivec 
                +
                \rho\uvec\otimes\uvec:\nablax\varphivec
                +
                \Peff(\rho)\div\varphivec
                \, \dd x \, \dd t\nonumber
                \\
                &\qquad=
                \int_0^T \int_\Omega 
                \Svisc(\nablax\uvec):\nablax\varphivec -\coup \rho \nablax c \cdot \varphivec
                \, \dd x \, \dd t 
                \qquad \forall \, \bm{\varphi}\in C^\infty_c(\OmegaT;\RR^d),\label{def wsol mom}
                \\
                &\int_0^T \int_\Omega 
                \paracoup c\partial_t \varphi 
                -
                \kappa\nablax c\cdot \nablax \varphi 
                -
                \coup(c-\rho) \varphi
                \, \dd x \, \dd t
                = 0  \qquad\forall \, \varphi \in C^\infty_c(\RR_T^d); \label{def wsol parab}
        \end{align}
        \item the initial conditions \eqref{NSK(alpha,beta) IC} hold in the sense that
        \begin{equation}\label{def wsol IC}
            \rho(0)=\rho_0,\quad 
            (\rho\uvec)(0)=(\rho\uvec)_0,\quad 
            c(0)=c_0 \quad \text{a.e.~in } \Omega;
        \end{equation}
        \item the energy inequality corresponding to \eqref{energy eq smth} holds in the weak sense, that is,
        \begin{equation}\label{def wsol Eineq}
            \begin{aligned}
                -\int_0^T E[\rho,\rho\uvec,c] \partial_t \psi \, \dd \tau
                +
                \int_0^T \psi \int_\Omega 
                \Svisc(\nablax\uvec) : \nablax\uvec + \paracoup|\partial_t c|^2
                \, \dd x \, \dd \tau 
                \leq 
                E_0 \, \psi(0)
            \end{aligned}
        \end{equation}
        for any test function $\psi \in C^\infty_c([0,T))$ with $\psi \geq 0$, where $E_0:=E[\rho_0,(\rho\uvec)_0,c_0]$ with $E[\cdot,\cdot,\cdot]$ being defined in \eqref{def E(t)}.
    \end{enumerate}
\end{definition}
\begin{remark}
        The regularity of $c$ in \eqref{def wsol reg} and \eqref{def wsol parab} imply that we have in fact
        \begin{equation*}
            \paracoup\partial_t c - \kappa \Deltax c + \coup(c-\rho) = 0 \quad \text{a.e.~in } \OmegaT
        \end{equation*}
        and
        \begin{equation*}
            \nablax c(\tau)\cdot \mathbf{n}_{\mid \partial\Omega} = 0
            \quad \text{in the sense of traces for a.e.~} \tau \in (0,T).
        \end{equation*}
\end{remark}
With the definition of finite energy weak solutions corresponding to the IBVP of the rNSKE at hand we can now state our main result which gives a positive answer concerning the global-in-time existence of such solutions.
The precise result reads as follows:
\begin{theorem}\label{main result:ex wkSol}
    Let $\coup,\paracoup,\kappa,T,\shearvisc>0$, $\bulkvisc \geq 0$, $d \in\{2,3\}$, and let $\Omega \subseteq \RR^d$ be a bounded domain with $\partial \Omega \in C^{2,\nu}$ for some $\nu \in (0,1]$.
    Let $p$ be an admissible pressure function with growth rate $\gamma \in (1,\infty)$ and denote $\tgamma:=\max\left\{ 2,\gamma \right\}$.
    Suppose that initial conditions $\rho_0\in L^{\tgamma}(\Omega)$, $(\rho\uvec)_0 \in L^{\frac{2\tgamma}{\tgamma+1}}(\Omega;\RR^d)$, $c_0 \in W^{1,2}(\Omega)$ satisfying \eqref{wkSol IC} are given.
    Then there exists a finite energy weak solution $(\rho,\uvec,c)$ to \eqref{NSK(alpha,beta) system}--\eqref{NSK(alpha,beta) IC} on $\OmegaT$ emanating from the initial conditions $(\rho_0,(\rho\uvec)_0,c_0)$.
\end{theorem}
It is remarkable that the global-in-time existence result \Cref{main result:ex wkSol} applies for the whole regime $\gamma \in (1,\infty)$, which is strictly larger than the regime $\gamma \in \big(\frac{d}{2},\infty\big)$ for which the global-in-time existence of finite energy weak solutions is known for the compressible NSE (see~\Cref{sec:Intro}).
This is attributable to the additional quadratic contributions in the energy functional corresponding to the rNSKE providing us with a suitable \emph{a priori} estimate for the density only in terms of suitable norms of the initial data (see~\eqref{def E(t)}).
We further emphasize that \Cref{main result:ex wkSol} does not impose any restrictions on the relaxation parameters $\coup,\paracoup\in (0,\infty)$. 
In particular, the artificial pressure function $\Peff$ is not required to be non-decreasing, although this can be guaranteed for pressure functions of Van-der-Waals type if $\coup\in(0,\infty)$ is chosen large enough.

%% file: exSol.tex
Our proof for \Cref{main result:ex wkSol} relies on the approximation method from \cite{FeireislNovotnyPetzeltova2001}.
More specifically, for two positive approximation parameters $\artprs,\artvisc>0$ and some constant $\artgrwth>0$ that is chosen suitably large (see Proposition~\ref{prop:Gal approx}), we first prove a suitable existence result for the regularized rNSKE
\begin{equation}\label{rNSK(delta,eps)}
\left\{
    \begin{aligned}
        &\partial_t \rho + \div(\rho\uvec) 
        = 
        \artvisc\Deltax \rho &&\text{in } \OmegaT,\\[0.15cm]
        &\partial_t(\rho\uvec) + \div(\rho\uvec\otimes\uvec) + \nablax p_\artprs(\rho) 
        +
        \artvisc \nablax\uvec\cdot \nablax\rho
        =
        \div\Svisc(\nablax\uvec) + \coup \rho\nablax(c-\rho)
        &&\text{in } \OmegaT,\\[0.15cm]
        &\paracoup\partial_tc - \kappa \Deltax c + \coup(c-\rho) 
        =
        0 &&\text{in } \OmegaT,\\
    \end{aligned}
\right.
\end{equation}
where $p_\artprs(r):= p(r) + \artprs r^\artgrwth$, subject to the boundary conditions
\begin{equation}\label{rNSK(delta,eps) BC}
        \nablax c\cdot \mathbf{n}_{\partial \Omega}=\nablax \rho \cdot \mathbf{n}_{\partial \Omega}=0,
        \quad \uvec
        = 0\quad \text{on } [0,T]\times\partial\Omega,
\end{equation}
and the initial conditions
\begin{equation}\label{rNSK(delta,eps) IC}
    \rho(0)=\rho_0,\quad (\rho\uvec)(0)=(\rho\uvec)_0,\quad c(0) = c_0 \quad \text{in } \Omega.
\end{equation}
We will use this existence result to approximate a finite energy weak solution by performing two consecutive limits:
First, we perform the vanishing artificial viscosity limit $\artvisc\to 0$ and obtain a finite energy weak solution to the rNSKE with the regularized pressure function $p_\artprs$ (see~\Cref{sec:VanArtVisc});
second, we perform the artificial pressure limit $\artprs \to 0$ and obtain a finite energy weak solution to the rNSKE with the original pressure function (see~\Cref{sec:VanArtPres}).
\par
The regularized rNSKE \eqref{rNSK(delta,eps)}--\eqref{rNSK(delta,eps) IC} also satisfy a suitable energy inequality.
However, due to the artificial viscosity term and the possible non-convexity of the pressure potential $W$, we will use the pressure decomposition in \eqref{dec prs} splitting the pressure potential $W$ into a convex and a non-convex part (see~\eqref{def press pot H Q W}).
For some positive time $T>0$, we fix a hypothetical smooth solution $(\rho,\uvec,c)$ of \eqref{rNSK(delta,eps)}--\eqref{rNSK(delta,eps) IC} with $\rho>0$ living on $\OmegaT$ emanating from the smooth initial conditions $(\rho_0,(\rho\uvec)_0,c_0)$ with $\rho_0>0$.
Then, similarly as in Section~\ref{sec:wkSol}, we deduce the energy identity
\begin{equation}\label{E_m energy eq smth}
\begin{aligned}
    &\frac{\dd}{\dd t} E_{\mathrm{m},\artprs}[\rho,\rho\uvec,c]
    +
    \int_\Omega 
    \Svisc(\nablax\uvec):\nablax\uvec 
    +
    \paracoup |\partial_t c|^2 
    +
    \artvisc \coup|\nablax \rho|^2 
    +
    \artgrwth \artvisc \artprs \rho^{\artgrwth-2}|\nablax\rho|^2
    \, \dd x
    \\
    &\quad=
    \int_\Omega 
    q(\rho) \div\uvec
    +
    \artvisc\coup \nablax \rho \cdot \nablax c
    \, \dd x \quad \text{on } (0,T),
\end{aligned}
\end{equation}
with
\begin{equation}\label{def E_m(t)}
    E_{\mathrm{m},\artprs}[\rho,\rho\uvec,c] 
    :=
    \int_\Omega
    \frac{|\rho\uvec|^2}{2\rho} + H(\rho) + \frac{\artprs}{\artgrwth-1}\rho^\artgrwth + \frac{\coup}{2}|\rho-c|^2 + \frac{\kappa}{2} |\nablax c|^2 
    \, \dd x.
\end{equation}
\par
In the sequel, we will only prove \Cref{main result:ex wkSol} in the case $d=3$.
The proof for the case $d=2$ follows the same lines with minor modifications that do not produce any further difficulties.
\par
In this subsection we construct global-in-time solutions to the regularized system \eqref{rNSK(delta,eps)}--\eqref{rNSK(delta,eps) IC}.
The precise statement that we prove in this section reads as follows:
\begin{proposition}\label{prop:wkSol rNSK(delta,eps)}
    Let $\coup,\paracoup,T,\kappa,\shearvisc>0$, $\bulkvisc\geq0$, $\artprs,\artvisc\in(0,1)$, and let $\Omega \subseteq \RR^3$ be a bounded domain with $\partial\Omega \in C^{2,\nu}$ for some $\nu \in (0,1]$.
    Let $p$ be an admissible pressure function with growth rate $\gamma \in  (1,\infty)$ satisfying the decomposition \eqref{dec prs}--\eqref{dec prs grwth h}, let $\artgrwth >\max\left\{15,\tgamma\right\}$, and denote
    \begin{equation}\label{def p_coup,delta}
        p_{\coup,\artprs}(r) := p_\coup(r) + \artprs r^\artgrwth \quad \forall\, r\in [0,\infty).
    \end{equation}
    Let $0<\underline{\rho} < \overline{\rho}<\infty$ denote two constants, and suppose that we are given initial conditions $\rho_0 \in W^{1,\infty}(\Omega)$, $(\rho\uvec)_0 \in L^2(\Omega;\RR^3)$, $c_0 \in W^{1,2}(\Omega)$ satisfying  $\underline{\rho}\leq \rho_0\leq \overline{\rho}$.
    Let $D_0>0$ be a positive constant such that
    \begin{equation}\label{init energy bd}
        \|\rho_0\|_{L^{\tgamma}(\Omega)} 
        +
        \bigg\|\frac{|(\rho\uvec)_0|^2}{\rho_0}\bigg\|_{L^1(\Omega)}
        +
        \|c_0\|_{W^{1,2}(\Omega)}
        +
        \artprs^\frac{1}{\artgrwth}\|\rho_0\|_{L^{\artgrwth}(\Omega)}     
        \leq \constD.
    \end{equation}
    Then there exists a triplet $(\rho,\uvec,c)$ such that
    \begin{enumerate}
        \item we have the regularity
        \begin{equation}\label{reg wksol rNSK(delta,eps)}
            \begin{aligned}
                &\rho\in \Cw([0,T];L^\artgrwth(\Omega))\cap L^{\frac{5}{3}\artgrwth}\left(\Omega_T\right),
                \quad 
                \rho\geq 0 \,\,\,\text{a.e.},
                \\
                &
                \rho^{\frac{\artgrwth}{2}}\in L^2(0,T;W^{1,2}(\Omega)),
                \quad 
                \partial_t \rho \in L^{s(\artgrwth)}(\Omega_T),
                \quad
                \nablax^2 \rho \in L^{s(\artgrwth)}(\Omega_T),
                \\
                & 
                \uvec\in L^2(0,T;W^{1,2}_0(\Omega;\RR^3)),
                \quad 
                c\in L^2(0,T;W^{2,2}(\Omega))\cap C([0,T];W^{1,2}(\Omega)),
                \\ 
                &\rho\uvec \in \Cw([0,T];L^{\frac{2\artgrwth}{\artgrwth+1}}(\Omega;\RR^3)),
                \quad
                \rho\uvec,\, \nablax \rho  \in L^{s(\artgrwth)}(0,T;E^{r(\artgrwth),s(\artgrwth)}_0(\Omega)),
                \\
                &
                \partial_t c \in L^2(\OmegaT),
            \end{aligned}
        \end{equation}
        where
        \begin{equation}\label{values r,s}
            3 < r(\artgrwth):= \frac{10\artgrwth-6}{3\artgrwth+3} < \frac{10}{3},
            \quad 
            \frac{6}{5} < s(\artgrwth):= \frac{5\artgrwth-3}{4\artgrwth} < \frac{5}{4};
        \end{equation}
        \item the regularized rNSKE \eqref{rNSK(delta,eps)} are satisfied in the weak sense, that is,
        \begin{align}
            &\int_0^T \int_\Omega 
            \rho 
            \partial_t \varphi
            +
            \rho\uvec \cdot \nablax \varphi
            -
            \artvisc\nablax\rho\cdot\nablax \varphi
            \, \dd x \, \dd t
            =
            0
            \qquad \forall\, \varphi \in C^\infty_c(\RR_T^3),\label{wk sol cont(eps,delt)}
            \\
            &\int_0^T\int_\Omega 
            \rho\uvec \cdot \partial_t\varphivec
            +
            \rho\uvec\otimes\uvec:\nablax\varphivec
            +
            p_{\coup,\artprs}(\rho) \div\varphivec
            -
            \artvisc\nablax\rho\cdot \nablax\uvec\cdot \varphivec
            \, \dd x \, \dd t\nonumber
            \\
            &\qquad=
            \int_0^T \int_\Omega 
            \Svisc(\nablax\uvec):\nablax\varphivec -\coup \rho \nablax c \cdot \varphivec
            \, \dd x \, \dd t 
            \qquad \forall \, \varphivec\in C^\infty_c(\OmegaT;\RR^3),\label{wk sol mom(eps,delt)}
            \\
            &\int_0^T \int_\Omega 
            \paracoup c\partial_t \varphi 
            -
            \kappa \nablax c\cdot \nablax \varphi 
            -
            \coup(c-\rho) \varphi
            \, \dd x \, \dd t
            = 0  \qquad\forall \, \varphi \in C^\infty_c(\RR_T^3); \label{weak sol parab(eps,delt)}
        \end{align}
        \item the initial conditions \eqref{rNSK(delta,eps) IC} hold in the sense that
        \begin{equation}\label{wkSol rNSK(eps,delt) IC}
            \rho(0)=\rho_0,
            \quad 
            (\rho\uvec)(0)=(\rho\uvec)_0,
            \quad 
            c(0)=c_0
            \quad \text{a.e.~in } \Omega;
        \end{equation}
        \item the energy inequality corresponding to \eqref{E_m energy eq smth} holds in the weak sense, that is,
        \begin{equation}\label{wkSol rNSK(eps,delt) energy}
            \begin{aligned}
                &-\int_0^T E_{\mathrm{m},\artprs}[\rho,\rho\uvec,c]\partial_t \psi 
                \, \dd \tau
                +
                \int_0^T \psi \int_\Omega 
                \Svisc(\nablax\uvec):\nablax\uvec
                +
                \paracoup|\partial_t c|^2
                \, \dd x \, \dd t\\
                &\quad +
                \int_0^T \psi \int_\Omega 
                \artvisc\coup |\nablax \rho|^2
                +
                \artvisc\artprs \artgrwth \rho^{\artgrwth-2} |\nablax\rho|^2
                \, \dd x \, \dd t
                \\
                & \leq 
                E_{\mathrm{m},\artprs}[\rho_0,(\rho\uvec)_0,c_0] \psi(0)
                +
                \int_0^T \psi \int_\Omega 
                q(\rho) \div\uvec + \artvisc\coup \nablax\rho \cdot \nablax c 
                \, \dd x \, \dd t,
            \end{aligned}
        \end{equation}
        for any $\psi \in C^\infty_c([0,T))$ with $\psi \geq 0$, where $E_{\mathrm{m},\artprs}[\cdot,\cdot,\cdot]$ is defined in \eqref{def E_m(t)};
        \item we have the bounds
        \begin{equation}\label{wkSol rNSK(eps,delta) unif-bds-I}
            \begin{aligned}
                &\left\|\sqrt{\rho}\uvec\right\|_{L^\infty(0,T;L^2(\Omega))}
                +
                \left\|\rho\right\|_{L^\infty(0,T;L^\tgamma(\Omega))}
                +
                \|c\|_{L^\infty(0,T;W^{1,2}(\Omega))}
                +
                \|\uvec\|_{L^2(0,T;W^{1,2}(\Omega))}
                \\
                &\quad 
                +
                \left\|\partial_t c\right\|_{L^2(\OmegaT)}
                +
                \|c\|_{L^2(0,T;W^{2,2}(\Omega))}
                +
                \artprs^{\frac{1}{\artgrwth}}\|\rho\|_{L^\infty(0,T;L^{\artgrwth}(\Omega))}
                \\
                &\quad
                +
                \sqrt{\artvisc}\|\nablax \rho\|_{L^2(\OmegaT)}
                \leq 
                C(D_0),
            \end{aligned}
        \end{equation}
        and
        \begin{equation}\label{wkSol rNSK(eps,delta) unif-bds-II}
            \begin{aligned}
                &\|\rho\uvec\|_{L^{r(\artgrwth)}(\OmegaT)}
                +
                \artvisc\|\nablax\rho\|_{L^{r(\artgrwth)}(\OmegaT)}
                +
                \artvisc\|\div(\rho\uvec)\|_{L^{s(\artgrwth)}(\OmegaT)}
                +
                \artvisc\|\nablax\rho \cdot \nablax \uvec\|_{L^{s(\artgrwth)}(\OmegaT)}
                \\
                &
                \leq 
                C(\constD,\artprs).
            \end{aligned}
        \end{equation}
    \end{enumerate}
    \end{proposition}
    \subsection{Galerkin approximation}\label{subsec:Galerkin Approx}
    To prove Proposition~\ref{prop:wkSol rNSK(delta,eps)}, we extend the method in \cite[Chapter~2]{FeireislNovotnyPetzeltova2001} to the rNSKE.
    More specifically, we use a Galerkin approximation for the momentum equation $\eqref{NSK(alpha,beta) system}_2$ while solving the continuity equation $\eqref{NSK(alpha,beta) system}_1$ and the parabolic equation $\eqref{NSK(alpha,beta) system}_3$ directly.
    To formulate this approximation, we have to fix some notations.
    We introduce the function spaces
    \begin{equation}\label{spaces Gal approx}
    \begin{aligned}
    &\mathcal{Q}_T:=\left\{d\mid d \in L^2(0,T;W^{2,2}(\Omega))\cap C([0,T];W^{1,2}(\Omega)),\, \partial_t d \in L^2(\OmegaT)\right\},
    \\
    &\mathcal{R}_T:=
                \left\{
                r\mid r \in L^2(0,T;W^{2,\infty^-}(\Omega))\cap C([0,T];W^{1,\infty^-}(\Omega)),\,
                \partial_t r \in L^2(0,T;L^{\infty^-}(\Omega))
                \right\}.
    \end{aligned}
    \end{equation}
    Furthermore, by standard theory on elliptic PDEs (see e.g.~\cite[Section~4.7]{NovotnyStraskraba2004}), there exist countable sets
    $\{\lambda_i\}_{i=1}^\infty$, $\{\psivec_i\}_{i=1}^{\infty}$ satisfying
    \begin{equation*}
    \begin{aligned}
        &\psivec_i \in W^{1,\infty^-}_0(\Omega;\RR^3)\cap W^{2,\infty^-}(\Omega;\RR^3),
        \quad
        -\div\mathbb{S}(\nablax \psivec_i)=\lambda_i \psivec_i \quad \text{a.e.~in } \Omega,
    \end{aligned}
    \end{equation*}
    with $\{\psivec_i\}_{i=1}^\infty$ being an orthonormal basis of $L^2(\Omega;\RR^3)$ with respect to the standard scalar product on $L^2(\Omega;\RR^3)$, and an orthogonal basis of $W^{1,2}_0(\Omega;\RR^3)$ with respect to the scalar product
    \begin{equation*}
        \left\langle \mathbf{v},\mathbf{w} \right\rangle_{W^{1,2}_0} := \int_\Omega \mathbb{S}(\nablax\mathbf{v}):\nablax \mathbf{w} \, \dd x \quad \forall\, \mathbf{v},\mathbf{w} \in W^{1,2}_0(\Omega;\RR^3).
    \end{equation*}
    Accordingly, we introduce for $n \in \NN$ the finite dimensional vector spaces
    \begin{equation*}
        X_n := \mathrm{span}_\RR\{\psivec_i\}_{i=1}^n,
    \end{equation*}
    equipped with the standard $L^2(\Omega;\RR^3)$-scalar product.
    Furthermore, we denote by 
    \begin{equation*}
        \proj\colon L^2(\Omega;\RR^3) \to X_n
    \end{equation*}
    the orthogonal projection of $L^2(\Omega;\RR^3)$ onto $X_n$.
    Since $X_n$ is a finite dimensional space, we have that
    \begin{equation}\label{equiv norms Xn}
        \|\cdot\|_{W^{k,p}(\Omega)} \text{ and } \|\cdot\|_{X_n} 
        \text{ are equivalent on } X_n \text{ for any } k\in \NN\cup\{0\}
        \text{ and any } p \in [1,\infty].
    \end{equation}
    Moreover, we have that the projection $\proj$ satisfies
    \begin{equation}\label{prpties proj}
    \begin{aligned}
        &\int_\Omega \proj(\vvec) \cdot \mathbf{w}\, \dd x = \int_\Omega \vvec \cdot \proj(\mathbf{w})\, \dd x ,
        \quad \forall \vvec,\mathbf{w} \in L^2(\Omega),
        \qquad 
        \|\proj\|_{\mathcal{L}(L^2(\Omega);L^2(\Omega))}\leq 1,
        \\
        &\proj \vvec\to \vvec \quad \text{in } L^2(\Omega)\quad \forall\, \vvec\in L^2(\Omega),
        \qquad 
        \proj \mathbf{w}\to \mathbf{w} \quad \text{in } W^{1,2}(\Omega) \quad \forall\, \mathbf{w}\in W_0^{1,2}(\Omega),
        \\
        &\sup\limits_{\mathbf{v}\in W^{1,s}(\Omega),\, \vvec\neq 0} 
        \left(\frac{\|\proj\vvec - \vvec\|_{L^2(\Omega)}}{\|\vvec\|_{W^{1,s}(\Omega)}}\right) \xrightarrow{n \to \infty} 0
        \quad \forall s \in \left(\frac{6}{5},\infty\right),
        \\
        &\|\proj\vvec\|_{W^{k,2}(\Omega)}\leq C \|\vvec\|_{W^{k,2}(\Omega)} \quad \forall \, \vvec\in W^{1,2}_0(\Omega)\cap W^{k,2}(\Omega),\quad k\in \{1,2\}.
    \end{aligned}
    \end{equation}
    We refer to \cite{NovotnyStraskraba2004} for the corresponding proofs.
    \par
    In what follows, the parameter $n\in\NN$ plays the role of an approximation parameter.
    We aim to construct a sequence of solutions $\{(\rho_n,\uvec_n,c_n)\}_{n\in\NN}$ that approximates a finite energy weak solution to the regularized rNSKE \eqref{rNSK(delta,eps)}--\eqref{rNSK(delta,eps) IC}.
    The precise statement that guarantees the global-in-time existence of such an approximate sequence of solutions reads as follows:
    \begin{proposition}\label{prop:Gal approx}
        Let $n\in \NN$ and denote
        \begin{equation*}
             \uvec_{0n} := \proj\left( \frac{(\rho\uvec)_0}{\rho_0} \right) \in X_n.
        \end{equation*}
        Under the hypotheses and notations of Proposition~\ref{prop:wkSol rNSK(delta,eps)}, there exists a unique triplet $(\rho_n,\uvec_n,c_n)\in \mathcal{R}_T\times C([0,T];X_n)\times \mathcal{Q}_T$ with $\partial_t \uvec_n \in L^2(0,T;X_n)$ such that
        \begin{enumerate}
        \item we have
        \begin{equation}\label{Gal approx cont+parab}
            \begin{aligned}
                &\partial_t \rho_n + \div(\rho_n\uvec_n) = \artvisc\Deltax \rho_n &&\text{a.e.~in } \OmegaT,\\
                &\paracoup\partial_t c_n - \kappa \Deltax c_n + \coup (c_n-\rho_n) = 0 &&\text{a.e.~in } \OmegaT,
            \end{aligned}
        \end{equation}
        \begin{equation}\label{Gal approx cont+parab BC}
            \begin{aligned}
                \nablax \rho_n(\tau) \cdot \mathbf{n}_{\partial\Omega } = \nablax c_n(\tau) \cdot \mathbf{n}_{\partial\Omega} = 0 \quad \text{in the sense of traces for a.e.~} \tau \in (0,T),
            \end{aligned}
        \end{equation}
        and
        \begin{equation}\label{Gal approx cont+parab IC}
            \rho_n(0) = \rho_0,
            \quad 
            c_n(0) = c_0 
            \qquad \text{a.e.~in } \Omega;
        \end{equation}
        \item we have a.e.~in $(0,T)$ that
        \begin{equation}\label{Gal approx mom}
            \begin{aligned}
                &\int_\Omega 
                \Bigl(
                    \partial_t(\rho_n\uvec_n) + \div(\rho_n\uvec_n\otimes\uvec_n) + \nablax p_{\coup,\artprs}(\rho_n) + \artvisc\nablax \rho_n\cdot\nablax\uvec_n
                \Bigr)
                \cdot \psivec\, \dd x
                \\
                &\quad -
                \int_\Omega \Bigl( \div\mathbb{S}(\nablax\uvec_n) +\coup\rho_n\nablax c_n \Bigr) \cdot \psivec\, \dd x = 0
                \qquad \forall \, \psivec\in X_n,
            \end{aligned}
        \end{equation}
        and
        \begin{equation}\label{Gal approx mom IC}
            \uvec_n(0)=\uvec_{0n}\quad \text{in } \Omega;
        \end{equation}
        \item we have 
        \begin{equation}\label{Gal approx Energy}
            \begin{aligned}
                &\frac{\dd}{\dd t} E_{\mathrm{m},\artprs}[\rho_n,\rho_n\uvec_n,c_n]
                \\
                &\quad+
                \int_\Omega 
                \Svisc(\nablax\uvec_n):\nablax\uvec_n+\paracoup|\partial_t c_n|^2 + \artvisc\coup|\nablax \rho_n|^2 + \artvisc\artprs\artgrwth \rho^{\artgrwth-2} |\nablax \rho_n|^2
                \, \dd x 
                \\
                &\leq 
                \int_\Omega 
                q(\rho_n)\div\uvec_n + \artvisc\coup\nablax\rho_n\cdot\nablax c_n
                \, \dd x
                \qquad \text{a.e.~in } (0,T),
            \end{aligned}
        \end{equation}
        where $E_{\mathrm{m},\artprs}[\cdot,\cdot,\cdot]$ is defined in \eqref{def E_m(t)}.
        \end{enumerate}
    \end{proposition}
    The proof of \Cref{prop:Gal approx} is based on a fixed point argument to obtain a local-in-time solution followed by a continuation argument based on suitable energy estimates in order to show that the local-in-time solution extends to a global-in-time one.
    For the fixed point argument, we need the following auxiliary result.
    \begin{lemma}\label{lem:sol operators}
        Let $T>0$ and let $\Omega \subseteq \RR^3$ be a bounded domain with $\partial\Omega \in C^{2,\nu}$ for some $\nu \in (0,1]$. Denote by $0<\underline{\rho}<\overline{\rho}<\infty$ two constants and suppose that initial conditions $\rho_0\in W^{1,\infty^-}(\Omega)$ and $c_0 \in W^{1,2}(\Omega)$ with $\underline{\rho}\leq \rho_0\leq \overline{\rho}$ are given.
        Then there exist two maps 
        \begin{align*}
            \mathcal{S}_{\rho_0}\colon L^\infty(0,T;W^{1,\infty}(\Omega)\cap W^{1,2}_0(\Omega))\to \mathcal{R}_T,
            &&
            \mathcal{T}_{c_0}\colon L^2(\OmegaT)\to \mathcal{Q}_T
        \end{align*}
        with the following properties:
        \begin{enumerate}
            \item [(i)]
            For any $\mathbf{v}\in L^\infty(0,T;W^{1,\infty}(\Omega)\cap W^{1,2}_0(\Omega))$, we have
            \begin{equation*}
            \begin{aligned}
                &\partial_t \mathcal{S}_{\rho_0}[\vvec] + \div\left(\mathcal{S}_{\rho_0}[\vvec]\vvec\right)
                =
                \artvisc \Deltax \mathcal{S}_{\rho_0}[\vvec] \quad \text{a.e.~in } \OmegaT,
                \\
                &\nablax \mathcal{S}_{\rho_0}[\vvec](\tau) \cdot \mathbf{n}_{\partial \Omega} = 0\quad \text{in the sense of traces for a.e.~} \tau \in (0,T),
                \\
                &\mathcal{S}_{\rho_0}[\vvec](0)=\rho_0 \quad \text{a.e.~in } \Omega.
            \end{aligned}
            \end{equation*}
            Moreover, we have for any $\tau \in [0,T]$ that
            \begin{equation}\label{S_rho0 infty bds}
                \underline{\rho}\exp\left(-\int_0^\tau \|\vvec\|_{W^{1,\infty}}\, \dt \right) 
                \leq 
                \mathcal{S}_{\rho_0}[\vvec](\tau) 
                \leq 
                \overline{\rho}\exp\left(\int_0^\tau \|\vvec\|_{W^{1,\infty}(\Omega)}\, \dt \right)
                \quad \text{a.e.~in } \Omega.
            \end{equation}
            Furthermore, we have for any $\vvec,\mathbf{w}\in L^\infty(0,T;W^{1,\infty}(\Omega)\cap W^{1,2}_0(\Omega))$ satisfying
            \begin{equation}\label{bound by K}
            \max\{\|\vvec\|_{L^\infty(0,T;W^{1,\infty}(\Omega))}, \|\mathbf{w}\|_{L^\infty(0,T;W^{1,\infty}(\Omega))}\}\leq K    
            \end{equation}
            for some constant $K>0$ that the estimates
            \begin{equation*}
                \left\|\mathcal{S}_{\rho_0}[\vvec]\right\|_{L^\infty(0,\tau;W^{1,2}(\Omega))}
                \leq 
                C\|\rho_0\|_{W^{1,2}(\Omega)} \exp\left(\frac{C}{2\artvisc}(K+K^2)\tau\right)
            \end{equation*}
            and
            \begin{equation}\label{S_rho0 diff est}
                \left\|\mathcal{S}_{\rho_0}[\vvec](\tau)-\mathcal{S}_{\rho_0}[\mathbf{w}](\tau)\right\|_{L^2(\Omega)}
                \leq 
                C(\artvisc,K) \tau \|\rho_0\|_{W^{1,2}(\Omega)}\|\vvec-\mathbf{w}\|_{L^\infty(0,\tau;W^{1,\infty}(\Omega))}
            \end{equation}
            hold for any $\tau \in [0,T]$.
            \item [(ii)]
            For any $r \in L^2(\OmegaT)$, we have
            \begin{equation}\label{T_c0 operator equations}
            \begin{aligned}
                &\paracoup\partial_t \mathcal{T}_{c_0}[r] - \kappa \Deltax \mathcal{T}_{c_0}[r] + \coup \mathcal{T}_{c_0}[r]
                =
                \coup r \quad \text{a.e.~in } \OmegaT,
                \\
                &\nablax\mathcal{T}_{c_0}[r](\tau) \cdot \mathbf{n}_{\partial \Omega} = 0 \quad \text{in the sense of traces for a.e.~} \tau \in (0,T),
                \\
                &\mathcal{T}_{c_0}[r](0)=c_0 \quad \text{a.e.~in } \Omega.
            \end{aligned}
            \end{equation}
            Moreover, we have for any $\tau \in [0,T]$ that
            \begin{equation*}
                \bigl\|\nablax(\mathcal{T}_{c_0}\circ \mathcal{S}_{\rho_0}[\vvec])(\tau)\bigr\|_{L^2(\Omega)}
                    \leq 
                    C
                    \left(
                    \|c_0\|_{W^{1,2}(\Omega)}
                    +
                    \overline{\rho}\exp \left( \int_0^\tau \|\vvec\|_{W^{1,\infty}(\Omega)}\, \dd t \right)
                    \right).
            \end{equation*}
            Furthermore, we have for any $\vvec, \mathbf{w}\in L^\infty(0,T;W^{1,\infty}(\Omega)\cap W^{1,2}_0(\Omega))$ satisfying \eqref{bound by K} that
            \begin{equation*}
                \begin{aligned}
                &\bigl\| \nablax(\mathcal{T}_{c_0}\circ\mathcal{S}_{\rho_0}[\vvec])(\tau) - \nablax ( \mathcal{T}_{c_0}\circ \mathcal{S}_{\rho_0}[\mathbf{w}])(\tau) \bigr\|_{L^2(\Omega)}
                \\
                &\leq 
                C(\artvisc,K)
                \tau^{\frac{3}{2}}
                \|\rho_0\|_{W^{1,2}(\Omega)}
                \|\vvec - \mathbf{w}\|_{L^\infty(0,\tau;W^{1,\infty}(\Omega))}
            \end{aligned}
            \end{equation*}
            for any $\tau \in [0,T]$.
        \end{enumerate}
    \end{lemma}
    \begin{proof}
        For a proof of the existence of the operator $\mathcal{S}_{\rho_0}$ satisfying all the relations in (i) we refer to \cite[Proposition~7.39]{NovotnyStraskraba2004}.
        The existence of the solution operator $\mathcal{T}_{c_0}$ satisfying the relations in \eqref{T_c0 operator equations} follows from standard theory on linear parabolic partial differential equations (see e.g.~\cite[Theorem~11.29]{FeireislNovotny2017singlim}).
        We fix $\vvec,\mathbf{w}\in L^\infty(0,T;W^{1,\infty}(\Omega)\cap W^{1,2}_0(\Omega))$ and set $r:=\mathcal{S}_{\rho_0}[\vvec]$, $s := \mathcal{S}_{\rho_0}[\mathbf{w}]$, $d:=r-s$, $f:=\mathcal{T}_{c_0}\circ\mathcal{S}_{\rho_0}[\vvec]$, $g:=\mathcal{T}_{c_0}\circ\mathcal{S}_{\rho_0}[\mathbf{w}]$, $h:=f-g$.
        Then, $f$ satisfies 
        \begin{equation*}
            \begin{aligned}
                &\paracoup\partial_t f - \kappa \Deltax f + \coup f = \coup r \quad \text{a.e.~in } \OmegaT,
                \\
                &\nablax f(\tau)\cdot \mathbf{n}_{\partial \Omega} = 0 \quad \text{in the sense of traces for a.e.~} \tau \in (0,T),
                \\
                &f(0) = c_0 \quad \text{a.e.~in } \Omega.
            \end{aligned}
        \end{equation*}
        From parabolic regularity estimates (see e.g.~\cite[Theorem~11.29]{FeireislNovotny2017singlim}) we infer that
        \begin{equation*}
            \begin{aligned}
                \|\nablax f(\tau)\|_{L^2(\Omega)}
                &\leq 
                C\bigg( \|f(0)\|_{W^{1,2}(\Omega)} + \|r\|_{L^2(0,\tau;L^2(\Omega))}\bigg)
                \\
                &\leq 
                C\bigg( \|c_0\|_{W^{1,2}(\Omega)} + \overline{\rho}\exp \Big( \int_0^\tau \|\vvec\|_{W^{1,\infty}(\Omega)}\, \dd t \Big) \bigg),
            \end{aligned}
        \end{equation*}
        where we have used $\eqref{S_rho0 infty bds}$ and the fact that $f(0)=c_0$ to obtain the last inequality.
        Using the fact that $h(0)=0$ as well as $\eqref{S_rho0 diff est}$, we obtain by the same argument that
        \begin{equation*}
            \|\nablax h(\tau)\|_{L^2(\Omega)} 
            \leq 
            C \|d\|_{L^2(0,\tau;L^2(\Omega))}
            \leq C(K,\eps) \tau^{\frac{3}{2}}\|\rho_0\|_{W^{1,2}(\Omega)}\|\vvec-\mathbf{w}\|_{L^\infty(0,\tau;W^{1,\infty}(\Omega))},
        \end{equation*}
        provided that $\eqref{bound by K}$ holds.
    \end{proof}
    In order to formulate the fixed point problem, we introduce for $f \in C([0,T];L^1(\Omega))$ a family of linear operators via
    \begin{equation}\label{curlyM operator}
        \mathcal{M}_{f(\tau)}\colon X_n \to X_n^\ast \simeq X_n,
        \quad 
        \left\langle\mathcal{M}_{f(\tau)}\vvec,\mathbf{w} \right\rangle:= \int_\Omega f(\tau) \vvec\cdot \mathbf{w}\, \dd x
    \end{equation}
    for any $\tau \in [0,T]$ and any $n \in \NN$.
    This family of operators satisfies
    \begin{equation}\label{curlyM est}
        \left\|\mathcal{M}_{f(\tau)}\right\|_{\mathcal{L}(X_n,X_n)} \leq C(n) \|f(\tau)\|_{L^1(\Omega)} \quad \forall \, \tau \in [0,T].
    \end{equation}
    If we assume in addition that $f\geq a$ a.e.~in $\OmegaT$ for some constant $a>0$, then $\mathcal{M}_{f(\tau)}$ is invertible with
    \begin{equation}\label{curlyM inv est}
        \left\|\mathcal{M}_{f(\tau)}^{-1}\right\|_{\mathcal{L}(X_n,X_n)} 
        \leq 
        \frac{1}{a} \quad \forall\, \tau \in [0,T].
    \end{equation}
    For a second function $g \in C([0,T];L^1(\Omega))$ with $g\geq a$ a.e.~in $\OmegaT$ we have the estimates
    \begin{align}
        &\left\| \mathcal{M}^{-1}_{g(\tau)} \mathcal{M}_{f(\tau)} \mathcal{M}_{g(\tau)}^{-1} \right\|_{\mathcal{L}(X_n,X_n)} 
        \leq 
        \frac{C(n)}{a^2}\|f(\tau)\|_{L^1(\Omega)}, \notag 
        \\
        &\left\|\mathcal{M}_{f(\tau)} - \mathcal{M}_{g(\tau)}\right\|_{\mathcal{L}(X_n,X_n)}
        \leq 
        C(n) \left\| f(\tau) - g(\tau) \right\|_{L^1(\Omega)}, \notag 
        \\
        &\left\|\mathcal{M}_{f(\tau)}^{-1} - \mathcal{M}_{g(\tau)}^{-1} \right\|_{\mathcal{L}(X_n,X_n)}
        \leq 
        \frac{C(n)}{a^2}\|f(\tau)-g(\tau)\|_{L^1(\Omega)},\label{curlyM inv diff est}
    \end{align}
    which hold true for any $\tau \in [0,T]$.
    If we assume in addition that $\partial_t f \in L^1(\OmegaT)$, then we have for any $\vvec,\mathbf{w}\in X_n$ that
    \begin{equation}\label{curlyM inv timederiv}
    \partial_t
        \big\langle
            \mathcal{M}_{f}^{-1} \vvec, \mathbf{w}
        \big\rangle
        =
        \big\langle 
            \mathcal{M}_{f}^{-1} \mathcal{M}_{\partial_t f} \mathcal{M}_f^{-1} \vvec,\mathbf{w}
        \big\rangle
        \quad \text{in } \mathcal{D}^\prime((0,T)).
    \end{equation}  
    For a proof of these properties we refer to \cite[Section~7.7]{NovotnyStraskraba2004}.
    In order to prove \Cref{prop:Gal approx}, we wish to find for $n\in\NN$ some $\uvec_n\in C([0,T];X_n)$ that satisfies for any $\tau \in [0,T]$ the relation
    \begin{equation}\label{mom prblm integrals}
        \begin{aligned}
            \int_\Omega 
            \left(
                \rho_n(\tau)\uvec_n(\tau) - \proj(\rho_0\uvec_{0n})
            \right)
            \cdot \psivec
            \, \dx
            =
            \int_0^\tau \int_\Omega 
            \mathcal{N}(\rho_n,\uvec_n,c_n) \cdot \psivec
            \, \dx \, \dt
            \quad \forall\, \psivec\in X_n,
        \end{aligned}
    \end{equation}
    with $\rho_n$ and $c_n$ satisfying \eqref{Gal approx cont+parab}--\eqref{Gal approx cont+parab IC}, and where
    \begin{equation}\label{fixed pt nonlin op}
        \mathcal{N}(\rho_n,\uvec_n,c_n)
        :=
        \div\Svisc(\nablax\uvec_n) - \div(\rho_n\uvec_n\otimes\uvec_n) - \nablax p_{\coup,\artprs}(\rho_n)  - \artvisc\nablax\rho_n\cdot\nablax\uvec_n + \coup \rho_n \nablax c_n.
    \end{equation}
    With the operators in \Cref{lem:sol operators} and the family of operators introduced in \eqref{curlyM operator} at hand, we can rewrite the problem \eqref{mom prblm integrals} equivalently as a fixed point problem via
    \begin{equation}\label{fixed pt prblm}
        \uvec_n(\tau) 
        =
        \mathcal{M}^{-1}_{\mathcal{S}_{\rho_0}[\uvec_n](\tau)}
        \left(
        \proj(\rho_0\uvec_{0n})
        +
        \int_0^\tau
        \proj
        \nonlin \bigl( \mathcal{S}_{\rho_0}[\uvec_n], \uvec_n, \mathcal{T}_{c_0}\circ \mathcal{S}_{\rho_0}[\uvec_n] \bigr)
        \, \dt
        \right)
        \quad \forall\, \tau \in [0,T].
    \end{equation}
    In the following auxiliary lemma we collect useful estimates that help us verifying that the right-hand side of \eqref{fixed pt prblm} defines a contraction locally-in-time.
    \begin{lemma}\label{lem:prep fixed pt}
        Let the hypotheses and notations of \Cref{prop:Gal approx} hold true and let $\tau \in [0,T]$.
        Assume
        \begin{equation}\label{fixed pt prbl bound IC}
            \|\rho_0\|_{W^{1,2}(\Omega)}
            +
            \|c_0\|_{W^{1,2}(\Omega)}
            +
            \|\proj(\rho_0\uvec_{0n})\|_{X_n}
            \leq 
            M,
        \end{equation}
        for some $M>0$.
        Then the map
        \begin{equation*}
            \mathbf{F}_\tau \colon C([0,\tau],X_n) \to C([0,\tau],X_n),
        \end{equation*}
        \begin{equation*}
        \vvec\mapsto
            \mathbf{F}_\tau(\vvec)(s) 
            :=
            \mathcal{M}_{S_{\rho_0}[\vvec](s)}^{-1}
            \left(
                \proj (\rho_0 \uvec_{0n})
                +
                \int_0^s 
                \proj \nonlin\bigl(\mathcal{S}_{\rho_0}[\vvec],\vvec,\mathcal{T}_{c_0}\circ \mathcal{S}_{\rho_0}[\vvec] \bigr) \, \dt
            \right)
            \quad \forall\, s \in [0,\tau],
        \end{equation*}
        with $\nonlin(\cdot,\cdot,\cdot)$ defined in \eqref{fixed pt nonlin op}, is well-defined.
        Moreover, we have for any $K\in(0,\infty)$ and any
        \begin{equation*}
            \vvec\in B_{\tau,K}
            :=
            \left\{\vvec \in C([0,\tau];X_n) \mid \|\vvec\|_{L^\infty(0,\tau;W^{1,\infty}(\Omega))} \leq K\right\}
        \end{equation*}
        that
        \begin{equation}\label{self-mapping est}
            \|\mathbf{F}(\vvec)\|_{C([0,\tau],X_n)}
            \leq 
            \underline{\rho}^{-1} \exp(K\tau) \left\| \proj (\rho\uvec)_0 \right\|_{X_n}
            +
            C(n,K,M,\overline{\rho},\underline{\rho})\tau,
        \end{equation}
        and for any $\vvec_1,\vvec_2 \in B_{\tau,K}$ that 
        \begin{equation}\label{contraction est}
        \begin{aligned}
            \left\|\mathbf{F}_\tau(\vvec_1) - \mathbf{F}_\tau(\vvec_2)\right\|_{C([0,\tau],X_n)}
            \leq C(n,K,M,\overline{\rho},\underline{\rho},\artvisc)
            \tau 
            \|\vvec_1-\vvec_2\|_{C([0,\tau],X_n)}.
        \end{aligned}
        \end{equation}
    \end{lemma}
    \begin{proof}
        We fix $\tau \in [0,T]$ and $\vvec\in B_{\tau,K}$.
        Using \eqref{equiv norms Xn} and \eqref{prpties proj}, we estimate
        \begin{equation}\label{est proj nonlin (t) I}
            \begin{aligned}
                &\left\|
                \proj\mathcal{N}\bigl(\mathcal{S}_{\rho_0}[\vvec],\vvec,\mathcal{T}_{c_0}\circ\mathcal{S}_{\rho_0}[\vvec]\bigr)(s)
                \right\|_{X_n}
                \\
                &\leq
                C(n)
                \Bigl\{
                    \left\|\vvec(s)\right\|_{X_n}
                    +
                    \left\|p_{\coup,\artprs}\bigl(\mathcal{S}_{\rho_0}[\vvec](s)\bigr)\right\|_{L^1(\Omega)}
                    \\
                    &\qquad\qquad+
                    \|\mathcal{S}_{\rho_0}[\vvec](s)\|_{L^\infty(\Omega)}
                    \Bigl(
                    \|\vvec(s)\|_{X_n}^2+\|\vvec(s)\|_{X_n}
                    +
                    \|\nablax \bigl(\mathcal{T}_{c_0}\circ \mathcal{S}_{\rho_0}[\vvec]\bigr)(s)\|_{L^1(\Omega)}
                    \Bigr)
                \Bigr\}
            \end{aligned}
        \end{equation}
        for any $s \in [0,\tau]$.
        Then, by using \Cref{lem:sol operators} we conclude
        \begin{equation}\label{est proj nonlin (t) II}
            \left\|
                \proj\mathcal{N}\bigl(\mathcal{S}_{\rho_0}[\vvec],\vvec,\mathcal{T}_{c_0}\circ\mathcal{S}_{\rho_0}[\vvec]\bigr)(s)
            \right\|_{X_n}
            \leq C(n,K,M,\overline{\rho})\qquad \forall \, s \in [0,\tau].
        \end{equation}
        Next, by writing for any $s_1,s_2 \in [0,\tau]$
        \begin{equation*}
            \begin{aligned}
                &\mathcal{M}_{\mathcal{S}_{\rho_0}[\vvec](s_1)}^{-1}
                \int_0^{s_1} \proj \nonlin\bigl( \mathcal{S}_{\rho_0}[\vvec], \vvec, \mathcal{T}_{c_0}\circ \mathcal{S}_{\rho_0}[\vvec] \bigr) \, \dt
                -
                \mathcal{M}_{S_{\rho_0}[\vvec](s_2)}^{-1}
                \int_0^{s_2} \proj \nonlin\bigl( S_{\rho_0}[\vvec], \vvec, \mathcal{T}_{c_0}\circ \mathcal{S}_{\rho_0}[\vvec]\bigr) \, \dt
                \\
                &= 
                \mathcal{M}_{\mathcal{S}_{\rho_0}(s_1)}^{-1} 
                \int_{s_2}^{s_1} \proj \nonlin \bigl(\mathcal{S}_{\rho_0}[\vvec], \vvec, \mathcal{T}_{c_0}\circ \mathcal{S}_{\rho_0}[\vvec] \bigr)\, \dt
                \\
                &\quad +
                \left( \mathcal{M}_{\mathcal{S}_{\rho_0}[\vvec](s_1)}^{-1} - \mathcal{M}_{\mathcal{S}_{\rho_0}[\vvec](s_2)}^{-1} \right)
                \int_0^{s_2} \proj \nonlin\bigl(\mathcal{S}_{\rho_0}[\vvec], \vvec, \mathcal{T}_{c_0}\circ \mathcal{S}_{\rho_0}[\vvec]\bigr) \, \dt,
            \end{aligned}
        \end{equation*}
        we conclude from \eqref{curlyM inv est}, \eqref{curlyM inv diff est}, \eqref{est proj nonlin (t) I}, and \Cref{lem:sol operators} that
        \begin{equation*}
            \mathcal{M}_{\mathcal{S}_{\rho_0}[\vvec](\cdot)}^{-1} \int_0^{\cdot} \proj\nonlin\bigl( \mathcal{S}_{\rho_0}[\vvec], \vvec, \mathcal{T}_{c_0} \circ \mathcal{S}_{\rho_0} [\vvec]\bigr) \, \dt \in C([0,\tau];X_n).
        \end{equation*}
        With \eqref{S_rho0 infty bds} and \eqref{curlyM inv diff est} we infer that
        \begin{equation*}
            \mathcal{M}_{\mathcal{S}_{\rho_0}[\vvec](\cdot)}^{-1} \proj (\rho_0\uvec_{0n}) \in C([0,\tau];X_n).
        \end{equation*}
        Thus we have $\mathbf{F}_\tau(\vvec) \in C([0,\tau],X_n)$ and $\mathbf{F}_\tau$ is well-defined.
        To prove \eqref{self-mapping est}, we use \Cref{lem:sol operators} as well as \eqref{curlyM inv est} and \eqref{est proj nonlin (t) II} to estimate 
        \begin{equation}\label{self-mapping est I}
        \begin{aligned}
            &\left\|
            \mathcal{M}_{\mathcal{S}_{\rho_0}[\vvec](\cdot)}^{-1}
            \int_0^{\cdot} \proj \nonlin \bigl( \mathcal{S}_{\rho_0}[\vvec], \vvec, \mathcal{T}_{c_0}\circ \mathcal{S}_{\rho_0}[\vvec]\bigr) \, \dd t
            \right\|_{C([0,\tau];X_n)}
            \\
            &\leq 
            C(n) \underline{\rho}^{-1} \exp(KT) \tau \left\| \proj \nonlin\bigl(\mathcal{S}_{\rho_0}[\vvec],\vvec, \mathcal{T}_{c_0}\circ \mathcal{S}_{\rho_0}[\vvec]\bigr)\right\|_{C([0,\tau];X_n)} 
            \\
            &\leq C(n,K,M,\overline{\rho},\underline{\rho})\tau,
        \end{aligned}
        \end{equation}
        and further,
        \begin{equation}\label{self-mapping est II}
            \left\|
            \mathcal{M}_{\mathcal{S}_{\rho_0}[\vvec](\cdot)}^{-1} \proj(\rho_0\uvec_{0n})
            \right\|_{C([0,\tau];X_n)}
            \leq 
            \underline{\rho}^{-1}\exp(K\tau) \|\proj(\rho_0\uvec_{0n})\|_{X_n}.
        \end{equation}
        Combining \eqref{self-mapping est I} and \eqref{self-mapping est II} leads to \eqref{self-mapping est}.
        To prove \eqref{contraction est}, we fix $\vvec_1,\vvec_2 \in B_{\tau,K}$.
        For any $\psivec\in X_n$, we have the relation
        \begin{equation*}
            \begin{aligned}
                &\left\langle 
                    \nonlin\bigl( \mathcal{S}_{\rho_0}[\vvec_1],\vvec_1,\mathcal{T}_{c_0}\circ\mathcal{S}_{\rho_0}[\vvec_1]\bigr)
                    -
                    \nonlin\bigl( \mathcal{S}_{\rho_0}[\vvec_2],\vvec_2,\mathcal{T}_{c_0}\circ\mathcal{S}_{\rho_0}[\vvec_2] \bigr)
                    ,
                    \psivec
                \right\rangle
                \\
                &=
                \int_\Omega \div\Svisc(\nablax\vvec_1-\nablax\vvec_2)\cdot \psivec \, \dd x
                +
                \int_\Omega \div\psivec\int_{\mathcal{S}_{\rho_0}[\vvec_2]}^{\mathcal{S}_{\rho_0}[\vvec_1]}p_{\coup,\artprs}^\prime(\xi)\, \dd \xi \, \dd x
                \\&\quad
                +
                \int_\Omega \left( \mathcal{S}_{\rho_0}[\vvec_1]-\mathcal{S}_{\rho_0}[\vvec_2] \right) \vvec_1\otimes \vvec_1 : \nablax \psivec \, \dd x
                +
                \int_\Omega \mathcal{S}_{\rho_0}[\vvec_2]\left(\vvec_1-\vvec_2\right)\otimes\vvec_1 : \nablax \psivec\, \dd x
                \\&\quad
                +
                \int_\Omega \mathcal{S}_{\rho_0}[\vvec_2] \vvec_2 \otimes \left(\vvec_1-\vvec_2\right) : \nablax \psivec \, \dd x
                +
                \int_\Omega\artvisc \left( \mathcal{S}_{\rho_0}[\vvec_1] - \mathcal{S}_{\rho_0}[\vvec_2] \right) \left( \Deltax \vvec_1 \cdot \psivec + \nablax \vvec_1 : \nablax \psivec \right) \, \dd x
                \\&\quad+
                \int_\Omega 
                \artvisc\mathcal{S}_{\rho_0}[\vvec_2]\left( \Deltax(\vvec_1-\vvec_2)\cdot\psivec + \nablax(\vvec_1-\vvec_2):\nablax\psivec \right)\, \dd x
                \\&\quad
                +
                \int_\Omega \coup \left( \mathcal{S}_{\rho_0}[\vvec_1] - \mathcal{S}_{\rho_0}[\vvec_2] \right) \nablax\left(\mathcal{T}_{c_0}\circ\mathcal{S}_{\rho_0}[\vvec_1]\right) \cdot \psivec\, \dd x
                \\&\quad
                +
                \int_\Omega \coup \mathcal{S}_{\rho_0}[\vvec_2]\left(\nablax\left(\mathcal{T}_{c_0}\circ\mathcal{S}_{\rho_0}[\vvec_1]\right) - \nablax\left(\mathcal{T}_{c_0}\circ\mathcal{S}_{\rho_0}[\vvec_2]\right)\right)\cdot \psivec
                \, \dd x,
            \end{aligned}
        \end{equation*}
        which holds on $[0,\tau]$.
        In view of \eqref{equiv norms Xn} and \Cref{lem:sol operators}, we deduce from this relation for any $s \in [0,\tau]$ that
        \begin{equation}\label{contraction est I}
            \begin{aligned}
                &\left\|
                \proj\nonlin\bigl( \mathcal{S}_{\rho_0}[\vvec_1],\vvec_1,\mathcal{T}_{c_0}\circ\mathcal{S}_{\rho_0}[\vvec_1]\bigr)(s)
                -
                \proj\nonlin\bigl( \mathcal{S}_{\rho_0}[\vvec_2],\vvec_2,\mathcal{T}_{c_0}\circ\mathcal{S}_{\rho_0}[\vvec_2] \bigr)(s)
                \right\|_{X_n}
                \\
                &\leq
                C(n,K,M,\overline{\rho},\artvisc)\|\vvec_1-\vvec_2\|_{C([0,\tau],X_n)}.
            \end{aligned}
        \end{equation}
        With \eqref{curlyM inv diff est} and \Cref{lem:sol operators} we estimate further for any $s \in [0,\tau]$
        \begin{equation}\label{contraction est II}
        \begin{aligned}
            &\left\|
            \mathcal{M}_{\mathcal{S}_{\rho_0}[\vvec_1](s)}^{-1}
            -
            \mathcal{M}_{\mathcal{S}_{\rho_0}[\vvec_2](s)}^{-1} 
            \right\|_{\mathcal{L}(X_n,X_n)}
            \\
            &\leq 
            C(n,K,\artvisc)\underline{\rho}^{-2} s \|\rho_0\|_{W^{1,2}(\Omega)} \|\vvec_1 - \vvec_2\|_{C([0,\tau],X_n)}
            \\
            &\leq C(n,K,M,\underline{\rho},\artvisc) \tau \|\vvec_1-\vvec_2\|_{C([0,\tau];X_n)}.
        \end{aligned}
        \end{equation}
        Now, we write for any $s \in [0,\tau]$
        \begin{equation}\label{contraction est III}
            \begin{aligned}
                &\mathbf{F}_\tau(\vvec_1)(s) - \mathbf{F}_{\tau}(\vvec_2)(s)
                \\
                &=
                \left( \mathcal{M}_{\mathcal{S}_{\rho_0}[\vvec_1](s)}^{-1} - \mathcal{M}_{\mathcal{S}_{\rho_0}[\vvec_2](s)}^{-1} \right)
                \left( \proj (\rho\uvec)_0 + \int_0^s \proj\nonlin\bigl(\mathcal{S}_{\rho_0}[\vvec_1],\vvec_1,\mathcal{T}_{c_0}\circ\mathcal{S}_{\rho_0}[\vvec_1]\bigr)\, \dt\right)
                \\
                &\quad +
                \mathcal{M}_{\mathcal{S}_{\rho_0}[\vvec_2](s)}^{-1}\int_0^s \Bigl( \proj\nonlin \bigl(\mathcal{S}_{\rho_0}[\vvec_1],\vvec_1,\mathcal{T}_{c_0}\circ\mathcal{S}_{\rho_0}[\vvec_1]\bigr) - \proj \nonlin \bigl(\mathcal{S}_{\rho_0}[\vvec_2],\vvec_2,\mathcal{T}_{c_0}\circ\mathcal{S}_{\rho_0}[\vvec_2]\bigr) \Bigr) \, \dt.
            \end{aligned}
        \end{equation}
        Combining \eqref{fixed pt prbl bound IC}, \eqref{est proj nonlin (t) II}, and \eqref{contraction est I}--\eqref{contraction est III} leads to
        \begin{equation*}
        \begin{aligned}
            \left\|\mathbf{F}_\tau(\vvec_1) - \mathbf{F}_\tau(\vvec_2)\right\|_{C([0,\tau],X_n)}
            \leq 
            C(n,K,M,\overline{\rho},\underline{\rho},\artvisc)\tau \|\vvec_1- \vvec_2\|_{C([0,\tau];X_n)}.
        \end{aligned}
        \end{equation*}
    \end{proof}
    Now we have everything prepared to prove the local-in-time existence and uniqueness of solutions to \eqref{Gal approx cont+parab}--\eqref{Gal approx Energy} via a fixed point argument.
    \begin{lemma}\label{lem:local-in-time exSol}
        Let the hypotheses and notations of \Cref{prop:Gal approx} hold true. 
        Then there exists a time $\Tloc>0$ only depending on $n,M,\underline{\rho},\overline{\rho},\varepsilon$, and on the quantities in $\eqref{data - fixed params}$ such that there is a unique triplet $(\rho_n,\uvec_n,c_n) \in \mathcal{R}_{\Tloc}\times C([0,\Tloc];X_n)\times \mathcal{Q}_{\Tloc}$ with $\partial_t \uvec_n \in L^2(0,\Tloc;X_n)$ satisfying \eqref{Gal approx cont+parab}--\eqref{Gal approx Energy} with $T$ replaced by $\Tloc$.
    \end{lemma}
    \begin{proof}
        We set
        \begin{equation*}
            K:=8\underline{\rho}^{-1}M.
        \end{equation*}
        In view of \Cref{lem:prep fixed pt}, there exists a constant $c_1>0$ only depending on $n,M,\overline{\rho},\underline{\rho}$, and on the quantities in $\eqref{data - fixed params}$ such that we have for any $\tau \in [0,T]$
        \begin{equation*}
            \|\mathbf{F}_{\tau} (\vvec)\|_{C([0,\tau];X_n)} 
            \leq 
            \underline{\rho}^{-1}\exp(K\tau) \|\proj(\rho_0\uvec_{0n})\|_{X_n}
            +
            c_1 \tau 
            \qquad \forall\, \vvec\in B_{\tau,K}.
        \end{equation*}
        In particular, by setting
        \begin{equation*}
            T_0:=\min\left\{ \frac{\ln 2}{K}, \frac{K}{2 c_1},T \right\},
        \end{equation*}
        we have 
        \begin{equation*}
            \mathbf{F}_\tau (B_{\tau,K})\subseteq B_{\tau,K} \qquad \forall \tau \in [0,T_0].
        \end{equation*}
        Note that $T_0$ only depends on $n,M,\overline{\rho},\underline{\rho}$, and on the quantities in $\eqref{data - fixed params}$.
        Again by \Cref{lem:prep fixed pt}, there exists some constant $c_2>0$ only depending on $n,M,\overline{\rho},\underline{\rho},\artvisc$, and on the quantities in $\eqref{data - fixed params}$ such that for any $\tau \in [0,T_0]$
        \begin{equation*}
            \|\mathbf{F}_\tau(\vvec_1)-\mathbf{F}_{\tau}(\vvec_2)\|_{C([0,\tau],X_n)}
            \leq 
            c_2 \tau \|\vvec_1-\vvec_2\|_{C([0,\tau];X_n)}
            \qquad \forall \, \vvec_1,\vvec_2 \in B_{\tau,K}.
        \end{equation*}
        Thus, by setting
        \begin{equation*}
            \Tloc:=\min\left\{ T_0, \frac{1}{2c_2} \right\},
        \end{equation*}
        we have that $\Tloc$ only depends on $n,M,\overline{\rho},\underline{\rho},\artvisc$, and on the quantities in $\eqref{data - fixed params}$, and moreover
        \begin{equation*}
            \|\mathbf{F}_{\Tloc}(\vvec_1) - \mathbf{F}_{\Tloc}(\vvec_2)\|_{C([0,\Tloc];X_n)}
            \leq 
            \frac{1}{2}\|\vvec_1-\vvec_2\|  \qquad\forall \vvec_1,\vvec_2 \in B_{\Tloc,K}.
        \end{equation*}
        In particular, $\mathbf{F}_{\Tloc}\colon B_{\Tloc,K}\to B_{\Tloc,K}$ is a well-defined contraction.
        With the help of Banach's fixed point theorem (see e.g.~\cite[Section~1.4.11.2]{NovotnyStraskraba2004}), we conclude that there exists precisely one $\uvec_n\in B_{\Tloc,K}$ that satisfies
        \begin{equation}\label{fixed pt eq sol}
            \mathbf{F}_{\Tloc}(\uvec_n) = \uvec_n \quad \text{in } [0,\Tloc].
        \end{equation}
        This implies that the triplet $\left(\rho_n,\uvec_n,c_n\right)$ with $\rho_n:=\mathcal{S}_{\rho_0}[\uvec_n]\in \mathcal{R}_{\Tloc}$ and $c_n:=\mathcal{T}_{c_0}\circ\mathcal{S}_{\rho_0}[\uvec_n]\in \mathcal{Q}_{\Tloc}$ satisfies \eqref{Gal approx cont+parab}--\eqref{Gal approx cont+parab IC}, \eqref{Gal approx mom IC}, and \eqref{mom prblm integrals} with $T$ replaced by $\Tloc$.
        With \eqref{curlyM inv timederiv} we conclude from \eqref{fixed pt eq sol} the relation
        \begin{equation*}
            \begin{aligned}
                \partial_t \uvec_n
                &=
                \mathcal{M}_{\rho_n}^{-1}\mathcal{M}_{\partial_t\rho_n}\mathcal{M}_{\rho_n}^{-1}\proj(\rho_0\uvec_{0n})
                +
                \mathcal{M}_{\rho_n}^{-1} \mathcal{M}_{\partial_t \rho_n}\mathcal{M}_{\rho_n}^{-1} 
                \left( \int_0^{\cdot} \proj \nonlin(\rho_n,\uvec_n,c_n)\, \dd s  \right)
                \\
                &\quad+
                \mathcal{M}_{\rho_n}^{-1}\proj \nonlin(\rho_n,\uvec_n,c_n)
                \qquad \text{in } \mathcal{D}^\prime((0,\Tloc);X_n).
            \end{aligned}
        \end{equation*}
        Using \eqref{curlyM est}, \eqref{curlyM inv est}, and the same estimates leading to \eqref{est proj nonlin (t) II}, we deduce from the preceding relation that $\partial_t \uvec_n \in L^2(0,\Tloc;X_n)$.
        In view of \eqref{mom prblm integrals}, we thus have that \eqref{Gal approx mom} holds with $T$ replaced by $\Tloc$, that is, we have for any $\psivec\in X_n$ the relation
        \begin{equation}\label{int rel time deriv u_n}
            \begin{aligned}
                \int_\Omega 
                \partial_t(\rho_n\uvec_n)\cdot \psivec 
                \, \dd x 
                =
                \int_\Omega 
                \nonlin(\rho_n,\uvec_n,c_n) \cdot \psivec\, 
                \, \dd x \quad \text{a.e.~in } (0,\Tloc),
            \end{aligned}
        \end{equation}
        where $\nonlin(\cdot,\cdot,\cdot)$ is defined in \eqref{fixed pt nonlin op}.
        To deduce from \eqref{int rel time deriv u_n} that the energy inequality \eqref{Gal approx Energy} holds with $T$ replaced by $\Tloc$, we use $\uvec_n$ as a test function in \eqref{int rel time deriv u_n}.
        We have that
        \begin{equation}\label{Gal energy id I}
            \int_\Omega \partial_t (\rho_n\uvec_n)\cdot \uvec_n \, \dd x
            =
            \frac{\dd}{\dd t} 
            \int_\Omega \frac{1}{2}\rho_n |\uvec_n|^2 
            \, \dd x 
            +
            \frac{1}{2} 
            \int_\Omega \partial_t\rho_n\, |\uvec_n|^2
            \, \dd x
            \quad \text{a.e.~in } (0,\Tloc),
        \end{equation}
        and, by using integration by parts,
        \begin{equation}\label{Gal energy id II}
        \begin{aligned}
            &\int_\Omega \div\Svisc(\nablax\uvec_n) \cdot \uvec_n \, \dd x 
            = 
            -\int_\Omega \Svisc(\nablax\uvec_n):\nablax\uvec_n \, \dd x
            \quad \text{a.e.~in } (0,\Tloc),
            \\
            &-\int_\Omega \artvisc \nablax\rho_n\cdot\nablax\uvec_n \cdot \uvec_n \, \dd x
            =
            \frac{1}{2} \int_\Omega \artvisc \Deltax \rho_n \, |\uvec_n|^2 \, \dd x
            \quad \text{a.e.~in } (0,\Tloc).
        \end{aligned}
        \end{equation}
        Using integration by parts and the fact that $\rho_n$ satisfies the continuity equation $\eqref{Gal approx cont+parab}_1$ a.e.~on $\Omega_{\Tloc}$ leads to
        \begin{equation}\label{Gal energy id III}
            -\int_\Omega \div(\rho_n\uvec_n\otimes\uvec_n) \cdot \uvec_n\, \dd x
            =
            \frac{1}{2}
            \int_\Omega  \partial_t \rho_n \, |\uvec_n|^2 \, \dd x 
            -\frac{1}{2}
            \int_\Omega \artvisc\Deltax \rho_n \, |\uvec_n|^2 \, \dd x
        \end{equation}
        a.e.~in $(0,\Tloc)$.
        Using integration by parts and the fact that $\rho_n$ and $c_n$ satisfy \eqref{Gal approx cont+parab} a.e.~in $\Omega_{\Tloc}$, we infer that
        \begin{equation}\label{Gal energy id IV}
        \begin{aligned}
            &\int_\Omega 
            \coup \rho_n \nablax (c_n - \rho_n) \cdot \uvec_n 
            \, \dx
            =
            \int_\Omega \coup \partial_t \rho_n (c_n - \rho_n) \, \dx
            -
            \int_\Omega \artvisc \coup \Deltax \rho_n (c_n-\rho_n)\, \dx
            \\
            &=
            -\frac{\dd}{\dd t}\int_\Omega  \frac{\coup}{2} |\rho_n-c_n|^2 + \frac{\kappa}{2} |\nablax c_n|^2  \, \dx
            -\int_\Omega\artvisc \coup |\nablax\rho_n|^2 + \paracoup|\partial_t c_n|^2 \, \dx
            \\
            &\qquad 
            +\int_\Omega \artvisc \coup \nablax\rho_n \cdot \nablax c_n \, \dx
            \qquad \text{a.e.~in } (0,\Tloc).
        \end{aligned}
        \end{equation}
        Since $(\rho_n,\uvec_n,c_n)$ satisfies the continuity equation $\eqref{Gal approx cont+parab}_1$, and since $\rho_n$ is bounded from below and above by some positive constant, we have for any $b \in C^0([0,\Tloc))\cap C^1((0,\Tloc))$ that 
        \begin{equation*}
            \partial_t b(\rho_n) + \div(b(\rho_n)\uvec_n) + \left( b^\prime(\rho_n)\rho_n - b(\rho_n) \right) \div\uvec_n
            =
            \artvisc\Deltax \rho_n b^\prime(\rho_n)
            \quad \text{a.e.~in } \Omega_{\Tloc}.
        \end{equation*}
        Taking into account \eqref{relations potential-pressure}, this leads in particular to the relations
        \begin{equation}\label{Gal energy id V}
            \int_\Omega h(\rho_n)\div\uvec_n\, \dx 
            = 
            -\frac{\dd}{\dd t} \int_\Omega H(\rho_n) \dx
            -\int_\Omega \artvisc|\nablax\rho_n|^2 H^{\prime\prime}(\rho_n) \, \dx
            \leq 
            -\frac{\dd}{\dd t} \int_\Omega H(\rho_n)\, \dx
        \end{equation}
        and
        \begin{equation}\label{Gal energy id VI}
            \int_\Omega \delta \rho_n^\artgrwth \div\uvec_n
            \, \dx
            =
            -\frac{\dd}{\dd t} \int_\Omega \frac{\delta}{\artgrwth-1}\rho_n^\artgrwth \, \dx
            -\int_\Omega \artvisc\artprs\artgrwth \rho_n^{\Gamma-2} |\nablax\rho_n|^2 \, \dx,
        \end{equation}
        which hold a.e.~in $\Omega_{\Tloc}$, and where $H$ is the convex pressure potential corresponding to the monotone part $h$ defined in \eqref{dec prs}.
        Thus, using \eqref{int rel time deriv u_n} with $\uvec_n$ as a test function and combining \eqref{Gal energy id I}--\eqref{Gal energy id VI} yields precisely \eqref{Gal approx Energy} with $T$ replaced by $\Tloc$.
        \\
        The uniqueness part follows from the energy method and Gr\"onwall's inequality, similarly as in \cite{Valli1983}.
        For the sake of brevity we omit the details here.
    \end{proof}
    Now we exploit the energy inequality to obtain suitable uniform bounds that allow us to continue the local-in-time solution from \Cref{lem:local-in-time exSol} to a global-in-time solution completing the proof of \Cref{prop:Gal approx}.
    \begin{proof}[Proof of \Cref{prop:Gal approx}]
    In view of \Cref{lem:local-in-time exSol} we have that
        \begin{align*}
            \emptyset\neq 
            \mathcal{I}
            :=
            \{
            \tau\in(0,T]\mid  &\text{ there exists a unique triplet } (\rho_n,\uvec_n,c_n)\in \mathcal{R}_\tau\times C([0,\tau],X_n)\times \mathcal{Q}_\tau \\
            &\text{with }\partial_t \uvec_n\in L^2(0,\tau;X_n) \text{ solving \eqref{Gal approx cont+parab}--\eqref{Gal approx Energy} with } T \text{ replaced by } \tau 
            \},
        \end{align*}
    where the spaces $\mathcal{R}_\tau$ and $\mathcal{Q}_\tau$ are defined in \eqref{spaces Gal approx}.
    In particular, we have that
    \begin{equation}\label{def T'}
        T^\prime := \sup\mathcal{I} \in (0,T],
    \end{equation}
    and that there exists a unique triplet $(\rho_n,\uvec_n,c_n)$ with
    \begin{equation*}
        (\rho_n,\uvec_n,c_n) \in \mathcal{R}_\tau\times C([0,\tau];X_n)\times \mathcal{Q}_{\tau},
        \quad 
        \partial_t \uvec_n \in L^2(0,\tau,X_n)
        \qquad \forall \tau \in [0,T^\prime)
    \end{equation*}
    satisfying \eqref{Gal approx cont+parab}--\eqref{Gal approx Energy} with $T$ replaced by $T^\prime$.
    In order to show that in fact $T=T^\prime$, we exploit the energy inequality \eqref{Gal approx Energy}, which holds a.e.~in $(0,T^\prime)$.
    With Young's inequality, we estimate the right-hand side of \eqref{Gal approx Energy} a.e.~in $(0,T^\prime)$ via
    \begin{equation*}
        \begin{aligned}
            &\left|
            \int_\Omega
            q(\rho_n)\div\uvec_n\, \dd x
            \right|
            \leq
            \frac{3}{2\shearvisc}
            \int_\Omega |q(\rho_n)|^2 \, \dd x 
            + 
            \frac{\shearvisc}{6}
            \int_\Omega |\div \uvec_n|^2 \, \dd x
            \leq
            C + \frac{1}{2}\int_\Omega \Svisc(\nablax\uvec_n):\nablax\uvec_n \, \dd x
        \end{aligned}
    \end{equation*}
    and
    \begin{equation*}
        \begin{aligned}
            \left|
            \int_\Omega 
            \coup\artvisc\nablax\rho_n\cdot \nablax c_n\, \dd x
            \right|
            &\leq 
            \int_\Omega \frac{\artvisc\coup}{2}|\nablax\rho_n|^2\, \dd x + \int_\Omega \frac{\artvisc\coup}{2}|\nablax c_n|^2 \, \dd x
            \\
            &
            \leq 
            \int_\Omega \frac{\artvisc\coup}{2} |\nablax \rho_n|^2 \, \dd x 
            +
            C
            +
            C E_{\mathrm{m},\artprs}[\rho_n,\rho_n\uvec_n,c_n],
        \end{aligned}
    \end{equation*}
    where $E_{\mathrm{m,\artprs}}[\cdot,\cdot,\cdot]$ is defined in \eqref{def E_m(t)} and where we have used that $\artvisc\leq 1$ in the last inequality.
    This leads to
    \begin{equation}\label{cont arg I}
    \begin{aligned}
        &\frac{\dd}{\dd t} E_{\mathrm{m},\artprs}[\rho_n, \rho_n\uvec_n,c_n]
        +
        \int_\Omega \frac{1}{2}\Svisc(\nablax\uvec_n):\nablax\uvec + \paracoup|\partial_t c_n|^2   \, \dd x
        \\
        &\quad +
        \int_\Omega \frac{\artvisc\coup}{2} |\nablax \rho_n|^2 + \artvisc\artprs\artgrwth\rho_n^{\artgrwth-2}|\nablax \rho_n|^2  \, \dd x
        \leq 
        C+C E_{\mathrm{m},\artprs}[\rho_n,\rho_n\uvec_n,c_n]
        \quad \text{a.e.~in } (0,T^\prime).
    \end{aligned}
    \end{equation}
    With Gr\"onwall's inequality, we deduce from \eqref{cont arg I} that
    \begin{equation}\label{cont arg II}
        E_{\mathrm{m},\artprs}[\rho_n,\rho_n\uvec_n,c_n](\tau)
        \leq 
        C + C E_{\mathrm{m},\artprs}[\rho_0,\rho_0\uvec_{0n},c_0] 
        =:
        K_0 \quad \forall\, \tau \in [0,T^\prime).
    \end{equation}
    Relation \eqref{cont arg II} implies in particular that
    \begin{equation}\label{cont arg III}
        \|\sqrt{\rho_n}\uvec_n(\tau)\|_{L^2(\Omega)}
        +
        \|\nablax c_n(\tau)\|_{L^2(\Omega)} 
        +
        \|\rho_n(\tau)-c_n(\tau)\|_{L^2(\Omega)}
        \leq 
        C(K_0)\quad \forall \, \tau \in [0,T^\prime).
    \end{equation}
    Integrating \eqref{cont arg I} in time and using \eqref{cont arg II} leads to 
    \begin{equation}\label{cont arg IV}
        \|\uvec_n(\tau)\|_{L^2(0,\tau;W^{1,2}(\Omega))} \leq C(K_0) \quad \forall\, \tau \in [0,T^\prime).
    \end{equation}
    By virtue of \eqref{equiv norms Xn}, we deduce with \Cref{lem:sol operators} from \eqref{cont arg IV} for any $\tau \in [0,T^\prime)$ the estimates
    \begin{equation}\label{cont arg V}
    \begin{aligned}
        &\rho(\tau) 
        \leq \overline\rho \exp\left(\int_0^\tau \|\uvec_n(s)\|_{W^{1,\infty}(\Omega)}\, \dd s\right)
        \leq 
        C(n,\overline{\rho},K_0) &&\text{on } \Omega,
        \\
        &\rho(\tau)
        \geq \underline{\rho}\exp\left(-\int_0^\tau \|\uvec_n(s)\|_{W^{1,\infty}(\Omega)}\, \dd s\right)
        \geq C(n,\underline{\rho},K_0) &&\text{on } \Omega.
    \end{aligned}
    \end{equation}
    Combining \eqref{cont arg III} and \eqref{cont arg V} leads to
    \begin{equation}\label{cont arg VI}
        \|\uvec_n(\tau)\|_{L^2(\Omega)}
        +
        \|c_n(\tau)\|_{L^2(\Omega)}
        \leq 
        C(n,\overline{\rho},\underline{\rho},K_0) \quad \forall\, \tau \in [0,T^\prime),
    \end{equation}
    and, in particular,
    \begin{equation}\label{cont arg VII}
        \|\proj(\rho_n\uvec_n(\tau))\|_{X_n}
        \leq 
        \|\rho_n\uvec_n(\tau)\|_{L^2(\Omega)}
        \leq 
        C(n,\overline{\rho},\underline{\rho},K_0) \quad \forall\, \tau \in [0,T^\prime).
    \end{equation}
    Using again \eqref{equiv norms Xn}, we deduce from \eqref{cont arg VI} that
    \begin{equation*}
        \|\uvec_n(\tau)\|_{W^{1,\infty}(\Omega)}
        \leq 
        C(n,\overline{\rho},\underline{\rho},K_0) \quad \forall\, \tau \in [0,T^\prime),
    \end{equation*}
    which implies in view of \Cref{lem:sol operators} that
    \begin{equation}\label{cont arg VIII}
        \|\nablax\rho_n(\tau)\|_{L^2(\Omega)}\leq 
        C(n,\overline{\rho},\underline{\rho},K_0,\|\rho_0\|_{W^{1,2}(\Omega)},\artvisc)\quad \forall\,\tau\in[0,T^\prime).
    \end{equation}
    By virtue of $\eqref{cont arg III}$ and \eqref{cont arg V}--\eqref{cont arg VIII} we have shown that the bounds
    \begin{equation}\label{bds indep of tau}
        \begin{aligned}
            &\|\rho_n(\tau)\|_{W^{1,2}(\Omega)} 
            +
            \|c_n(\tau)\|_{W^{1,2}(\Omega)}
            +
            \|\proj(\rho_n\uvec_n)(\tau)\|_{X_n}
            \leq  m,\\
            &r \leq \rho_n(\tau) \leq R \qquad \text{on } \Omega
        \end{aligned}
    \end{equation}
    hold for any $\tau \in [0,T^\prime)$, for some constants $m,r,R>0$ that do not depend on $\tau$.
    Suppose now that $T^\prime < T$.
    Then, there exists an increasing sequence $(T_k)_{k\in\NN}\subseteq (0,T^\prime)$, such that $T_k\to T^\prime$ as $k \to \infty$.
    By virtue of \Cref{lem:local-in-time exSol} and the bounds in \eqref{bds indep of tau}, there exists some positive time $T_0>0$ that only depends on $m,r,R$, and on the quantities in $\eqref{data - fixed params}$ such that for any $k\in\NN$ the solution $(\rho_n,\uvec_n,c_n)$ can be uniquely extended onto the interval
    \begin{equation*}
        \bigl[0,\min\{T_k+T_0,T\}\bigr].
    \end{equation*}
    In particular, $T_0$ does not depend on $k$.
    Since $T_k\to T^\prime$ as $k\to\infty$, we have for large $k$ that 
    \begin{equation*}
        \min\left\{T_k+T_0, T\right\} > T^\prime,
    \end{equation*}
    which contradicts \eqref{def T'}.
    \end{proof}

    
    \subsection{Uniform estimates}\label{subsec:unif bds}
    By virtue of Proposition~\ref{prop:Gal approx}, there exists for any $n\in\NN$ a unique triplet $(\rho_n,\uvec_n,c_n)$ that satisfies \eqref{Gal approx cont+parab}--\eqref{Gal approx Energy}.
    This triplet can be seen as an approximate solution to the regularized rNSKE \eqref{rNSK(delta,eps)}--\eqref{rNSK(delta,eps) IC}, where $n$ plays the role of an approximation parameter.
    To prove Proposition~\ref{prop:wkSol rNSK(delta,eps)}, we wish to show that a solution of the regularized rNSKE \eqref{rNSK(delta,eps)}--\eqref{rNSK(delta,eps) IC} will be obtained by performing the approximation limit $n\to\infty$.
    To perform this limit, we have to apply compactness arguments that require certain bounds that are uniform with respect to $n$.
    These uniform bounds are the goal of this subsection.
    A first set of uniform bounds is directly provided by the energy inequality \eqref{Gal approx Energy}. 
    \begin{lemma}\label{lem:unif-in-n bds I}
        Let the hypotheses and notations of Proposition~\ref{prop:Gal approx} hold true.
        For any $n\in\NN$, let $(\rho_n,\uvec_n,c_n)\in \mathcal{R}_T\times C([0,T],X_n)\times\mathcal{Q}_T$ denote the unique triplet satisfying \eqref{Gal approx cont+parab}--\eqref{Gal approx Energy}.
        Then we have for any $n \in \NN$ that
        \begin{equation}\label{energy bds(n) I}
            \begin{aligned}
                &\|\sqrt{\rho_n}\uvec_n\|_{L^\infty(0,T;L^2(\Omega))}
                +
                \|\rho_n\|_{L^\infty(0,T;L^\gamma(\Omega))}
                +
                \|\rho_n-c_n\|_{L^\infty(0,T;L^2(\Omega))}
                \\
                &\quad +
                \|\nablax c_n\|_{L^\infty(0,T;L^2(\Omega))}
                +
                \|\partial_t c_n\|_{L^2(\Omega_T)}
                +
                \|\uvec_n\|_{L^2(0,T;W^{1,2}(\Omega))}
                \\
                &\quad +
                \artprs^{\frac{1}{\artgrwth}}
                \|\rho_n\|_{L^\infty(0,T;L^\artgrwth(\Omega))}
                +
                \sqrt{\artvisc}\|\nablax \rho_n\|_{L^2(\OmegaT)}
                \leq 
                C(\constD),
            \end{aligned}
        \end{equation}
        and
        \begin{equation}\label{energy bds(n) II}
            \begin{aligned}
                \left\| \rho_n^{\frac{\artgrwth}{2}} \right\|_{L^2(0,T;W^{1,2}(\Omega))}
                +
                \left\| \rho_n \right\|_{L^{\frac{5\artgrwth}{3}}(\Omega_T)}
                +
                \|p_{\coup,\artprs}(\rho_n)\|_{L^{\frac{5}{3}}(\Omega_T)}
                \leq 
                C(\constD,\artprs,\artvisc),
            \end{aligned}
        \end{equation}
        where $\constD>0$ denotes the constant from \eqref{init energy bd}.
    \end{lemma}
    \begin{proof}
        By definition of $\uvec_{0n}$ we have 
        \begin{equation*}
            \int_\Omega \frac{|\rho_0\uvec_{0n}|^2}{2\rho_0}\, \dd x \to \int_\Omega \frac{|(\rho\uvec)_0|^2}{2\rho_0}\, \dd x.
        \end{equation*}
        This implies
        \begin{equation*}
            E_{\mathrm{m},\delta}[\rho_0,(\rho_0\uvec_{0n}),c_0]
            \leq 
            C + C E_{\mathrm{m},\artprs}[\rho_0,(\rho\uvec)_0,c_0]
            \leq 
            C(\constD),
        \end{equation*}
        where we have used \eqref{init energy bd} to obtain the second inequality.
        Performing the same arguments leading to \eqref{cont arg I}, we have
        \begin{equation}\label{unif-in-n energy ineq}
        \begin{aligned}
            &\frac{\dd}{\dd t}E_{\mathrm{m},\artprs}[\rho_n,\rho_n\uvec_n,c_n]
            +
            \int_\Omega \frac{1}{2} \Svisc(\nablax\uvec_n):\nablax\uvec_n + \paracoup|\partial_t c_n|^2 + \frac{\artvisc\coup}{2} |\nablax \rho_n|^2 \, \dd x
            \\
            &\quad+
            \int_\Omega \artvisc\artprs\artgrwth\rho_n^{\artgrwth-2}|\nablax\rho_n|^2 \, \dd x 
            \leq 
            C + CE_{\mathrm{m},\artprs}[\rho_n,\rho_n\uvec_n,c_n] 
            \quad \text{a.e.~in } (0,T).
        \end{aligned}
        \end{equation}
        With Gr\"onwall's inequality, we conclude
        \begin{equation}\label{unif-in-n time bd energy}
            \|E_{\mathrm{m},\artprs}[\rho_n,\rho_n\uvec_n,c_n]\|_{L^\infty((0,T))}
            \leq 
            C(\constD),
        \end{equation}
        which implies in view of \eqref{bounds W,H} that
        \begin{equation}\label{unif-in-n and time bd energy}
        \begin{aligned}
            &\|\sqrt{\rho_n}\uvec_n\|_{L^\infty(0,T;L^2(\Omega))}
            +
            \|\rho_n\|_{L^\infty(0,T;L^\gamma(\Omega))}
            +
            \artprs^{\frac{1}{\artgrwth}}\|\rho_n\|_{L^\infty(0,T;L^\Gamma(\Omega))}
            \\
            &\quad +
            \|\rho_n-c_n\|_{L^\infty(0,T;L^2(\Omega))}
            +
            \|\nablax c_n\|_{L^\infty(0,T;L^2(\Omega))}
            \leq
            C(\constD).
        \end{aligned}
        \end{equation}
        Integrating \eqref{unif-in-n energy ineq} over $(0,T)$ and using \eqref{unif-in-n time bd energy} and Poincar\'e's inequality leads to
        \begin{equation}\label{diss bd energy}
            \|\uvec_n\|_{L^2(0,T;W^{1,2}(\Omega))}
            +
            \|\partial_t c_n\|_{L^2(\Omega_T)}
            +
            \sqrt{\artvisc}\|\nablax \rho_n\|_{L^2(\Omega_T)}
            \leq 
            C(\constD)
        \end{equation}
        and
        \begin{equation*}
            \left\|\rho_n^{\frac{\artgrwth}{2}}\right\|_{L^2(0,T;W^{1,2}(\Omega))}
            \leq 
            C(\constD,\artprs,\artvisc).
        \end{equation*}
        The latter estimate implies in view of the Sobolev embedding $W^{1,2}(\Omega)\hookrightarrow L^6(\Omega)$ that
        \begin{equation*}
            \left\|\rho_n^\artgrwth\right\|_{L^1(0,T;L^3(\Omega))} 
            \leq 
            C(\constD,\artprs,\artvisc),
        \end{equation*}
        and with interpolation we further deduce
        \begin{equation*}
            \|\rho_n\|_{L^{\frac{5}{3}\artgrwth}(\Omega_T)}
            \leq 
            C(\constD,\artprs,\artvisc).
        \end{equation*}
        The estimate for the last term on the left-hand side in \eqref{energy bds(n) II} follows from \eqref{grwth adm p} and the fact that $\artgrwth\geq \tgamma$.
    \end{proof}
    We obtain further uniform bounds for $\rho_n$ and $c_n$ in \Cref{lem:unif-in-n bds I} by exploiting the parabolic equation that is satisfied by $c_n$.
    \begin{lemma}\label{lem:unif-in-n bds II}
        Under the hypotheses of \Cref{lem:unif-in-n bds I} we have for any $n\in\NN$ that
        \begin{equation}\label{bds on W12 cn, L2 rhon}
            \|c_n\|_{L^\infty(0,T;W^{1,2}(\Omega))}
            +
            \|\rho_n\|_{L^\infty(0,T;L^2(\Omega))}
            +
            \|c_n\|_{L^2(0,T;W^{2,2}(\Omega))}
            \leq
            C\left( D_0 \right).
        \end{equation}
    \end{lemma}
    \begin{proof}
        Integrating the parabolic equation $\eqref{Gal approx cont+parab}_2$ over $\Omega$ yields
        \begin{equation*}
            \frac{\dd}{\dd t}\int_\Omega c_n\, \dd x 
            =
            \frac{\coup}{\paracoup}\left(\int_\Omega \rho_n\, \dd x - \int_\Omega c_n\, \dd x\right)
            =
            \frac{\coup}{\paracoup} \left(\int_\Omega \rho_0 \, \dd x - \int_\Omega c_n\, \dd x \right)
            \quad \text{a.e.~in } (0,T),
        \end{equation*}
        where we have used mass conservation which holds due to the continuity equation $\eqref{Gal approx cont+parab}_1$.
        Integrating this relation over $(0,\tau)$ for $\tau \in [0,T]$ leads to
        \begin{equation*}
            \left|\int_\Omega c_n(\tau) \, \dd x\right|
            \leq 
            C\left(\|\rho_0\|_{L^1(\Omega)} + \|c_0\|_{L^1(\Omega)} \right) + C\int_0^\tau \left|\int_\Omega c_n\, \dd x\right| \, \dd t
            \qquad \forall \tau \in [0,T].
        \end{equation*}
        With Gr\"onwall's inequality, we conclude
        \begin{equation*}
            \left\|\int_\Omega c_n \, \dd x\right\|_{L^\infty((0,T))}
            \leq 
            C\left(\|\rho_0\|_{L^1(\Omega)} + \|c_0\|_{L^1(\Omega)}\right) 
            \leq 
            C\left(\constD\right).
        \end{equation*}
        Using the Poincar\'e--Wirtinger inequality, we conclude from \eqref{energy bds(n) I} and the latter estimate first that
        \begin{equation*}
        \begin{aligned}
            \|c_n\|_{L^\infty(0,T;W^{1,2}(\Omega))}
            \leq 
            C\left\|\int_\Omega c_n\, \dd x\right\|_{L^\infty((0,T))}
            +
            C\|\nablax c_n\|_{L^\infty(0,T;L^2(\Omega))}
            \leq 
            C\left(\constD\right),
        \end{aligned}
        \end{equation*}
        and then
        \begin{equation*}
        \begin{aligned}
            \|\rho_n\|_{L^\infty(0,T;L^2(\Omega))}
            \leq 
            \|\rho_n - c_n\|_{L^\infty(0,T;L^2(\Omega))}
            +
            \|c_n\|_{L^\infty(0,T;L^2(\Omega))}
            \leq 
            C\left(\constD\right).
        \end{aligned}
        \end{equation*}
        With the latter inequality and the fact that $\rho_n$ and $c_n$ satisfy the parabolic equation $\eqref{Gal approx cont+parab}_2$, we deduce with parabolic regularity estimates (see e.g.~\cite[Theorem~11.29]{FeireislNovotny2017singlim}) that
        \begin{equation*}
            \|c_n\|_{L^2(0,T;W^{2,2}(\Omega))}
            \leq 
            C\left(\|c_0\|_{W^{1,2}(\Omega)} + \|\rho_n\|_{L^2(\Omega_T)}\right)
            \leq 
            C\left(\constD\right).
        \end{equation*}
    \end{proof}
    Finally, we derive improved uniform bounds for $\rho_n$ by exploiting $L^p$-regularity estimates for the continuity equation $\eqref{Gal approx cont+parab}_1$, which holds in a strong sense.
    \begin{lemma}\label{lem:Lp ests}
        Let the hypotheses and notations of \Cref{lem:unif-in-n bds I} hold true. 
        Then we have for any $n \in \NN$ that
        \begin{equation}\label{Lp ests I}
        \begin{aligned}
            &\|\rho_n\uvec_n\|_{L^{r(\artgrwth)}(\Omega_T)}
            +
            \artvisc\|\nablax\rho_n\|_{L^{r(\artgrwth)}(\Omega_T)}
            +
            \artvisc \| \div(\rho_n\uvec_n) \|_{L^{s(\artgrwth)}(\Omega_T)}
            \\
            &\quad +
            \artvisc
            \|\nablax\rho_n \cdot \nablax \uvec_n\|_{L^{s(\artgrwth)}(\Omega_T)}
            \leq 
            C\left(\constD,\artprs\right)
        \end{aligned}
        \end{equation}
        and
        \begin{equation}\label{Lp ests II}
            \|\partial_t \rho_n\|_{L^{s(\artgrwth)}(\Omega_T)}
            +
            \|\nablax^2 \rho_n\|_{L^{s(\artgrwth)}(\Omega_T)}
            +
            \|\rho_n\uvec_n\|_{L^{s(\artgrwth)}(0,T;W^{1,2}(\Omega))}
            \leq 
            C\left(\constD,\artprs,\artvisc,\|\rho_0\|_{W^{1,\infty}(\Omega)}\right),
        \end{equation}
        where $r(\artgrwth)$ and $s(\artgrwth)$ are defined in \eqref{values r,s}.
    \end{lemma}
    \begin{proof}
        Using interpolation between Lebesgue spaces as well as the Sobolev embedding $W^{1,2}(\Omega)\hookrightarrow L^6(\Omega)$, we conclude from \eqref{energy bds(n) I} that
        \begin{equation}\label{interpol est}
            \|\rho_n\uvec_n\|_{L^{r(\artgrwth)}(\Omega_T)}
            \leq 
            \|\rho_n\uvec_n\|_{L^{\infty}(0,T;L^{\frac{2\artgrwth}{\artgrwth+1}}(\Omega))}^{\frac{r(\artgrwth)-2}{r(\artgrwth)}}
            \|\rho_n\uvec_n\|_{L^2(0,T;L^{\frac{6\artgrwth}{\artgrwth+6}}(\Omega))}^{\frac{2}{r(\artgrwth)}}
            \leq 
            C(\constD,\artprs).
        \end{equation}
        Using $L^p$-theory for parabolic equations (see e.g.~\cite[Lemma 7.38]{NovotnyStraskraba2004}), we deduce from the fact that $\rho_n$ and $\uvec_n$ satisfy the continuity equation $\eqref{Gal approx cont+parab}_1$ in a strong sense that
        \begin{equation}\label{eps grad rhon est}
            \artvisc\|\nablax\rho_n\|_{L^{r(\Gamma)}(\Omega_T)}
            \leq 
            C\left(\artvisc^{\frac{r(\artgrwth)-1}{r(\artgrwth)}} ||\rho_0\|_{L^{r(\artgrwth)}(\Omega)} + \|\rho_n\uvec_n\|_{L^{r(\artgrwth)}(\Omega_T)} \right)
            \leq 
            C\left(\constD,\artprs\right).
        \end{equation}
        In the last estimate, we have used \eqref{interpol est} as well as $\varepsilon\leq 1$, $1<r(\artgrwth)<\artgrwth$, and \eqref{init energy bd}.
        By H\"older's inequality, we conclude from \eqref{energy bds(n) I} and \eqref{eps grad rhon est} that
        \begin{equation*}
            \artvisc\|\div(\rho_n\uvec_n)\|_{L^{s(\artgrwth)}(\Omega_T)}
            +
            \artvisc\|\nablax\rho_n\cdot\nablax\uvec_n\|_{L^{s(\artgrwth)}(\Omega_T)}
            \leq 
            C(\constD,\artprs).
        \end{equation*}
        Using again $L^p$-theory for parabolic equations (see e.g.~\cite[Lemma~7.37]{NovotnyStraskraba2004}), we conclude from the latter estimate that
        \begin{equation*}
        \begin{aligned}
            \|\partial_t \rho_n\|_{L^{s(\artgrwth)}(\Omega_T)}
            +
            \artvisc
            \|\nablax^2 \rho_n\|_{L^{s(\artgrwth)}(\Omega_T)}
            &\leq 
            C\left(\|\rho_0\|_{W^{1,\infty}(\Omega)} + \|\div(\rho_n\uvec_n)\|_{L^{s(\artgrwth)}(\Omega_T)}\right)
            \\
            &\leq 
            C\left(\constD,\artprs,\artvisc,\|\rho_0\|_{W^{1,\infty}(\Omega)}\right).
        \end{aligned}
        \end{equation*}
        The estimate for the last term on the right-hand side in \eqref{Lp ests II} follows from \eqref{energy bds(n) I} and \eqref{Lp ests I} by using the Sobolev embeddings $W^{1,2}(\Omega)\hookrightarrow L^6(\Omega)$ and $W^{1,r(\artgrwth)}(\Omega)\hookrightarrow L^\infty(\Omega)$ which holds thanks to $r(\artgrwth)>3$ (see~\eqref{values r,s}).
    \end{proof}
    \subsection{Approximation limit}
    In this subsection we perform the approximation limit $n\to\infty$ for the approximate sequence of solutions $\{(\rho_n,\uvec_n,c_n)\}_{n\in\NN}$ and complete the proof of Proposition~\ref{prop:wkSol rNSK(delta,eps)}.
    \begin{proof}[Proof of Proposition~\ref{prop:wkSol rNSK(delta,eps)}]
        Using the Sobolev embedding $W^{1,2}(\Omega)\hookrightarrow L^6(\Omega)$, H\"older's inequality, and the bounds in Lemmas~\ref{lem:unif-in-n bds I}, \ref{lem:unif-in-n bds II}, and \ref{lem:Lp ests}, we conclude with the Banach--Alaoglu theorem that, after passing to a non-relabeled subsequence, 
    \begin{equation}\label{weak cv(n)}
        \begin{aligned}
            &\rho_n\overset{\ast}{\weak} \rho\quad  \text{in } L^\infty(0,T;L^{\artgrwth}(\Omega)),
            \quad 
            \rho_n \weak \rho  \quad \text{in } L^{\frac{5}{3}\artgrwth}(\Omega_T)\cap L^{r(\artgrwth)}(0,T;W^{1,r(\artgrwth)}(\Omega)),
            \\
            &\uvec_n \weak \uvec\quad   \text{in } L^2(0,T;W^{1,2}_0(\Omega)),
            \quad 
            c_n\weak c \quad \text{in } L^2(0,T;W^{2,2}(\Omega)),
            \\
            &c_n\weakstar c \quad \text{in } L^\infty(0,T;W^{1,2}(\Omega)),
            \quad 
            \rho_n\uvec_n\weak \overline{\rho\uvec} \quad \text{in } L^2(0,T;L^{\frac{6\artgrwth}{\artgrwth+6}}(\Omega))\cap L^{r(\artgrwth)}(\OmegaT),
            \\
            &\rho_n\uvec_n \weakstar  \overline{\rho\uvec} \quad \text{in } L^\infty(0,T;L^{\frac{2\artgrwth}{\artgrwth+1}}(\Omega)),
            \quad 
            \rho_n\uvec_n\otimes\uvec_n \weak \overline{\rho\uvec\otimes\uvec} \quad \text{in } L^2(0,T;L^{\frac{3\artgrwth}{\artgrwth+3}}(\Omega)),
            \\
            &\rho_n\nablax c_n \weak \overline{\rho\nablax c} \quad \text{in } L^2(0,T;L^{\frac{6\artgrwth}{\artgrwth+6}}(\Omega)),
            \quad
            p_{\coup,\artprs}(\rho_n) \weak \overline{p_{\coup,\artprs}}\quad \text{in } L^{\frac{5}{3}}(\Omega_T),
            \\
            &\rho_n^{\frac{\artgrwth}{2}}\weak \overline{\rho^{\frac{\artgrwth}{2}}} \quad \text{in } L^2(0,T;W^{1,2}(\Omega)),
            \,\,\,\,
            \partial_t c_n \weak \partial_t c \quad \text{in } L^2(\OmegaT),
            \\
            &\partial_t \rho_n, \, \nablax^2 \rho_n, \, \nablax\rho_n\cdot\nablax \uvec_n, \, \div(\rho_n\uvec_n)
            \weak 
            \partial_t \rho, \, \nablax^2\rho, \, \overline{\nablax\rho \cdot \nablax \uvec}, \, \overline{\div(\rho\uvec)}
            \quad\text{in } L^{s(\artgrwth)}(\Omega_T),
        \end{aligned}
    \end{equation}
    with $\rho\geq0$ a.e. in $\OmegaT$.
    Here, the overlined quantities denote the weak limits of the corresponding sequences in the respective spaces.
    With classical results on Bochner spaces (see e.g.~\cite[Chapter~5]{EvansPDE2010}), we conclude from the regularity of $c$ further that $c \in C([0,T];W^{1,2}(\Omega))$.
    Using the compact Sobolev embeddings $W^{2,2}(\Omega)\hookrightarrow\hookrightarrow W^{1,2}(\Omega)$ and $W^{2,s(\artgrwth)}(\Omega)\hookrightarrow\hookrightarrow W^{1,s(\artgrwth)}(\Omega)$, we deduce from \eqref{energy bds(n) I}, \eqref{bds on W12 cn, L2 rhon}, \eqref{Lp ests I}, and \eqref{Lp ests II} by applying the Aubin--Lions theorem (see e.g.~\cite{Simon1987}) that
    \begin{equation}\label{str cv rho(n) c(n)}
        \rho_n \to \rho \quad \text{in } L^{s(\artgrwth)}(0,T;W^{1,s(\artgrwth)}(\Omega)),
        \quad 
        c_n \to c \quad \text{in } L^2(0,T;W^{1,2}(\Omega)).
    \end{equation}
    By virtue of \eqref{energy bds(n) I} and \eqref{Lp ests I}, we improve the first convergence in \eqref{str cv rho(n) c(n)} via interpolation to
    \begin{equation}\label{str cv rho(n)}
        \rho_n \to \rho \quad\text{in } L^{(\frac{5\artgrwth}{3})^-}(\Omega_T), 
        \quad 
        \nablax \rho_n \to \nablax \rho \quad \text{in } L^{r(\artgrwth)^-}(\Omega_T).
    \end{equation}
    Combining \eqref{weak cv(n)}--\eqref{str cv rho(n)} yields
    \begin{equation}\label{id limit(n) I}
    \begin{aligned}
        &\overline{\rho\uvec} = \rho\uvec,
        \quad 
        \overline{\rho\nablax c} = \rho\nablax c,
        \quad 
        \overline{\rho^{\frac{\artgrwth}{2}}} = \rho^{\frac{\artgrwth}{2}},
        \quad 
        \overline{p_{\coup,\artprs}}=p_{\coup,\artprs}(\rho),
        \\ 
        &\overline{\div(\rho\uvec)} = \div(\rho\uvec),
        \quad 
        \overline{\nablax \rho \cdot \nablax \uvec} = \nablax \rho \cdot \nablax \uvec
        \quad\text{a.e.~in } \OmegaT.
    \end{aligned}
    \end{equation}
    Due to \eqref{Gal approx cont+parab BC}, \eqref{weak cv(n)}, and ${\uvec_n}_{\mid\partial\Omega}=0$, we further have 
    \begin{equation*}
        \rho\uvec,\, \nablax \rho \in L^{s(\artgrwth)}(0,T;E^{r(\artgrwth),s(\artgrwth)}_0(\Omega)).
    \end{equation*}
    By \eqref{Gal approx mom}, we have for any $n \in \NN$ and any $\phivec\in C^\infty_c(\Omega;\RR^3)$ that
    \begin{equation}\label{mom eq P_n}
        \begin{aligned}
            &\frac{\dd}{\dd t}\int_\Omega 
            \proj (\rho_n\uvec_n) \cdot \phivec
            \, \dd x
            \\
            &=
            \int_\Omega 
            \rho_n\uvec_n\otimes \uvec_n :\nablax\proj\phivec
            +
            p_{\coup,\artprs}(\rho_n)\div\proj\phivec 
            -
            \artvisc \nablax\rho_n\cdot \nablax\uvec_n\cdot \proj\phivec
            \, \dd x
            \\
            &\qquad -\int_\Omega 
            \Svisc(\nablax\uvec_n):\nablax\proj\phivec - \coup\rho_n\nablax c_n \cdot \proj \phivec
            \, \dd x
            \qquad \text{a.e.~in } (0,T).
        \end{aligned}
    \end{equation}
    With \Cref{lem:unif-in-n bds I}, \Cref{lem:Lp ests}, and \eqref{prpties proj}, we conclude from \eqref{mom eq P_n} that $\{\partial_t \proj(\rho_n\uvec_n)\}_{n\in\NN}$ is uniformly bounded in $L^{s(\artgrwth)}(0,T;W^{-2,2}(\Omega))$, and further that $\{\proj(\rho_n\uvec_n)\}_{n\in\NN}$ is uniformly bounded in $L^{s(\artgrwth)}(0,T;W^{1,2}(\Omega))$.
    Using \eqref{weak cv(n)}, \eqref{id limit(n) I}, and the Aubin--Lions theorem, this leads to
    \begin{equation*}
        \proj(\rho_n\uvec_n) \to \rho\uvec \quad \text{in } L^{s(\artgrwth)}(0,T;L^2(\Omega)).
    \end{equation*}
    With \eqref{prpties proj} and \eqref{Lp ests II}, we deduce from the latter convergence that
    \begin{equation*}
    \begin{aligned}
        &\|\rho_n\uvec_n - \rho\uvec\|_{L^{s(\artgrwth)}(0,T;L^2(\Omega))}
        \\
        &\leq 
        \left\| \rho_n\uvec_n - \proj(\rho_n\uvec_n) \right\|_{L^{s(\artgrwth)}(0,T;L^2(\Omega))}
        +
        \left\| \proj (\rho_n\uvec_n) - \rho\uvec \right\|_{L^{s(\artgrwth)}(0,T;L^2(\Omega))}
        \to 0.
    \end{aligned}
    \end{equation*}
    In view of \eqref{Lp ests I}, this implies via interpolation that
    \begin{equation*}
        \rho_n \uvec_n \to \rho \uvec \quad \text{in } L^{r(\artgrwth)^-}(\Omega_T),
    \end{equation*}
    and, in particular, 
    \begin{equation}\label{id limit(n) II}
        \overline{\rho\uvec\otimes\uvec} = \rho\uvec\otimes\uvec \qquad \text{a.e.~in } \OmegaT.
    \end{equation}
    Next, for any $\psi\in C^\infty_c([0,T))$ and any $\phivec\in C^\infty_c(\Omega)$, we obtain by multiplying \eqref{mom eq P_n} with $\psi$ and integrating in time over $(0,T)$ that
    \begin{equation*}
        \begin{aligned}
            &-\psi(0)\int_\Omega \rho_0\uvec_{0n}\cdot \proj\phivec  \, \dd x
            -
            \int_0^T \partial_t\psi \int_\Omega
            \rho_n\uvec_n \cdot \proj\phivec \, \dd x \dd t
            \\
            &=
            \int_0^T \psi \int_\Omega 
            \rho_n\uvec_n\otimes \uvec_n :\nablax\proj\phivec
            +
            p_{\coup,\artprs}(\rho_n)\div\proj\phivec
            -
            \artvisc\nablax\rho_n\cdot\nablax\uvec_n \cdot \proj\phivec
            \, \dd x \, \dd t
            \\
            &\qquad-
            \int_0^T \psi \int_\Omega
            \Svisc(\nablax\uvec_n):\nablax\proj\phivec
            -
            \coup \rho_n\nablax c_n\cdot \proj \phivec
            \, \dd x \, \dd t.
        \end{aligned}
    \end{equation*}
    By virtue of \eqref{prpties proj}, \eqref{weak cv(n)}, \eqref{id limit(n) I}, \eqref{id limit(n) II}, and the fact that $\rho_0\uvec_{0n}\to (\rho\uvec)_0$ in $L^2(\Omega)$, we can pass to the limit $n \to \infty$ and obtain
    \begin{equation*}
        \begin{aligned}
            &-\psi(0)\int_\Omega (\rho\uvec)_0\cdot \phivec\, \dd x
            -
            \int_0^T \partial_t \psi \int_\Omega
            \rho\uvec \cdot \phivec\, \dd x \, \dd t
            \\
            &=
            \int_0^T \psi \int_\Omega 
            \rho\uvec\otimes\uvec : \nablax\phivec
            +
            p_{\coup,\artprs}(\rho) \div \phivec 
            -
            \artvisc\nablax\rho\cdot\nablax\uvec\cdot \phivec
            \, \dd x \, \dd t
            \\
            &\qquad -
            \int_0^T \psi \int_\Omega 
            \Svisc(\nablax\uvec):\nablax\phivec
            -
            \coup \rho\nablax c\cdot \phivec
            \, \dd x \, \dd t.
        \end{aligned}
    \end{equation*}
    This implies \eqref{wk sol mom(eps,delt)} and further, in view of \eqref{weak cv(n)},
    \begin{equation*}
        \rho\uvec \in \Cw([0,T];L^{\frac{2\artgrwth}{\artgrwth+1}}(\Omega)),
        \quad (\rho\uvec)(0) = (\rho\uvec)_0
        \quad \text{a.e.~in } \Omega.
    \end{equation*}
    Similarly, we obtain for any $\psi \in C^\infty_c([0,T))$ and any $\phi \in C^\infty_c(\RR^3)$ by multiplying each equation in \eqref{Gal approx cont+parab} with $\psi\phivec$, integrating over $\OmegaT$, and using \eqref{Gal approx cont+parab BC} and \eqref{Gal approx cont+parab IC} that
    \begin{equation*}
        \begin{aligned}
            &-\psi (0)\int_\Omega \rho_0 \phi\, \dx
            =
            \int_0^T \partial_t \psi \int_\Omega \rho_n\phi \, \dd x \, \dt
            +
            \int_0^T \psi \int_\Omega 
            \rho_n\uvec_n\cdot\nablax \phi 
            -
            \artvisc\nablax\rho_n\cdot\nablax\phi 
            \, \dx\, \dt,
            \\
            &-\psi(0)\int_\Omega\paracoup c_0\, \phi \, \dd x
            =
            \int_0^T \partial_t \psi \int_\Omega \paracoup c_n\phi\, \dx\, \dt
            -
            \int_0^T \psi\int_\Omega 
            \kappa\nablax c_n\cdot\nablax \phi 
            +
            \coup(c_n-\rho_n)\phi
            \, \dx \, \dt.
        \end{aligned}
    \end{equation*}
    In view of \eqref{weak cv(n)} and \eqref{id limit(n) I} we can pass to the limit $n\to\infty$ and obtain
    \begin{equation*}
        \begin{aligned}
            &-\psi (0)\int_\Omega \rho_0 \phi\, \dd x
            =
            \int_0^T \partial_t \psi \int_\Omega \rho \phi \, \dx \, \dt 
            +
            \int_0^T \psi \int_\Omega 
            \rho\uvec\cdot\nablax \phi 
            -
            \artvisc\nablax\rho\cdot\nablax\phi 
            \, \dd x\, \dd t,
            \\
            &-\psi(0)\int_\Omega\paracoup c_0\, \phi \, \dx
            =
            \int_0^T \partial_t \psi \int_\Omega \paracoup c \phi \, \dx \, \dt
            -
            \int_0^T \psi\int_\Omega 
            \kappa\nablax c\cdot\nablax \phi 
            +
            \coup(c-\rho)\phi
            \, \dx \, \dt.
        \end{aligned}
    \end{equation*}
    This implies \eqref{wk sol cont(eps,delt)}, \eqref{weak sol parab(eps,delt)}, and further, in view of \eqref{weak cv(n)} and since $c\in C([0,T];W^{1,2}(\Omega))$,
    \begin{equation*}
        \rho\in \Cw([0,T];L^{\artgrwth}(\Omega)),
        \quad 
        \rho(0)=\rho_0,
        \quad
        c(0)=c_0
        \quad \text{a.e.~in } \Omega.
    \end{equation*}
    For $\psi\in C^\infty_c([0,\infty))$ with $\psi \geq 0$, multiplying \eqref{Gal approx Energy} by $\psi$ and integrating over $(0,T)$ yields
    \begin{equation*}
        \begin{aligned}
            &-\int_0^T 
            E_{\mathrm{m},\artprs}[\rho_n,\rho_n\uvec_n,c_n]
            \partial_t \psi
            \, \dt
            \\
            &\quad +
            \int_0^T \psi \int_\Omega
            \Svisc(\nablax\uvec_n):\nablax\uvec_n
            +
            \paracoup|\partial_t c_n|^2
            +
            \artvisc \coup |\nablax \rho_n|^2
            +
            \artvisc\artprs\artgrwth\rho_n^{\artgrwth-2}\left|\nablax\rho_n\right|^2
            \, \dx\, \dt
            \\
            &\leq
            E_{\mathrm{m},\artprs}[\rho_0,\rho_0\uvec_{0n},c_n]\psi(0)
            +
            \int_0^T\psi\int_\Omega 
            q(\rho_n)\div\uvec_n 
            +
            \artvisc\coup\nablax\rho_n\cdot\nablax c_n
            \, \dx\, \dt.
        \end{aligned}
    \end{equation*}
    By virtue of \eqref{weak cv(n)}--\eqref{id limit(n) I}, \eqref{id limit(n) II}, the weak lower semi-continuity of convex functionals, and since we have by \eqref{prpties proj} that
    \begin{equation*}
        E_{\mathrm{m},\artprs}[\rho_0,\rho_0\uvec_{0n},c_0]
        \to 
        E_{\mathrm{m},\artprs}[\rho_0,(\rho\uvec)_0,c_0],
    \end{equation*}
    we obtain after passing to the limit $n\to\infty$ precisely \eqref{wkSol rNSK(eps,delt) energy}.
    The bounds in \eqref{wkSol rNSK(eps,delta) unif-bds-I} and \eqref{wkSol rNSK(eps,delta) unif-bds-II} follow from the weak convergences in \eqref{weak cv(n)} and the bounds in Lemmas~\ref{lem:unif-in-n bds I}, \ref{lem:unif-in-n bds II}, and \ref{lem:Lp ests} by using the weak lower semi-continuity of norms.
    The proof of Proposition~\ref{prop:wkSol rNSK(delta,eps)} is now complete.
    \end{proof}

%% file: VanArtVisc.tex
To prove \Cref{main result:ex wkSol}, we can now use Proposition~\ref{prop:wkSol rNSK(delta,eps)} to construct approximate solutions to the rNSKE \eqref{NSK(alpha,beta) system}--\eqref{NSK(alpha,beta) IC}.
    More specifically, let us assume that the hypotheses of \Cref{main result:ex wkSol} hold true and let $\artgrwth>\max\left\{15,\tgamma\right\}$.
    For $\artprs\in(0,1)$, we find by a density argument (see \cite[Chapter~7]{NovotnyStraskraba2004}) functions
    \begin{equation}\label{mollify I}
        \rho_{0,\artprs}\in C^\infty(\overline{\Omega}),
        \quad 
        (\rho\uvec)_{0,\artprs}\in C^\infty_c(\Omega;\RR^3),
        \quad 
        0<\artprs\leq \rho_{0,\artprs}\leq \artprs^{-\frac{1}{\artgrwth}}<\infty 
        \quad \text{on } \Omega
    \end{equation}
    with
    \begin{equation}\label{mollify II}
    \begin{aligned}
        &\rho_{0,\artprs}\to \rho_0\quad \text{in } L^{\tgamma}(\Omega),
        \quad 
        (\rho\uvec)_{0,\artprs}\to (\rho\uvec) \quad \text{in } L^{\frac{2\tgamma}{\tgamma+1}}(\Omega),
        \\
        &\frac{(\rho\uvec)_{0,\artprs}}{\sqrt{\rho_{0,\artprs}}} \to \frac{(\rho\uvec)_0}{\sqrt{\rho_0}} \quad \text{in } L^2(\Omega),
        \quad 
        \artprs^{\frac{1}{\artgrwth}}\rho_{0,\artprs} \to 0  \quad \text{in } L^{\artgrwth}(\Omega),
    \end{aligned}
    \end{equation}
    and
    \begin{equation}\label{mollify III}
        \|\rho_{0,\artprs}\|_{L^{\tgamma}(\Omega)}
        +
        \bigg\|\frac{|(\rho\uvec)_{0,\artprs}|^2}{\rho_{0,\artprs}}\bigg\|_{L^1(\Omega)}
        +
        \artprs^{\frac{1}{\artgrwth}}\|\rho_{0,\artprs}\|_{L^{\artgrwth}(\Omega)}
        \leq 
        \Cdata,
    \end{equation}
    where $\data$ is defined in \eqref{data}.
    Then, by Proposition~\ref{prop:wkSol rNSK(delta,eps)} and \eqref{mollify III}, there exists for $\artprs,\artvisc\in (0,1)$ a triplet $(\rhoepsdelt,\uvecepsdelt,\cepsdelt)$ that satisfies \eqref{reg wksol rNSK(delta,eps)}--\eqref{wkSol rNSK(eps,delta) unif-bds-II} with $\rho_0$, $(\rho\uvec)_0$, and $\constD$ replaced by  $\rho_{0,\artprs}$, $(\rho\uvec)_{0,\artprs}$, and $\data$, respectively.
    To prove \Cref{main result:ex wkSol}, we wish to perform the vanishing artificial viscosity limit $\artvisc\to 0$ and afterwards the vanishing artificial pressure limit $\artprs\to 0$.
    In this section, we perform the artificial viscosity limit and prove as an outcome the following statement:
    \begin{proposition}\label{prop:wkSol rNSK(delta)}
        Let the hypotheses and notations of \Cref{main result:ex wkSol} hold true and suppose that $\artgrwth>\max\left\{15,\tgamma\right\}$.
        Assume for $\artprs\in (0,1)$ that regularized initial conditions $\rho_{0,\artprs}$ and $(\rho\uvec)_{0,\artprs}$ satisfying \eqref{mollify I}--\eqref{mollify III} are given.
        Then there exists a triplet $(\rhod,\uvecd,\cd)$ such that
        \begin{enumerate}
            \item we have the regularity
            \begin{equation}\label{wk sol rNSK(delt) regularity}
                \begin{aligned}
                    &\rhod \in \Cw([0,T];L^\artgrwth(\Omega))\cap L^{\artgrwth+1}(\OmegaT),
                    \quad 
                    \rhod \geq 0 \quad \text{a.e.},
                    \\
                    &\uvecd \in L^2(0,T;W^{1,2}_0(\Omega;\RR^3)),
                    \quad 
                    \rhod\uvecd\in \Cw([0,T];L^{\frac{2\artgrwth}{\artgrwth+1}}(\Omega;\RR^3)),
                    \\
                    &\cd \in C([0,T];W^{1,2}(\Omega))\cap L^2(0,T;W^{2,2}(\Omega)),
                    \quad 
                    \partial_t \cd \in L^2(\OmegaT);
                \end{aligned}
            \end{equation}
            \item the rNSKE \eqref{NSK(alpha,beta) system}--\eqref{NSK(alpha,beta) IC} with the artificial pressure $p_{\coup,\artprs}$ (see \eqref{def p_coup,delta}) and regularized initial conditions $(\rho_{0,\artprs},(\rho\uvec)_{0,\artprs},c_0)$ are satisfied in the weak sense, that is,
            \begin{align}
            &\int_0^T \int_\Omega 
            \rhod 
            \partial_t \varphi
            +
            \rhod\uvecd \cdot \nablax \varphi
            \, \dd x \, \dd t
            =
            0,
            \quad \forall\, \varphi \in C^\infty_c(\RR_T^3),\label{wk sol rNSK cont(delt)}
            \\
            &\int_0^T\int_\Omega 
            \rhod\uvecd \cdot \partial_t\varphivec
            +
            \rhod\uvecd\otimes\uvecd:\nablax\varphivec
            +
            p_{\coup,\artprs}(\rhod) \div\varphivec
            \, \dd x \, \dd t\nonumber
            \\
            &\qquad=
            \int_0^T \int_\Omega 
            \Svisc(\nablax\uvecd):\nablax\varphivec -\coup \rhod \nablax \cd \cdot \varphivec
            \, \dd x \, \dd t 
            \quad \forall \, \varphivec\in C^\infty_c(\OmegaT;\RR^3),\label{wk sol rNSK mom(delt)}
            \\
            &\int_0^T \int_\Omega 
            \paracoup \cd\partial_t \varphi 
            -
            \kappa \nablax \cd\cdot \nablax \varphi 
            -
            \coup(\cd-\rhod) \varphi
            \, \dd x \, \dd t
            = 0  \quad\forall\, \varphi \in C^\infty_c(\RR_T^3),\label{wk sol rNSK parab(delt)}
            \end{align}
            and
            \begin{equation}\label{wk sol rNSK IC(delt)}
                \rhod(0)=\rho_{0,\artprs},
                \quad 
                \rhod\uvecd(0)=(\rho\uvec)_{0,\artprs},
                \quad 
                \cd(0)=c_{0} \quad \text{a.e.~in } \Omega;
            \end{equation}
            \item we have the bounds
            \begin{equation}\label{wk sol rNSK bounds(delt)}
                \begin{aligned}
                    &\left\| \sqrt{\rhod}\uvecd \right\|_{L^\infty(0,T;L^2(\Omega))}
                    +
                    \left\| \rhod \right\|_{L^\infty(0,T;L^\tgamma(\Omega))}
                    +
                    \left\| \cd \right\|_{L^\infty(0,T;W^{1,2}(\Omega))}
                    \\
                    &\quad 
                    +
                    \left\| \cd \right\|_{L^2(0,T;W^{2,2}(\Omega))}
                    +
                    \left\| \partial_t \cd \right\|_{L^2(\OmegaT)}
                    +
                    \artprs^{\frac{1}{\artgrwth}}\|\rhod\|_{L^\infty(0,T;L^\artgrwth(\Omega))}
                    \leq \Cdata,
                \end{aligned}
            \end{equation}
            where $\data$ is defined in $\eqref{data}$;
            \item the energy inequality corresponding to the rNSKE \eqref{NSK(alpha,beta) system}--\eqref{NSK(alpha,beta) IC} with artificial pressure $p_{\coup,\artprs}$ holds in the weak sense, that is,
            \begin{equation}\label{wk sol rNSK energy(delt)}
                \begin{aligned}
                    &-\int_0^T  E_\artprs[\rhod,\rhod\uvecd,\cd]\partial_t \psi \, \dt 
                    +
                    \int_0^T \psi \int_\Omega 
                    \Svisc(\nablax\uvecd):\nablax\uvecd + \paracoup|\partial_t \cd|^2
                    \, \dx \, \dt 
                    \\
                    &\leq 
                    E_{\artprs}[\rho_{0,\artprs},(\rho\uvec)_{0,\artprs},c_0]\psi(0)
                \end{aligned}
            \end{equation}
            for any $\psi \in C_c^\infty([0,T))$ with $\psi\geq 0$, where
            \begin{equation*}
                E_{\artprs}[\rhod,\rhod\uvecd,\cd]
                :=
                \int_\Omega 
                \frac{|\rhod\uvecd|^2}{2\rhod}
                +
                W(\rhod)
                +
                \frac{\artprs}{\artgrwth-1}\rhod^{\artgrwth}
                +
                \frac{\coup}{2} |\rhod-\cd|^2
                +
                \frac{\kappa}{2}|\nablax \cd|^2 
                \, \dx 
            \end{equation*}
            with $W$ denoting the pressure potential corresponding to $p$ defined in \eqref{def press pot H Q W}.
        \end{enumerate}
    \end{proposition}
    \subsection{Uniform estimates and weak limit}
    To perform for fixed $\artprs>0$ the vanishing viscosity limit $\artvisc\to 0$ for the approximate sequence of solutions $\{(\rhoepsdelt,\uvecepsdelt,\cepsdelt)\}_{\artvisc >0}$, we need bounds that are uniform with respect to $\artvisc$.
    The bounds in \eqref{wkSol rNSK(eps,delta) unif-bds-I} and \eqref{wkSol rNSK(eps,delta) unif-bds-II} do not provide equi-integrability for the sequence $\{p_{\coup,\artprs}(\rhoepsdelt)\}_{\artvisc>0}$.
    Thus, we have to obtain some improved uniform bound for the density in order to pass to the limit $\artvisc\to 0$ in the pressure term. 
    To obtain such a bound, we adapt the technique that is known from the compressible NSE (see e.g.~\cite{FeireislPetzeltova2000,FeireislNovotnyPetzeltova2001}) and test the momentum equation by a suitable multiplier that is constructed by using the Bogovski\u{\i} operator.
    For the sake of completeness we recall the relevant properties of this operator proven in \cite[Theorem~11.17]{FeireislNovotny2017singlim}.
    \begin{lemma}\label{lem:Bog}
        Let $\Omega \subseteq \RR^3$ be a bounded domain with $\partial\Omega \in C^{0,1}$ and let $s \in (1,\infty)$.
        Then there exists a bounded linear operator $\Bog \colon L^s_0(\Omega) \to W^{1,s}_0(\Omega;\RR^3)$ such that for any $f \in L^s_0(\Omega)$,
        \begin{equation*}
            \div\Bog(f) = f \quad \text{a.e.~in } \Omega,
            \quad 
            \big\|\Bog(f)\big\|_{W^{1,s}_0(\Omega)} \leq C\|f\|_{L^s(\Omega)}.
        \end{equation*}
        Moreover, for any $s \in (1,\infty)$, this operator can be extended to a linear operator 
        \begin{equation*}
            \Bog\colon \big[\dot{W}^{1,s^\prime}(\Omega)\big]^\ast 
            := \big\{ f \in \big[W^{1,s^\prime}(\Omega)\big]^\ast \mid \langle f , 1 \rangle = 0 \big\} \to L^s(\Omega;\RR^3)
        \end{equation*}
        such that for any $f \in \big[\dot{W}^{1,s^\prime}(\Omega)\big]^\ast$,
        \begin{equation*}
            -\int_\Omega \Bog(f) \cdot \nablax \phi \, \dd x 
            =
            \langle f , \phi \rangle \quad \forall \, \phi \in W^{1,s^\prime}(\Omega),
            \quad 
            \|\Bog(f)\|_{L^s(\Omega)} \leq C \|f\|_{[W^{1,s^\prime}(\Omega)]^\ast}.
        \end{equation*}
    \end{lemma}
    With the operator $\Bog$ at hand, we deduce an improved estimate for the density:
    \begin{lemma}\label{lem:impr bd(delt)}
        Let the hypotheses and notations of Proposition~\ref{prop:wkSol rNSK(delta)} hold true.
        For $\artprs,\artvisc\in (0,1)$, let $(\rhoepsdelt,\uvecepsdelt,\cepsdelt)$ be a triplet satisfying \eqref{reg wksol rNSK(delta,eps)}--\eqref{wkSol rNSK(eps,delta) unif-bds-II} with $\rho_0$, $(\rho\uvec)_0$, and $\constD$ replaced by $\rho_{0,\artprs}$, $(\rho\uvec)_{0,\artprs}$, and $\data$, respectively. Then we have
        \begin{equation}\label{lem:impr bd(delt) I}
            \int_0^T \int_\Omega p_{\coup,\artprs}(\rhoepsdelt)\rhoepsdelt\, \dx \, \dt 
            \leq 
            C\left(\data,\artprs\right).
        \end{equation}
    \end{lemma}
    \begin{proof}
        We define
        \begin{equation*}
            \varphivec_{\artprs,\artvisc}
            :=
            \Bog\left(\rho_{\artprs,\artvisc} - \fint_\Omega \rho_{\artprs,\artvisc} \right)
            =
            \Bog \left(\rho_{\artprs,\artvisc} - \fint_\Omega \rho_{0,\artprs} \right),
        \end{equation*}
        and deduce from the continuity equation \eqref{wk sol cont(eps,delt)} that
        \begin{equation*}
            \partial_t \varphivec_{\artprs,\artvisc}
            =
            \Bog\bigl(\div(\artvisc\nablax\rhoepsdelt - \rhoepsdelt\uvecepsdelt)\bigr)
            \quad \text{a.e.~in } \OmegaT.
        \end{equation*}
        By virtue of \eqref{reg wksol rNSK(delta,eps)}, \eqref{wkSol rNSK(eps,delta) unif-bds-I}, \eqref{wkSol rNSK(eps,delta) unif-bds-II}, and \Cref{lem:Bog}, we have
        \begin{equation}\label{est phi_epsdelt & grad}
        \begin{aligned}
            \left\|\varphivec_{\artprs,\artvisc}\right\|_{L^\infty(0,T;L^\infty(\Omega))}
            +
            \|\nablax \varphivec_{\artprs,\artvisc}\|_{L^\infty(0,T;L^\artgrwth(\Omega))}
            &\leq 
            C
            \left(
            \|\rhoepsdelt\|_{L^\infty(0,T;L^{\artgrwth}(\Omega))}
            +
            \|\rho_{0,\artprs}\|_{L^1(\Omega)}
            \right)
            \\
            &\leq 
            C\left(\data,\artprs\right),
        \end{aligned}
        \end{equation}
        and
        \begin{equation}\label{est partial_t phi_epsdelt}
                \left\| \partial_t \varphivec_{\artprs,\artvisc} \right\|_{L^{r(\artgrwth)}(\OmegaT)}
                \leq 
                C\left( \artvisc\|\nablax \rhoepsdelt\|_{L^{r(\artgrwth)}(\OmegaT)} + \|\rhoepsdelt\uvecepsdelt\|_{L^{r(\artgrwth)}(\OmegaT)}\right)
                \leq 
                C\left(\data,\artprs\right).
        \end{equation}
        We fix $\psi \in C^\infty_c((0,T))$.
        By a density argument, we have that $\psi\varphivec_{\artprs,\artvisc}$ is an admissible test function for the momentum equation \eqref{wk sol mom(eps,delt)} such that we obtain
        \begin{equation*}
            \begin{aligned}
                &\int_0^T \psi\int_\Omega  p_{\coup,\artprs}(\rhoepsdelt)\rhoepsdelt \, \dx \, \dt 
                \\
                &=
                \int_0^T \psi\int_\Omega p_{\coup,\artprs}(\rhoepsdelt)\left(\fint_\Omega \rho_{0,\artprs} \right) \, \dx \, \dt 
                -
                \int_0^T \partial_t\psi\int_\Omega \rhoepsdelt\uvecepsdelt\cdot \varphivec_{\artprs,\artvisc}\, \dx \, \dt \\
                &\quad 
                -
                \int_0^T \psi \int_\Omega \rhoepsdelt\uvecepsdelt\cdot \partial_t \varphivec_{\artprs,\artvisc} \, \dx \, \dt 
                -
                \int_0^T \psi \int_\Omega \rhoepsdelt\uvecepsdelt\otimes\uvecepsdelt : \nablax \varphivec_{\artprs,\artvisc} \, \dx \, \dt 
                \\
                &\quad 
                +
                \int_0^T \psi \int_\Omega \artvisc\nablax\rhoepsdelt\cdot\nablax\uvecepsdelt \cdot \varphivec_{\artprs,\artvisc} \, \dx  \, \dt 
                +
                \int_0^T \psi \int_\Omega \Svisc(\nablax \uvecepsdelt) : \nablax \varphivec_{\artprs,\artvisc} \, \dx \, \dt 
                \\
                &\quad 
                - 
                \int_0^T \psi \int_\Omega \rhoepsdelt\nablax \cepsdelt\cdot \varphivec_{\artprs,\artvisc}\, \dx  \, \dt
                =:\sum\limits_{i=1}^{7} I_{i}^{\artprs,\artvisc}.
            \end{aligned}
        \end{equation*}
        With \eqref{wkSol rNSK(eps,delta) unif-bds-I}, \eqref{est phi_epsdelt & grad}, and \eqref{est partial_t phi_epsdelt}, we estimate by using H\"older's inequality
        \begin{equation*}
            \sum\limits_{i=1}^{6} \left|I_{i}^{\artprs,\artvisc}\right|
            \leq 
            C\left(\data,\artprs\right)
            \left(\|\psi\|_{L^\infty((0,T))} + \|\partial_t \psi\|_{L^1((0,T))}\right).
        \end{equation*}
        Since these estimates are well-known from the theory on the compressible NSE, we omit the details here and refer to \cite[Chapter~3]{FeireislNovotnyPetzeltova2001} for a detailed exposition.
        Similarly, we estimate the term $I_7^{\artprs,\artvisc}$ as
        \begin{equation*}
            \begin{aligned}
                \left|I_7^{\artprs,\artvisc}\right|
                \leq 
                C\|\psi\|_{L^\infty((0,T))} \|\rhoepsdelt\nablax \cepsdelt\|_{L^1(\OmegaT)}\|\varphivec_{\artprs,\artvisc}\|_{L^\infty(0,T;L^\infty(\Omega))}
                \leq 
                C\left(\data,\artprs\right) \|\psi\|_{L^\infty((0,T))}.
            \end{aligned}
        \end{equation*}
        We have shown that
        \begin{equation*}
            \int_0^T \psi \int_\Omega p_{\coup,\artprs}(\rhoepsdelt)\rhoepsdelt\, \dx \, \dt 
            \leq 
            C\left(\data,\artprs\right)
            \left(\|\psi\|_{L^\infty((0,T))} + \|\partial_t \psi\|_{L^1((0,T))}\right).
        \end{equation*}
        From this relation, we conclude \eqref{lem:impr bd(delt) I} by approximating the identity function $\mathrm{id}_{(0,T)}$ on $(0,T)$ with suitable test functions $\psi \in C^\infty_c((0,T))$. 
    \end{proof}
    With \Cref{lem:impr bd(delt)} and the uniform bounds in \eqref{wkSol rNSK(eps,delta) unif-bds-I}, we identify weak limits that allow us, after passing to a non-relabeled subsequence, to perform the limit $\artvisc\to 0$ in the equations \eqref{wk sol cont(eps,delt)}--\eqref{weak sol parab(eps,delt)}.
    \begin{lemma}\label{lem:wk cv(eps,delt)}
        Let the hypothesis and notation of \Cref{lem:impr bd(delt)} hold true.
        Then we have 
        \begin{equation}\label{eps grad rhoepsdelt cv}
            \artvisc\nablax\rhoepsdelt
            \to 0 \quad \text{in  } L^2(0,T;L^{\frac{2\artgrwth}{\artgrwth-1}}(\Omega)),
            \quad 
            \artvisc\nablax \rhoepsdelt\cdot\nablax\uvecepsdelt \weak 0 \quad \text{in } L^{s(\artgrwth)}(\OmegaT).
        \end{equation}
        Moreover, after passing to a non-relabeled subsequence,
        \begin{equation}\label{lem:wk cv(eps,delt) convergences}
            \begin{aligned}
                &\rhoepsdelt\to \rhod \quad \text{in } \Cw([0,T];L^{\artgrwth}(\Omega)),
                \quad
                \rhoepsdelt\weak \rhod \quad \text{in } L^{\artgrwth+1}(\OmegaT),
                \\
                &\uvecepsdelt\weak \uvecd \quad \text{in } L^2(0,T;W^{1,2}_0(\Omega)),
                \quad
                \cepsdelt\weakstar \cd \quad \text{in } L^\infty(0,T;W^{1,2}(\Omega)),
                \\
                &\cepsdelt\weak \cd \quad \text{in } L^2(0,T;W^{2,2}(\Omega)),
                \quad
                \cepsdelt\to \cd \quad \text{in }C([0,T];L^2(\Omega))\cap L^2(0,T;W^{1,2}(\Omega))
                \\
                &\rhoepsdelt\uvecepsdelt \to \rhod\uvecd \quad \text{in } \Cw([0,T];L^{\frac{2\artgrwth}{\artgrwth+1}}(\Omega)),
                \quad 
                \partial_t \cepsdelt\weak \partial_t c \quad \text{in } L^2(\OmegaT),
                \\
                &\rhoepsdelt\nablax \cepsdelt \weak \rhod\nablax \cd \quad \text{in } L^2(0,T;L^{\frac{6\artgrwth}{\artgrwth+6}}(\Omega)),
                \quad
                p_{\coup,\artprs}(\rhoepsdelt)\weak \overline{p_{\coup,\artprs}} \quad \text{in } L^{\frac{\artgrwth+1}{\artgrwth}}(\OmegaT),
            \end{aligned}
        \end{equation}
        with $\rho \geq 0$ a.e.~in $\OmegaT$.
        Furthermore, the triplet $(\rhod,\uvecd,\cd)$ satisfies \eqref{wk sol rNSK(delt) regularity}--\eqref{wk sol rNSK IC(delt)} with $p_{\coup,\artprs}(\rhod)$ replaced by $\overline{p_{\coup,\artprs}}$.
    \end{lemma}
    \begin{proof}
        Since $r(\artgrwth)>\frac{2\artgrwth}{\artgrwth-1}>2$, we have by interpolation for some $\theta \in (0,1)$ that
        \begin{equation*}
            \artvisc\|\nablax \rhoepsdelt\|_{L^2(0,T;L^{\frac{2\artgrwth}{\artgrwth-1}}(\Omega))}
            \leq 
            \left(\artvisc\|\nablax \rhoepsdelt\|_{L^2(\OmegaT)}\right)^{\theta}
            \left(\artvisc\|\nablax\rhoepsdelt\|_{L^2(0,T;L^{r(\artgrwth)}(\Omega))}\right)^{1-\theta}
        \end{equation*}
        and the first convergence in \eqref{eps grad rhoepsdelt cv} follows from \eqref{wkSol rNSK(eps,delta) unif-bds-I} and \eqref{wkSol rNSK(eps,delta) unif-bds-II}.
        From \eqref{wkSol rNSK(eps,delta) unif-bds-I} we further conclude
        \begin{equation*}
            \artvisc\nablax\rhoepsdelt\cdot\nablax\uvecepsdelt \to 0 \quad  \text{in } \mathcal{D}^\prime(\OmegaT).
        \end{equation*}
        Taking into account \eqref{wkSol rNSK(eps,delta) unif-bds-II}, this implies the second convergence in \eqref{eps grad rhoepsdelt cv}.
        From \eqref{wkSol rNSK(eps,delta) unif-bds-I} and \eqref{wkSol rNSK(eps,delta) unif-bds-II} we conclude with the Sobolev embedding $W^{1,2}(\Omega)\hookrightarrow L^6(\Omega)$, H\"older's inequality, and the fact that the triplet $(\rhoepsdelt,\uvecepsdelt,\cepsdelt)$ satisfies \eqref{wk sol cont(eps,delt)} and \eqref{wk sol mom(eps,delt)}, that
        \begin{equation*}
            \|\partial_t \rhoepsdelt\|_{L^s(0,T;W^{-1,s}(\Omega))}
            +
            \|\partial_t (\rhoepsdelt\uvecepsdelt)\|_{L^s(0,T;W^{-1,s}(\Omega))}
            \leq 
            C(\constD)
        \end{equation*}
        for some $s \in (1,\infty)$.
        In combination with \eqref{wkSol rNSK(eps,delta) unif-bds-I}, this implies after passing to a non-relabeled subsequence that
        \begin{equation}\label{Cw(eps,delt)}
            \rhoepsdelt\to \rhod \quad \text{in } \Cw([0,T];L^\artgrwth(\Omega)),
            \quad 
            \rhoepsdelt\uvecepsdelt\to \overline{\rhod\uvecd} \quad \text{in } \Cw([0,T];L^{\frac{2\artgrwth}{\artgrwth+1}}(\Omega)).
        \end{equation}
        Using the Banach--Alaoglu theorem as well as the Sobolev embedding $W^{1,2}(\Omega)\hookrightarrow L^6(\Omega)$ and H\"older's inequality, we conclude from the uniform bounds in \eqref{wkSol rNSK(eps,delta) unif-bds-I} and \Cref{lem:impr bd(delt)} that, after passing to a non-relabeled subsequence,
        \begin{equation}\label{weak cv(eps,delt)}
            \begin{aligned}
                &\rhoepsdelt\weak \rhod \quad \text{in } L^{\artgrwth+1}(\OmegaT),
                \quad
                \uvecepsdelt\weak \uvecd \quad \text{in } L^2(0,T;W^{1,2}_0(\Omega)),
                \\
                &\cepsdelt\weakstar \cd \quad \text{in } L^\infty(0,T;W^{1,2}(\Omega)),
                \quad
                \cepsdelt\weak \cd \quad \text{in } L^2(0,T;W^{2,2}(\Omega)),
                \quad
                \\ 
                &p_{\coup,\artprs}(\rhoepsdelt)\weak \overline{p_{\coup,\artprs}} \quad \text{in } L^{\frac{\artgrwth+1}{\artgrwth}}(\OmegaT),
                \quad
                \rhoepsdelt\nablax\cepsdelt \weak \overline{\rhod\nablax \cd} \quad \text{in } L^2(0,T;L^{\frac{6\artgrwth}{\artgrwth+6}}(\Omega)),
                \\
                &\rhoepsdelt\uvecepsdelt\otimes\uvecepsdelt\weak \overline{\rhod\uvecd\otimes\uvecd} \quad \text{in } L^2(0,T;L^{\frac{3\artgrwth}{\artgrwth+3}}(\Omega)),
                \quad 
                \partial_t \cd \weak \partial_t c \quad \text{in } L^2(\OmegaT),
            \end{aligned}
        \end{equation}
        with $\rho \geq 0$ a.e.~in $\OmegaT$.
        Here, the overlined quantities denote weak limits of the corresponding sequences in their respective spaces.
        From the regularity of $\cd$ and $\partial_t \cd$, we conclude by standard results on Bochner spaces (see e.g.~\cite[Chapter~5]{EvansPDE2010}) that $\cd \in C([0,T];W^{1,2}(\Omega))$.
        In particular, we have shown that the triplet $(\rhod,\uvecd,\cd)$ satisfies the regularity \eqref{wk sol rNSK(delt) regularity}.
        Combining \eqref{Cw(eps,delt)} with the second convergence in \eqref{weak cv(eps,delt)} and using the compact Sobolev embeddings $L^\artgrwth(\Omega) \hookrightarrow\hookrightarrow W^{-1,2}(\Omega)$ and $L^{\frac{2\artgrwth}{\artgrwth+1}}(\Omega)\hookrightarrow\hookrightarrow W^{-1,2}(\Omega)$ yields
        \begin{equation}\label{id I lim (delt)}
            \overline{\rhod\uvecd} = \rhod\uvecd,
            \quad 
            \overline{\rhod\uvecd\otimes\uvecd} = \rhod\uvecd\otimes\uvecd 
            \quad \text{a.e.~in } \OmegaT.
        \end{equation}
        Moreover, by virtue of the uniform bounds in \eqref{wkSol rNSK(eps,delta) unif-bds-I} as well as the compact Sobolev embedding $W^{1,2}(\Omega)\hookrightarrow\hookrightarrow L^2(\Omega)$, we conclude with the Aubin--Lions theorem that
        \begin{equation}\label{C(eps,delt)}
            \cepsdelt\to \cd \quad \text{in } C([0,T];L^2(\Omega))\cap L^2(0,T;W^{1,2}(\Omega)).
        \end{equation}
        Together with the first convergence in \eqref{weak cv(eps,delt)}, this yields
        \begin{equation}\label{id II lim(delt)}
            \overline{\rhod\nablax \cd} = \rhod\nablax \cd 
            \quad \text{a.e.~in } \OmegaT.
        \end{equation}
        With \eqref{Cw(eps,delt)}--\eqref{id II lim(delt)} we have shown that the convergences \eqref{lem:wk cv(eps,delt) convergences} hold true.
        In particular, we can pass to the limit $\artvisc\to 0$ in \eqref{wk sol cont(eps,delt)}--\eqref{weak sol parab(eps,delt)} and obtain that the triplet $(\rhod,\uvecd,\cd)$ satisfies precisely \eqref{wk sol rNSK cont(delt)}--\eqref{wk sol rNSK parab(delt)} with $p_{\coup,\artprs}(\rhod)$ replaced by $\overline{p_{\coup,\artprs}}$.
        From the convergences \eqref{Cw(eps,delt)} and \eqref{C(eps,delt)}, we conclude that the triplet $(\rhod,\uvecd,\cd)$ satisfies the initial conditions \eqref{wk sol rNSK IC(delt)}.
    \end{proof}

    
    \subsection{Strong convergence of the density}\label{subsec:strong convergence of the density}
    To prove Proposition~\ref{prop:wkSol rNSK(delta)}, we wish to verify the strong convergence $\rhoepsdelt\to \rhod$ in $L^1(\OmegaT)$, which yields in particular the relation $\overline{p_{\coup,\artprs}}=p_{\coup,\artprs}(\rhod)$.
    To obtain this convergence, we adapt the approach that is known from the theory on the compressible NSE (see e.g.~\cite{Lions1998,FeireislNovotnyPetzeltova2001,Feireisl2002}) to the rNSKE.
    More specifically, we interpret the additional term $\rhod\nablax \cd$ in the momentum equation as a force term and use the general version of the effective viscous flux lemma (see \Cref{prop:EVF(general)}) to derive a compactness relation for the effective viscous flux.
    Then we exploit this relation as well as the decomposition for the pressure function \eqref{dec prs} to conclude the strong convergence of the density by using techniques from \cite{Feireisl2002}.
    \begin{lemma}\label{lem:str cv density(eps,delt)}
        Let the hypotheses and notations of Lemma~\ref{lem:wk cv(eps,delt)} hold true. Then we have
        \begin{equation*}
            \begin{aligned}
                \rho_{\artprs,\artvisc} \to \rhod \quad \text{in } L^{(\artgrwth+1)^-}(\OmegaT),
            \end{aligned}
        \end{equation*}
        and, in particular,
        \begin{equation*}
            \overline{p_{\coup,\artprs}} = p_{\coup,\artprs}(\rhod) \quad \text{a.e.~in } \OmegaT.
        \end{equation*}
    \end{lemma}
    \begin{proof}
        We set $h_{\coup,\artprs}(r):= h(r)+\frac{\coup}{2}r^2 + \artprs r^\Gamma$.
        With the Banach--Alaoglu theorem we conclude from the uniform bounds in \eqref{wkSol rNSK(eps,delta) unif-bds-I} and in \Cref{lem:impr bd(delt)} that there exist functions $\overline{h_{\coup,\artprs}}\in L^{\frac{\artgrwth+1}{\artgrwth}}(\OmegaT)$, $\overline{q_\artprs}\in L^\infty(0,T;L^\infty(\Omega))$, $\overline{\rhod q_\artprs}\in L^{\artgrwth+1}(\OmegaT)$, $\overline{\rhod \div\uvecd} \in L^2(0,T;L^{\frac{2\artgrwth}{\artgrwth+2}}(\Omega))$ with $\overline{p_{\coup,\artprs}}=\overline{h_{\coup,\artprs}}+\overline{q_{\artprs}}$ and $\overline{q_\artprs}\leq 0$, such that, after passing to a non-relabeled subsequence,
        \begin{equation}\label{EVF(eps,delt) I}
            \begin{aligned}
                &h_{\coup,\artprs}(\rhoepsdelt)\weak \overline{h_{\coup,\artprs}} \quad \text{in } L^{\frac{\artgrwth+1}{\artgrwth}}(\OmegaT),
                \quad 
                q(\rhoepsdelt)\weakstar \overline{q_\artprs} \quad \text{in } L^\infty(0,T;L^\infty(\Omega)),
                \\
                &\rhoepsdelt q(\rhoepsdelt) \weak \overline{\rhod q_\artprs} \quad \text{in } L^{\artgrwth+1}(\OmegaT),
                \quad
                \rhoepsdelt\div\uvecepsdelt \weak \overline{\rhod\div\uvecd} \quad \text{in } L^2(0,T;L^{\frac{2\artgrwth}{\artgrwth+2}}(\Omega)).
            \end{aligned}
        \end{equation}
        Let now $\xi \in C^\infty_c(\Omega)$. 
        By virtue of \eqref{eps grad rhoepsdelt cv}, we have that
        \begin{equation}\label{weak cv deltax rhoepsdelt}
            \xi \artvisc\Deltax \rhoepsdelt \weak 0 \quad \text{in } L^2(0,T;W^{-1,2}(\Omega)).
        \end{equation}
        Moreover, we have for any $\varphi \in C^\infty_c(\OmegaT)$ that
        \begin{equation*}
            \begin{aligned}
            &\left|
                \int_0^T \int_\Omega 
                \varphi \nablax\Deltax^{-1}\left[\xi \artvisc \Deltax \rhoepsdelt\right]
                \, \dx  \, \dt
            \right|
            \\
            &\leq 
            \left|
            \int_0^T \int_\Omega 
            \artvisc \xi \,\nablax\otimes\nablax \Deltax^{-1}[\varphi] \cdot \nablax \rhoepsdelt 
            \, \dx  \, \dt 
            \right|
            +
            \left|
            \int_0^T \int_\Omega \artvisc \nablax \Deltax^{-1}[\varphi] \nablax \xi \cdot \nablax \rhoepsdelt \, \dx  \, \dt
            \right|
            \\
            &\leq 
            C \|\xi\|_{W^{1,\infty}(\Omega)} 
            \|\varphi \|_{L^2(0,T;L^{\frac{2\artgrwth}{\artgrwth+1}}(\Omega))} \artvisc\|\nablax\rhoepsdelt\|_{L^2(0,T;L^{\frac{2\artgrwth}{\artgrwth-1}}(\Omega))},
            \end{aligned}
        \end{equation*}
        where $\Deltax^{-1}$ denotes the inverse Laplace operator on $\RR^3$, and where we have used the fact that $\nablax\otimes\nablax\Deltax^{-1}\colon L^s(\Omega) \to L^s(\Omega;\RR^{3\times 3})$ and $\nablax \Deltax^{-1}\colon L^s(\Omega) \to W^{1,s}(\Omega;\RR^3)$ are continuous linear operators for any $s\in (1,\infty)$, see e.g.~\cite[Theorem~9]{FeireislNovotny2017singlim}.
        With \eqref{eps grad rhoepsdelt cv} we conclude by duality
        \begin{equation}\label{strong cv deltax rhoepsdelt}
            \nablax \Deltax^{-1}\big[\xi \artvisc\Deltax\rhoepsdelt\big] \to 0 \quad \text{in } L^2(0,T;L^{\frac{2\artgrwth}{\artgrwth-1}}(\Omega)).
        \end{equation}
        By virtue of \eqref{eps grad rhoepsdelt cv}, \eqref{lem:wk cv(eps,delt) convergences}, \eqref{weak cv deltax rhoepsdelt}, \eqref{strong cv deltax rhoepsdelt}, and the fact that the triplet $(\rhoepsdelt,\uvecepsdelt,\cepsdelt)$ satisfies \eqref{wk sol cont(eps,delt)} and \eqref{wk sol mom(eps,delt)}, we may apply \Cref{prop:EVF(general)} to conclude that the relation
        \begin{equation*}
        \begin{aligned}
            &\lim\limits_{\artvisc\to 0}
            \int_0^T\psi \int_\Omega \phi 
            \rhoepsdelt\left(p_{\coup,\artprs}(\rhoepsdelt) - \left(\bulkvisc+\frac{4\shearvisc}{3}\right)\div\uvecepsdelt\right)
            \, \dx \, \dt
            \\
            &=
            \int_0^T \psi \int_\Omega \phi
            \rhod \left(\overline{p_{\coup,\artprs}} - \left(\bulkvisc + \frac{4\shearvisc}{3}\right) \div\uvecd \right)
            \, \dx \, \dt
        \end{aligned}
        \end{equation*}
        holds for any test functions $\psi \in C^\infty_c((0,T))$ and $\phi \in C^\infty_c(\Omega)$.
        This implies in view of the weak convergences in \eqref{EVF(eps,delt) I} that 
        \begin{equation}\label{EVF(eps,delt) II}
            \rhoepsdelt h_{\coup,\artprs}(\rhoepsdelt) \to \overline{\rhod h_{\coup,\artprs}} \quad \text{in } \mathcal{D}^\prime(\OmegaT),
        \end{equation}
        with 
        \begin{equation}\label{EVF(eps,delt) III}
            \overline{\rhod h_{\coup,\artprs}}
            :=
            \rhod \overline{h_{\coup,\artprs}} + \rhod\overline{q_\artprs} - \overline{\rhod q_\artprs}
            +
            \left(\bulkvisc + \frac{4\shearvisc}{3}\right)\big(\overline{\rhod\div\uvecd} - \rhod \div\uvecd \big)
            \in L^1(\OmegaT).
        \end{equation}
        In particular, we have due to the monotonicity of $r\mapsto h_{\coup,\artprs}(r)$ (see e.g.~\cite[Lemma~3.35]{NovotnyStraskraba2004}) that
        \begin{equation}\label{EVF(eps,delt) IV}
            \overline{\rhod h_{\coup,\artprs}} - \rhod \overline{h_{\coup,\artprs}}\geq 0 \quad \text{a.e.~in } \OmegaT.
        \end{equation}
        From the continuity equation \eqref{wk sol rNSK cont(delt)} and the regularity of $\rhod$ and $\uvecd$ (see \eqref{lem:wk cv(eps,delt) convergences}) we deduce that
        \begin{equation*}
            \partial_t(\rhod\log\rhod) 
            +
            \div\bigl((\rhod\log\rhod)\uvecd\bigr)
            +
            \rhod \div\uvecd = 0
            \quad \text{in } \mathcal{D}^\prime(\RR^3_T),
        \end{equation*}
        see e.g.~\cite[Lemma~11.13]{FeireislNovotny2017singlim}.
        In particular, we have $\rhod\log\rhod \in \Cw([0,T];L^{\artgrwth^-}(\Omega))$ and
        \begin{equation}\label{rhodlogrhod(delt)}
            \int_\Omega 
            \rhod\log\rhod(\tau) 
            -
            \rho_{0,\artprs}\log\rho_{0,\artprs}
            \, \dx 
            =
            -\int_0^\tau\int_\Omega \rhod \div\uvecd \, \dx \, \dt \quad \forall \tau \in [0,T].
        \end{equation}
        By virtue of the uniform bounds in \eqref{lem:unif-in-n bds I}, we have for a non-relabeled subsequence that
        \begin{equation*}
            \rhoepsdelt\log\rhoepsdelt \weakstar \overline{\rhod\log\rhod} \quad \text{in } L^\infty(0,T;L^{\artgrwth^-}(\Omega)).
        \end{equation*}
        From \eqref{reg wksol rNSK(delta,eps)} and \eqref{wk sol cont(eps,delt)}, we conclude by approximating $r\mapsto r \log r$ by suitable smooth and convex functions (see e.g.~\cite[Section~7.9.3]{NovotnyStraskraba2004}) that
        \begin{equation*}
            \int_\Omega 
            \rhoepsdelt\log \rhoepsdelt(\tau) - \rho_{0,\artprs}\log\rho_{0,\artprs} \, \dx
            \leq 
            -\int_0^\tau \int_\Omega \rhoepsdelt\div\uvecepsdelt\, \dx \, \dt
            \quad \text{for a.e.~} \tau \in (0,T).
        \end{equation*}
        Passing to the limit $\artvisc\to 0$ leads to 
        \begin{equation}\label{wklim rhodlogrhod(delt)}
            \int_\Omega 
            \overline{\rhod\log\rhod}(\tau) - \rho_{0,\artprs}\log \rho_{0,\artprs}
            \, \dx 
            \leq
            -
            \int_0^\tau \int_\Omega 
            \overline{\rhod\div\uvecd}
            \, \dx \, \dt 
            \quad \text{for a.e.~} \tau \in (0,T).
        \end{equation}
        Combining \eqref{rhodlogrhod(delt)} and \eqref{wklim rhodlogrhod(delt)} leads for almost all $\tau \in (0,T)$ to
        \begin{equation}\label{wk-lim-q-str-conv}
        \begin{aligned}
            \int_\Omega 
            \overline{\rhod \log \rhod}(\tau) - \rhod\log\rhod(\tau)
            \, \dx 
            &\leq 
            \int_0^\tau \int_\Omega 
            \rhod\div\uvecd - \overline{\rhod\div\uvecd}
            \, \dx \, \dt
            \\
            &\leq 
            C\int_0^\tau \int_\Omega \rhod\overline{q_\artprs} - \overline{\rhod q_\artprs} \, \dd x \, \dd t,
        \end{aligned}
        \end{equation}
        where we have used \eqref{EVF(eps,delt) III} and \eqref{EVF(eps,delt) IV} to obtain the second line.
        From now on, we can proceed as in \cite[Section~5]{Feireisl2002} to derive the strong convergence of the density. For the convenience of the reader we repeat these arguments here.
        Since $q\in C^\infty_c([0,\infty))$, we find some $R>0$ such that $q=0$ on $[R,\infty)$.
        Moreover, we can choose some constant $\Lambda>0$ large enough such that
        \begin{align*}
            r\mapsto \Lambda r \log r + rq(r),
            &&
            r\mapsto \Lambda r \log r - q(r)
        \end{align*}
        are both convex functions on $[0,\infty)$.
        In view of the weak convergences in \eqref{EVF(eps,delt) I}, we then have
        \begin{equation*}
            \Lambda\left(\overline{\rhod\log\rhod} - \rhod\log\rhod\right) \geq \rhod q(\rhod) - \overline{\rhod q_{\artprs}},
            \quad 
            \Lambda\left(\overline{\rhod\log\rhod} - \rhod\log\rhod\right) \geq \overline{q_\artprs}-q(\rhod)
            \quad \text{a.e.~in } \OmegaT,
        \end{equation*}
        see e.g.~\cite[Theorem~11.27]{FeireislNovotny2017singlim}.
        Thanks to these relations we estimate for almost all $\tau \in (0,T)$ 
        \begin{equation*}
        \begin{aligned}
            &\int_0^\tau \int_\Omega 
            \rhod \overline{q_\artprs} - \overline{\rhod q_\artprs}
            \, \dx \, \dt
            \\
            &=
            \int_0^\tau \int_\Omega 
            \rhod\left(\overline{q_\artprs} - q(\rhod)\right)
            \, \dx  \, \dt
            +
            \int_0^\tau \int_\Omega 
            \rhod q(\rhod) - \overline{\rhod q_\artprs}
            \, \dx \, \dt 
            \\
            &\leq 
            \int\int_{I_{R,\tau}} 
            \rhod\left(\overline{q_\artprs} - q(\rhod)\right)
            \, \dx \, \dt
            +
            \Lambda \int_0^\tau \int_\Omega 
            \overline{\rhod\log\rhod} - \rhod\log\rhod
            \, \dx \, \dt 
            \\
            &\leq
            \Lambda
            \int\int_{I_{R,\tau}} 
            \rhod\left(\overline{\rhod\log\rhod} - \rhod\log\rhod\right)
            \, \dx \, \dt
            +
            \Lambda \int_0^\tau \int_\Omega 
            \overline{\rhod\log\rhod} - \rhod\log\rhod
            \, \dx \, \dt  
            \\
            &\leq 
            \Lambda(R+1)
            \int_0^\tau \int_\Omega 
            \overline{\rhod\log\rhod} - \rhod\log\rhod
            \, \dx \, \dt,
        \end{aligned}
        \end{equation*}
        where $I_{R,\tau}:=\{(t,x)\in(0,\tau)\times\Omega \mid \rhod(t,x) \leq R\}$, and where we have used in the third line that $\overline{q_{\artprs}}\leq 0$ a.e.~in $\OmegaT$ and $q=0$ on $[R,\infty)$.
        In total, we have shown that for almost all $\tau \in (0,T)$ we have
        \begin{equation*}
            \int_\Omega 
            \overline{\rhod\log\rhod}(\tau) - \rhod\log\rhod(\tau)
            \, \dx
            \leq 
            C\int_0^\tau \int_\Omega 
            \overline{\rhod \log \rhod} - \rhod\log\rhod
            \, \dx \, \dt.
        \end{equation*}
        With Gr\"onwall's inequality, we conclude that
        \begin{equation*}
            \int_\Omega\overline{\rhod\log\rhod}(\tau) \, \dx \leq \int_\Omega \rhod\log\rhod (\tau) \quad \text{for a.e.~} \tau \in (0,T),
        \end{equation*}
        which implies by convexity of $r\mapsto r\log r$ that
        \begin{equation*}
            \overline{\rhod\log\rhod} = \rhod\log\rhod \quad \text{a.e.~in } \OmegaT.
        \end{equation*}
        In particular, we have that
        \begin{equation*}
            \rhoepsdelt \to \rhod \quad \text{in } L^1(\OmegaT),
        \end{equation*}
        see e.g.~\cite[Lemma~3.34]{NovotnyStraskraba2004}.
        In view of \eqref{lem:impr bd(delt) I}, this implies by interpolation
        \begin{equation*}
            \rhoepsdelt\to \rhod \quad \text{in } L^{(\artgrwth+1)^-}(\OmegaT)
        \end{equation*}
        and, in particular, $\overline{p_{\coup,\artprs}}=p_{\coup,\artprs}(\rhod)$ a.e.~in $\OmegaT$.
    \end{proof}
    Finally, we use the convergences derived so far to conclude the proof of Proposition~\ref{prop:wkSol rNSK(delta)} by performing the artificial viscosity limit $\artvisc\to 0$ in the energy inequality.
    \begin{proof}[Proof of Proposition~\ref{prop:wkSol rNSK(delta)}]
        In view of Lemma~\ref{lem:wk cv(eps,delt)} and Lemma~\ref{lem:str cv density(eps,delt)}, we only have to show that the triplet $(\rhod,\uvecd,\cd)$ satisfies the bound \eqref{wk sol rNSK bounds(delt)} as well as the energy inequality \eqref{wk sol rNSK energy(delt)}.
        The triplet $(\rhoepsdelt,\uvecepsdelt,\cepsdelt)$ satisfies \eqref{wkSol rNSK(eps,delt) energy}, which implies that 
        \begin{equation}\label{energy ineq E_m,delt I}
            \begin{aligned}
                &-\int_0^T E_{\mathrm{m},\artprs}[\rhoepsdelt,\rhoepsdelt\uvecepsdelt,\cepsdelt]\partial_t \psi \, \dt 
                +
                \int_0^T \psi \int_\Omega 
                \Svisc(\nablax \uvecepsdelt):\nablax\uvecepsdelt + \paracoup|\partial_t\cepsdelt|^2
                \, \dx \, \dt \\
                &\leq 
                E_{\mathrm{m},\artprs}[\rho_{0,\artprs},(\rho\uvec)_{0,\artprs},c_0] \psi(0)
                +
                \int_0^T \psi \int_\Omega
                q(\rhoepsdelt)\div\uvecepsdelt
                +\artvisc\coup \nablax \rhoepsdelt\cdot\nablax\cepsdelt
                \, \dx \, \dt 
            \end{aligned}
        \end{equation}
        holds for any $\psi \in C^\infty_c([0,T))$ with $\psi \geq 0$.
        In view of Lemma~\ref{lem:wk cv(eps,delt)} and Lemma~\ref{lem:str cv density(eps,delt)}, we have that
        \begin{equation*}
        \begin{aligned}
            &\left|\int_0^T \psi \int_\Omega \artvisc\coup \nablax \rhoepsdelt\cdot\nablax\cepsdelt
                \, \dx \, \dt \right|
                \leq 
                \coup\|\psi\|_{L^\infty(0,T)}\|\artvisc\nablax\rhoepsdelt\|_{L^2(\OmegaT)}\|\nablax\cepsdelt\|_{L^2(\OmegaT)}
                \to 0,
                \\
            &\int_0^T \psi\int_\Omega q(\rhoepsdelt)\div\uvecepsdelt\, \dx \, \dt 
            \to 
            \int_0^T \psi \int_\Omega q(\rhod) \div\uvecd 
            \, \dx \, \dt,
        \end{aligned}
        \end{equation*}
        and
        \begin{equation*}
            E_{\mathrm{m},\artprs}[\rhoepsdelt,\rhoepsdelt\uvecepsdelt,\cepsdelt] \weak E_{m,\artprs}[\rhod, \rhod\uvecd,\cd]
            \quad \text{in }L^1((0,T)).
        \end{equation*}
        Thus, by using the weak lower semi-continuity of convex functionals, passing to the limit $\artvisc\to 0$ in \eqref{energy ineq E_m,delt I} leads to
        \begin{equation}\label{energy ineq E_m,delt II}
            \begin{aligned}
                &-\int_0^T \, E_{\mathrm{m},\artprs}[\rhod,\rhod\uvecd,\cd]\partial_t \psi \, \dt 
                +
                \int_0^T \psi \int_\Omega 
                \Svisc(\nablax \uvecd):\nablax\uvecd + \paracoup|\partial_t\cd|^2
                \, \dx \, \dt \\
                &\leq 
                E_{\mathrm{m},\artprs}[\rho_{0,\artprs},(\rho\uvec)_{0,\artprs},c_0] \psi(0)
                +
                \int_0^T \psi \int_\Omega q(\rhod)\div\uvecd\, 
                \dx \, \dt .
            \end{aligned}
        \end{equation}
        From \eqref{wk sol rNSK cont(delt)} and the regularity of $\rhod$ and $\uvecd$ (see~\eqref{lem:wk cv(eps,delt) convergences}) we deduce
        \begin{equation*}
            \partial_tQ(\rhod)
            +
            \div\bigl(Q(\rhod)\uvecd\bigr)
            +
            q(\rhod)\div\uvecd
            = 0
            \quad \text{in } \mathcal{D}^\prime(\RR^3_T),
        \end{equation*}
        where $Q$ denotes the pressure potential corresponding to $q$ in \eqref{def press pot H Q W}, see e.g.~\cite[Lemma~11.13]{FeireislNovotny2017singlim}.
        This implies 
        \begin{equation}\label{int id q(rhod)div uvecd}
            \int_0^T \psi \int_\Omega  q(\rhod)\div\uvecd 
            \, \dx \, \dt 
            =
            \int_0^T \partial_t \psi \int_\Omega Q(\rhod) \, \dx \, \dt 
            +
            \psi(0) \int_\Omega Q(\rho_{0,\artprs}) \, \dx \, .
        \end{equation}
        Substituting \eqref{int id q(rhod)div uvecd} into \eqref{energy ineq E_m,delt II} yields precisely \eqref{wk sol rNSK energy(delt)}.
        Due to \eqref{mollify III}, we have that the triplet $(\rhoepsdelt,\uvecepsdelt,\cepsdelt)$ satisfies the bound \eqref{energy bds(n) I} with $D_0$ replaced by $\data$.
        This implies \eqref{wk sol rNSK bounds(delt)} by using Lemma~\ref{lem:wk cv(eps,delt)}, Lemma~\ref{lem:str cv density(eps,delt)}, and the weak lower semi-continuity of norms.
    \end{proof}

%% file: VanArtPres.tex
    Let us assume that the hypotheses and notations of \Cref{main result:ex wkSol} hold true, and let $\rho_{0,\artprs}$, $(\rho\uvec)_{0,\artprs}$ denote regularized initial conditions satisfying \eqref{mollify I}--\eqref{mollify III}.
    Thanks to Proposition~\ref{prop:wkSol rNSK(delta)}, we find for any $\artprs\in (0,1)$ some triplet $(\rhod,\uvecd,\cd)$ satisfying \eqref{wk sol rNSK(delt) regularity}--\eqref{wk sol rNSK energy(delt)}.
    In this section, we perform the vanishing artificial pressure limit, that is, the limit $\artprs\to 0$ for the approximating sequence of solutions $\{(\rhod,\uvecd,\cd)\}_{\artprs>0}$.
    As an outcome of this limit procedure, we obtain a limit triplet $(\rho,\uvec,c)$ that is a finite energy weak solution to the rNSKE \eqref{NSK(alpha,beta) system}--\eqref{NSK(alpha,beta) IC} on $\OmegaT$ emanating from the initial conditions $(\rho_0,(\rho\uvec)_0,c_0)$, thus completing the proof of \Cref{main result:ex wkSol}.
    \subsection{Uniform estimates and weak limit}
    To perform the vanishing artificial pressure limit $\artprs\to 0$ for the approximate sequence of solutions $\{(\rhod,\uvecd,\cd)\}_{\artprs>0}$, we need some bounds that are uniform in $\artprs$.
    Such uniform bounds are given in \eqref{wk sol rNSK bounds(delt)}, however, similarly as in \Cref{sec:VanArtVisc}, we notice that these bounds cannot guarantee equi-integrability for the sequence $\{p_{\coup,\artprs}(\rhod)\}_{\artprs>0}$.
    Therefore, we need to improve the uniform bound on the sequence of densities by adapting again the technique that is known from the compressible NSE (see e.g.~\cite{FeireislNovotnyPetzeltova2001, FeireislPetzeltova2000}).
    However, this time, we have to construct the suitable test function for the momentum equation with more care, since the uniform bound on the density is weaker than in \Cref{sec:VanArtVisc}. The proof of the following Lemma is similar to the one in \cite[Lemma~4.4]{RohdeWendt2025a}.
    \begin{lemma}\label{lem:impr prs}
        Let the hypotheses and notations of Proposition~\ref{prop:wkSol rNSK(delta)} hold true.
        Assume for $\artprs\in (0,1)$ that the triplet $(\rhod,\uvecd,\cd)$ satisfies \eqref{wk sol rNSK(delt) regularity}--\eqref{wk sol rNSK energy(delt)}.
        Then we have
        \begin{equation}\label{impr prs est I}
            \|\rhod\|_{L^{\tgamma+\theta}(\OmegaT)}
            +
            \artprs^{\frac{1}{\artgrwth+\theta}}\|\rhod\|_{L^{\artgrwth+\theta}(\OmegaT)}
            \leq 
            \Cdata,
        \end{equation}
        where $\theta := \frac{2\tgamma-3}{3}$.
        In particular, we have 
        \begin{equation}\label{impr prs cv art prs}
            \delta\rhod^\artgrwth \to 0 \quad \text{in } L^1(\OmegaT).
        \end{equation}
    \end{lemma}
    \begin{remark}
        In dimension $d=2$, we have the same result with $\theta = (\tgamma-1)^-$.
    \end{remark}
    \begin{proof}
        For the sake of completeness we repeat the main arguments here.
        We set 
        \begin{equation*}
            \tilde{\theta}:= \min \left\{\frac{2\tgamma-3}{3}, \frac{\tgamma}{2}\right\}.
        \end{equation*}
        Using the same regularization argument as in \cite[Section~7.9.5]{NovotnyStraskraba2004} we can assume without loss of generality that
        \begin{equation}\label{wlog regularity}
            \rhod^{\tilde{\theta}}, \, \partial_t(\rhod^{\tilde{\theta}}) \in L^\infty(0,T;L^\infty(\Omega)).
        \end{equation}
        We define
        \begin{equation*}
            \varphivec_{\artprs}
            :=
            \Bog\left(\rhod^{\tilde{\theta}} - \fint_\Omega \rhod^{\tilde{\theta}} \right),
        \end{equation*}
        where $\Bog$ denotes the operator from \Cref{lem:Bog}.
        Due to the continuity equation \eqref{wk sol rNSK(delt) regularity} we have that
        \begin{equation}\label{impr prs RENORM}
            \partial_t \left(\rhod^{\tilde{\theta}}\right) 
            +
            \div\left(\rhod^{\tilde{\theta}} \uvecd\right)
            =
            \left(1-{\tilde{\theta}}\right)\rhod^{\tilde{\theta}} \div\uvecd
            \quad \text{in } \mathcal{D}^\prime\left(\RR^3_T\right),
        \end{equation}
        see e.g.~\cite[Lemma~11.13]{FeireislNovotny2017singlim}.
        Thanks to \eqref{wlog regularity} this yields $\div(\rhod^{\tilde{\theta}}\uvecd)\in L^2(\OmegaT)$, thus, \eqref{impr prs RENORM} holds almost everywhere in $\OmegaT$.
        This implies $\rhod^{\tilde{\theta}}\uvecd \in L^2(0,T;E^{6,2}_0(\Omega))$, and moreover
        \begin{equation*}
        \begin{aligned}
            \partial_t \varphivec_\artprs
            =
            \left(1-{\tilde{\theta}}\right) \Bog \left(\rhod^{\tilde{\theta}} \div\uvecd - \fint_\Omega \rhod^{\tilde{\theta}}\div\uvecd \right)
            -
            \Bog\left(\div(\rhod^{\tilde{\theta}}\uvecd)\right)
            \quad \text{a.e.~in } \OmegaT.
        \end{aligned}
        \end{equation*}
        Let $\psi \in C_c^\infty((0,T))$. By a density argument we may use $\psi \varphivec_\artprs$ as a test function in the momentum equation \eqref{wk sol rNSK mom(delt)} to obtain
         \begin{equation*}
             \begin{aligned}
                 &\int_0^T \psi \int_\Omega 
                 p_{\coup,\artprs}(\rhod) \rhod^{\tilde{\theta}} 
                \, \dx \, \dt 
                =
                \int_0^T \psi \int_\Omega
                p_{\coup,\artprs}(\rhod)
                \left( \fint_\Omega \rhod^{\tilde{\theta}}  \right)
                \, \dx \, \dt 
                \\
                &\quad
                -
                \int_0^T \partial_t \psi \int_\Omega 
                \rhod\uvecd\cdot \varphivec_\artprs
                \, \dx \, \dt  
                -
                \int_0^T \psi \int_\Omega 
                \rhod\uvecd \cdot \partial_t \varphivec_\artprs
                \, \dx \, \dt 
                \\
                &\quad 
                -
                \int_0^T \psi \int_\Omega 
                \rhod\uvecd\otimes\uvecd : \nablax \varphivec_\artprs
                \, \dx \, \dt 
                +
                \int_0^T\psi \int_\Omega 
                \Svisc(\nablax\uvecd):\nablax\varphivec_\artprs
                \, \dx \, \dt 
                \\
                &\quad -
                \int_0^T \psi \int_\Omega 
                \rhod\nablax\cd \cdot \varphivec_\artprs
                \, \dx \, \dt
                =:
                \sum\limits_{i}^{6} I_i^{\artprs}.
            \end{aligned}
        \end{equation*}
        In view of \eqref{wk sol rNSK bounds(delt)} and \Cref{lem:Bog}, we have for $r:= \frac{\tgamma}{\tilde{\theta}}$ that
        \begin{equation}\label{bds varphivec delt}
        \begin{aligned}
            &\|\varphivec_\artprs\|_{L^\infty(0,T;L^{\overline{r}}(\Omega))}
            +
            \|\nablax \varphivec_\artprs\|_{L^\infty(0,T;L^{r}(\Omega))}
            +
            \|\partial_t \varphivec_\artprs\|_{L^2(0,T;L^\frac{6r}{6+r}(\Omega))}
            \leq 
            \Cdata,
        \end{aligned}
        \end{equation}
        where $\overline{r}$ is defined in \eqref{defi Sobolev exponents}.
        Thanks to
        \begin{equation*}
            \frac{\tgamma}{\tilde{\theta}} \geq \max \left\{ \frac{3\tgamma}{2\tgamma-3}, 2 \right\} \geq \frac{6\tgamma}{5\tgamma-6},
        \end{equation*}
        we infer with \eqref{wk sol rNSK bounds(delt)}, \eqref{bds varphivec delt}, and H\"older's inequality that
        \begin{equation*}
        \begin{aligned}
            \left|I_6^\artprs\right|
            &\leq 
            C\|\psi\|_{L^\infty((0,T))}\|\rhod \nablax  \cd \|_{L^2(0,T;L^{\frac{6\tgamma}{\tgamma+6}}(\Omega))} \|\varphivec_\artprs\|_{L^2(0,T;L^{\frac{6\tgamma}{5\tgamma-6}}(\Omega))}
            \leq 
            \Cdata\|\psi\|_{L^\infty((0,T))}.
        \end{aligned}
        \end{equation*}
        By similar arguments, we infer for the remaining terms that
        \begin{equation*}
            \sum\limits_{i=1}^{5}\left|I_i^\artprs\right|
            \leq 
            \Cdata
            \left(\|\psi\|_{L^\infty((0,T))} + \|\partial_t\psi\|_{L^1((0,T))}\right).
        \end{equation*}
        These estimates are known from the theory on the compressible NSE and thus, we omit the details here. For a detailed exposition on these estimates we refer to \cite[Chapter~4]{FeireislNovotnyPetzeltova2001}.
        In total, we have shown for any $\psi \in C^\infty_c((0,T))$ that
        \begin{equation*}
            \int_0^T \psi \int_\Omega 
            p_{\coup,\artprs}(\rhod) \rhod^{\tilde{\theta}} 
            \, \dx \, \dt 
            \leq 
            \Cdata
            \left(\|\psi\|_{L^\infty((0,T))} + \|\partial_t\psi\|_{L^1((0,T))}\right).
        \end{equation*}
        Approximating the identity function $\mathrm{id}_{(0,T)}$ on $(0,T)$ with suitable test functions $\psi \in C^\infty_c((0,T))$ and using \eqref{grwth adm p} yields
        \begin{equation*}
            \|\rhod\|_{L^{\tgamma+\tilde{\theta}}(\OmegaT)} + \delta^{\frac{1}{\artgrwth+\tilde{\theta}}} \|\rhod\|_{L^{\artgrwth+\tilde{\theta}}(\OmegaT)} \leq  
            \Cdata.
        \end{equation*}
        Repeating the arguments with $\tilde{\theta}$ replaced by $\theta := \frac{2\tgamma-3}{3}$ by using the preceding estimate leads to \eqref{impr prs est I}.
        The convergence \eqref{impr prs cv art prs} follows from \eqref{impr prs est I} by using H\"older's inequality.
    \end{proof}
    Similarly as in \Cref{sec:VanArtVisc}, we use now Lemma~\ref{lem:impr prs} and the uniform bounds in \eqref{wk sol rNSK bounds(delt)} to identify weak limits that allow us, after passing to a non-relabeled subsequence, to pass to the limit $\artprs\to 0$ in the equations \eqref{wk sol rNSK cont(delt)}--\eqref{wk sol rNSK parab(delt)}.
    \begin{lemma}\label{lem:wk cv(delt)}
        Let the hypotheses and notations of Proposition~\ref{prop:wkSol rNSK(delta)} hold true.
        Then we have, after passing to a non-relabeled subsequence,
        \begin{equation}\label{wk lim(delt) convergences}
            \begin{aligned}
                &\rhod \to \rho \quad \text{in } \Cw([0,T];L^\tgamma(\Omega)),
                \quad
                \rhod \weak \rho \quad \text{in } L^{\frac{5\tgamma-3}{3}}(\OmegaT),
                \\
                &\uvecd \weak \uvec \quad \text{in } L^2(0,T;W^{1,2}_0(\Omega)),
                \quad
                \cd \weak c \quad \text{in } L^2(0,T;W^{2,2}(\Omega)),
                \\
                &\cd \weakstar c \quad \text{in } L^\infty(0,T;W^{1,2}(\Omega)),
                \quad
                \cd \to c \quad \text{in } C([0,T];L^2(\Omega))\cap L^2(0,T;W^{1,2}(\Omega)),
                \\
                &\rhod\uvecd \to \rho\uvec \quad \text{in } \Cw([0,T];L^{\frac{2\tgamma}{\tgamma+1}}(\Omega)),
                \quad
                \partial_t \cd \weak \partial_t c \quad \text{in } L^2(\OmegaT)
                \\
                &\rhod\nablax \cd \weak \rho\nablax\cd \quad \text{in } L^2(0,T;L^{\frac{6\tgamma}{\tgamma+6}}(\Omega)),
                \quad 
                \Peff(\rhod) \weak \overline{\Peff} \quad \text{in } L^{\frac{5\tgamma-3}{3\tgamma}}(\OmegaT),
            \end{aligned}
        \end{equation}
        with $\rho \geq 0$ a.e.~in $\OmegaT$.
        Moreover, the triplet $(\rho,\uvec,c)$ satisfies \eqref{def wsol reg}--\eqref{def wsol IC} with $\Peff(\rho)$ replaced by $\overline{\Peff}$.
    \end{lemma}
    \begin{proof}
        Using the Sobolev embedding $W^{1,2}(\Omega)\hookrightarrow L^6(\Omega)$ and H\"older's inequality we deduce from \eqref{wk sol rNSK cont(delt)}, \eqref{wk sol rNSK mom(delt)}, and \eqref{wk sol rNSK bounds(delt)} that
        \begin{equation*}
            \|\partial_t \rhod\|_{L^2(0,T;W^{-1,s}(\Omega))}
            +
            \|\partial_t (\rhod\uvecd) \|_{L^s(0,T;W^{-1,s}(\Omega))}
            \leq 
            \Cdata
        \end{equation*}
        for some $s\in (1,\infty)$.
        Taking into account \eqref{wk sol rNSK bounds(delt)}, we thus have, after passing to a non-relabled subsequence,
        \begin{equation}\label{wk limit(delt) Cw cv}
            \rhod \to \rho \quad \text{in } \Cw([0,T];L^\tgamma(\Omega)),
            \quad 
            \rhod\uvecd \to \overline{\rho\uvec} \quad \text{in } \Cw([0,T];L^{\frac{2\tgamma}{\tgamma+1}}(\Omega)).
        \end{equation}
        Using the Sobolev embedding $W^{1,2}(\Omega)\hookrightarrow L^6(\Omega)$, H\"older's inequality, and the Banach--Alaoglu theorem, we further conclude from the uniform bounds in \eqref{wk sol rNSK bounds(delt)} that, up to a subsequence,
        \begin{equation}\label{wk cv(delt)}
            \begin{aligned}
                &\rhod \weak \rho \quad\text{in } L^{\frac{5\tgamma-3}{3}}(\OmegaT),
                \quad
                \uvecd \weak \uvec \quad \text{in } L^2(0,T;W^{1,2}_0(\Omega)),
                \\
                &\cd \weakstar c \quad \text{in } L^\infty(0,T;W^{1,2}(\Omega)),
                \quad
                \cd \weak c \quad \text{in } L^2(0,T;W^{2,2}(\Omega)),
                \\
                &\Peff(\rhod)\weak \overline{\Peff} \quad \text{in } L^{\frac{5\tgamma-3}{3\tgamma}}(\OmegaT),
                \quad
                \rhod\nablax\cd \weak \overline{\rho\nablax c} \quad \text{in } L^2(0,T;L^{\frac{6\tgamma}{\tgamma+6}}(\Omega)),
                \\
                &\rhod\uvecd\otimes\uvecd \weak \overline{\rho\uvec\otimes\uvec} \quad \text{in } L^2(0,T;L^{\frac{3\tgamma}{\tgamma+3}}(\Omega)),
                \quad
                \partial_t \cd \weak \partial_t c \quad \text{in } L^2(\OmegaT),
            \end{aligned}
        \end{equation}
        with $\rho \geq 0$ a.e.~in $\OmegaT$.
        Here, the overlined quantities denote weak limits of the corresponding sequences in their respective spaces.
        By standard results on Bochner spaces (see e.g.~\cite[Chapter~5]{EvansPDE2010}), we conclude that $c \in C([0,T];W^{1,2}(\Omega))$.
        Using the compact Sobolev embeddings $L^\tgamma(\Omega) \hookrightarrow\hookrightarrow W^{-1,2}(\Omega)$ and $L^\frac{2\tgamma}{\tgamma+1}(\Omega) \hookrightarrow\hookrightarrow W^{-1,2}(\Omega)$, we deduce from \eqref{wk limit(delt) Cw cv} and the second convergence in \eqref{wk cv(delt)} that
        \begin{equation}\label{wk limit(delt) id I}
            \overline{\rho\uvec} = \rho\uvec,
            \quad 
            \overline{\rho\uvec\otimes\uvec} = \rho\uvec\otimes\uvec 
            \quad \text{a.e.~in } \OmegaT.
        \end{equation}
        Moreover, using \eqref{wk sol rNSK bounds(delt)} as well as the compact Sobolev embedding $W^{1,2}(\Omega)\hookrightarrow\hookrightarrow L^2(\Omega)$, we conclude with the Aubin--Lions theorem that
        \begin{equation}\label{wk limit(delt) str cv c}
            \cd \to c\quad \text{in } C([0,T];L^2(\Omega)),
            \quad 
            \cd \to c \quad \text{in } L^2(0,T;W^{1,2}(\Omega)),
        \end{equation}
        which yields in combination with \eqref{wk cv(delt)} that
        \begin{equation}\label{wk limit(delt) id II}
            \overline{\rho\nablax c} = \rho\nablax c \quad \text{a.e.~in } \OmegaT.
        \end{equation}
        In view of \eqref{wk cv(delt)}--\eqref{wk limit(delt) id II}, we have shown that the triplet $(\rho,\uvec,c)$ satisfies \eqref{def wsol reg} and \eqref{wk lim(delt) convergences}.
        Thanks to \eqref{impr prs cv art prs} and \eqref{wk lim(delt) convergences}, passing to the limit $\artprs\to 0$ in the equations \eqref{wk sol rNSK cont(delt)}--\eqref{wk sol rNSK parab(delt)} yields that the triplet $(\rho,\uvec,c)$ satisfies \eqref{def wsol cont}--\eqref{def wsol parab} with $\Peff(\rho)$ replaced by $\overline{\Peff}$.
        Combining the convergences for the regularized initial conditions in \eqref{mollify II} with the convergences in \eqref{wk limit(delt) Cw cv} and in \eqref{wk limit(delt) str cv c} yields that the triplet $(\rho,\uvec,c)$ satisfies the initial conditions \eqref{def wsol IC}.
    \end{proof}
    \subsection{Strong convergence of the density}
    To complete the proof of \Cref{main result:ex wkSol}, we wish to deduce, as in \Cref{sec:VanArtVisc}, the strong convergence $\rhod \to \rho$ in $L^1(\OmegaT)$, which implies, in particular, that $\overline{p_\coup} = p_\coup (\rho)$.
    To do so, we adapt the approach in \cite{FeireislNovotnyPetzeltova2001,Feireisl2002} to the rNSKE. 
    More specifically, this time, since the pressure function $\Peff$ is in general non-monotone and since the uniform bounds are weaker than in \Cref{sec:VanArtVisc}, we have to use specific cut-off functions that were introduced in \cite{FeireislNovotnyPetzeltova2001,Feireisl2002} to approximate the density.
    Note that if $\Peff$ would be non-decreasing, we could also adapt the approach from \cite{Lions1998} to the rNSKE without specific cut-off functions.
    \begin{lemma}\label{lem:str cv dens}
        Let the hypotheses and notations of \Cref{lem:wk cv(delt)} hold true. Then we have
        \begin{equation*}
            \rhod \to \rho \quad \text{in } L^{\big(\frac{5\tgamma-3}{3}\big)^-}(\OmegaT).
        \end{equation*}
        In particular, we have
        \begin{equation*}
            \overline{\Peff} = \Peff(\rho) \quad \text{a.e.~in } \OmegaT.
        \end{equation*}
    \end{lemma}
    \begin{proof}
        We fix some smooth function $T \in C^\infty([0,\infty))$ that satisfies
        \begin{equation*}
            \begin{aligned}
                &T(z) = z \quad \forall\, z \in [0,1],
                \qquad 
                T(z)=2 \quad \forall\, z \in [3,\infty),
                \qquad 
                T^{\prime\prime}(z)\leq 0 \quad \forall\, z \in [0,\infty).
            \end{aligned}
        \end{equation*}
        For $k \in \NN$, we then introduce the cut-off functions $T_k$, $L_k\colon [0,\infty)\to \RR$ via
        \begin{equation*}
        \begin{aligned}
            T_k(z):= kT\left(\frac{z}{k} \right),
            \qquad 
            L_k(z):= z\int_1^z \frac{T_k(s)}{s^2} \, \dd s
            \qquad \forall\, z \in [0,\infty).
        \end{aligned}
        \end{equation*}
        We readily verify that $T_k \in W^{1,\infty}((0,\infty))$ and 
        \begin{equation*}
            \begin{aligned}
                &|L_k(z)|\leq z\log z,
                \quad 
                |T_k(z)|\leq z,
                \quad
                L_k^\prime(z)z-L_k(z)=T_k(z),
                \\
                &T_k^\prime(z)\geq 0,
                \quad 
                T_k(z) \to z,
                \quad L_k(z) \to z\log z
                \qquad \forall\, z \in [0,\infty).
            \end{aligned}
        \end{equation*}
        Thanks to \eqref{wk sol rNSK(delt) regularity} and \eqref{wk sol rNSK cont(delt)} we have that
        \begin{equation*}
            \partial_t T_k(\rhod)
            +
            \div\big(T_k(\rhod)\uvecd\big)
            +
            \big(T_k^\prime(\rhod)\rhod - T_k(\rhod)\big)\div\uvecd
            =
            0
            \quad \text{in } \mathcal{D}^\prime(\RR^3_T).
        \end{equation*}
        In view of \eqref{wk sol rNSK bounds(delt)}, this implies, after passing to a non-relabeled subsequence, that
        \begin{equation*}
        \begin{aligned}
            &T_k(\rhod)\to \overline{T_k} \quad \text{in } \Cw([0,T];L^{\infty^-}(\Omega)),
            \quad 
            T_k(\rhod)\weakstar \overline{T_k}\quad \text{in } L^\infty(0,T;L^\infty(\Omega)),
            \\
            &\big(T_k^\prime(\rhod)\rhod-T_k(\rhod)\big)\div\uvecd \weak \overline{\big(T_k^\prime(\rho)\rho- T_k(\rho)\big)\div
            \uvec} \quad \text{in } L^2(\OmegaT).
        \end{aligned}
        \end{equation*}
        By virtue of \Cref{lem:wk cv(delt)} we may thus apply Proposition~\ref{prop:EVF(general)} to conclude that the relation
        \begin{equation}\label{str cv EVF rel}
            \begin{aligned}
                &\lim\limits_{\artprs\to 0} \int_0^T\int_\Omega 
                \varphi\, T_k(\rhod) \left(\Peff(\rhod) - \left(\bulkvisc + \frac{4\shearvisc}{3}\right)\div\uvecd\right)
                \, \dd x \, \dd t
                \\
                &\quad=
                \int_0^T \int_\Omega 
                \varphi\, \overline{T_k} \left(\overline{\Peff} - \left(\bulkvisc + \frac{4\shearvisc}{3}\right) \div\uvec\right)
                \, \dd x \, \dd t
            \end{aligned}
        \end{equation}
        holds for any test function $\varphi \in C^\infty_c(\OmegaT)$.
        In view of \eqref{wk sol rNSK bounds(delt)} and \Cref{lem:impr prs}, we have, after passing to a non-relabeled subsequence, that
        \begin{equation*}
            \begin{aligned}
                &h_\coup(\rhod) \weak \overline{h_\coup} \quad \text{in } L^{\frac{5\tgamma-3}{3\tgamma}}(\OmegaT),
                \quad 
                q(\rhod) \weakstar \overline{q} \quad \text{in } L^\infty(0,T;L^\infty(\Omega)),
                \\
                &T_k(\rhod)h_\coup(\rhod) \weak \overline{T_kh_\coup}\quad \text{in } L^{\frac{5\tgamma-3}{3\tgamma}}(\OmegaT),
                \quad 
                T_k(\rhod)q(\rhod) \weakstar \overline{T_kq} \quad \text{in } L^\infty(0,T;L^\infty(\Omega)),
                \\
                &T_k(\rhod) \div\uvecd \weak \overline{T_k\div\uvec} \quad \text{in } L^2(\OmegaT),
                \quad 
                \rhod q(\rhod) \weak \overline{\rho q} \quad \text{in } L^{\frac{5\tgamma-3}{3}}(\OmegaT),
            \end{aligned}
        \end{equation*}
        with $\overline{\Peff}=\overline{h_\coup}+\overline{q}$ and $\overline{q} \leq 0$.
        Combining these convergences with \eqref{str cv EVF rel} and using the monotonicity of $r\mapsto h_\coup(r)$ and $r\mapsto T_k(r)$ (see e.g. \cite[Lemma~3.35]{NovotnyStraskraba2004}) leads to
        \begin{equation}\label{EVF mon ineq}
            \left(\bulkvisc + \frac{4\shearvisc}{3}\right)\Big(\overline{T_k}\div\uvec - \overline{T_k\div\uvec}\Big)
            \leq 
            \overline{T_k}\,\overline{q} - \overline{T_kq}
            \quad \text{a.e.~in } \OmegaT.
        \end{equation}
        Using \eqref{wk sol rNSK(delt) regularity}, \eqref{wk sol rNSK cont(delt)}, as well as \Cref{lem:wk cv(delt)}, we deduce that
        \begin{equation}\label{renorm Lk}
            \begin{aligned}
                &\partial_t L_k(\rhod) + \div\big(L_k(\rhod)\uvecd\big) = - T_k(\rhod) \div \uvecd 
                \quad \text{in } \mathcal{D}^\prime(\RR^3_T),
                \\
                &\partial_t L_k(\rho) + \div\big(L_k(\rho)\uvec\big) = -T_k(\rho)\div\uvec \quad \text{in } \mathcal{D}^\prime(\RR^3_T).
            \end{aligned}
        \end{equation}
        In view of \eqref{wk sol rNSK bounds(delt)}, this implies that $L_k(\rho),L_k(\rhod)\in \Cw([0,T];L^{\tgamma^-}(\Omega))$ and moreover, after passing to a non-relabeled subsequence, that
        \begin{equation*}
            L_k(\rhod) \to \overline{L_k}\quad \text{in } \Cw([0,T];L^{\tgamma^-}(\Omega)).
        \end{equation*}
        By similar arguments, we infer that $\rho\log \rho,\, \rhod\log\rhod \in \Cw([0,T];L^{\tgamma^-}(\Omega))$ and, after passing to a non-relabeled subsequence,
        \begin{equation*}
            \rhod\log\rhod \to \overline{\rho\log\rho} \qquad \text{in } \Cw([0,T];L^{\tgamma^-}(\Omega)).
        \end{equation*}
        Taking into account \eqref{mollify II}, \Cref{lem:wk cv(delt)}, and using the Sobolev embedding $L^s(\Omega)\hookrightarrow\hookrightarrow W^{-1,2}(\Omega)$ which is compact for any $s >\frac{6}{5}$, we deduce from $\eqref{renorm Lk}_1$ after passing to the limit $\artprs\to 0$ that
        \begin{equation*}
            \int_\Omega \overline{L_k}(\tau) - L_k(\rho_{0}) , \dd x 
            =
            -\int_0^\tau \int_\Omega 
            \overline{T_k\div\uvec}
            \, \dd x \, \dd t
            \quad \forall\, \tau \in [0,T].
        \end{equation*}
        From $\eqref{renorm Lk}_2$ we conclude
        \begin{equation*}
            \int_\Omega L_k(\rho) - L_k(\rho_{0}) \, \dd x
            =
            -\int_0^\tau \int_\Omega T_k(\rho) \div\uvec \, \dd x 
            \quad \forall\, \tau \in [0,T].
        \end{equation*}
        Combining both relations and \eqref{EVF mon ineq} yields for any $\tau \in [0,T]$ that
        \begin{equation}\label{ineq diff Lk}
            \begin{aligned}
                &\int_\Omega 
                \overline{L_k}(\tau) - L_k(\rho(\tau))
                \, \dd x 
                =
                \int_0^\tau \int_\Omega 
                T_k(\rho)\div\uvec - \overline{T_k \div\uvec}
                \, \dd x \, \dd t
                \\
                &\quad \leq 
                \int_0^\tau \int_\Omega 
                \big(T_k(\rho) - \overline{T_k}\big)\div\uvec 
                \, \dd x \, \dd t
                +
                C\int_0^\tau \int_\Omega 
                \overline{T_k}\,\overline{q} - \overline{T_k q}
                \, \dd x \, \dd t .
            \end{aligned}
        \end{equation}
        Now we wish to pass to the limit $k \to \infty$ in \eqref{ineq diff Lk}.
        To do so, we first obtain by using the weak lower semi-continuity of norms that
        \begin{equation*}
        \begin{aligned}
            \|\rho - \overline{T_k}\|_{L^1(\OmegaT)}
            \leq 
            \liminf\limits_{\artprs\to 0} \int_0^T \int_\Omega |\rhod - T_k(\rhod)|\, \dd x \, \dd t
            \leq 
            C\sup\limits_{\delta>0}\int\int_{\{\rhod\geq k\}} 
            \rhod\, \dd x \, \dd t
            \xrightarrow{k\to\infty} 0,
        \end{aligned}
        \end{equation*}
        since $\rhod$ is bounded in $L^{\frac{5\tgamma-3}{3}}(\OmegaT)$ uniformly with respect to $\delta$. 
        With Lebesgue's dominated convergence theorem, we further deduce that $T_k(\rho) \to \rho$ in $L^1(\OmegaT)$. Combining both convergences yields $\|T_k(\rho)-\overline{T_k}\|_{L^1(\OmegaT)}\to 0$.
        Since $T_k(\rho)$ and $\overline{T_k}$ are both bounded in $L^{\frac{5\tgamma-3}{3}}(\OmegaT)$ uniformly with respect to $k$, and since $\frac{5\tgamma-3}{3}>2$ thanks to $\tgamma\geq 2$, we conclude by using interpolation that $\|T_k(\rho) - \overline{T_k}\|_{L^2(\OmegaT)}\to 0$.
        By using H\"older's inequality, we thus have
        \begin{equation*}
            \lim\limits_{k\to\infty} \int_0^\tau \int_\Omega \big(T_k(\rho) - \overline{T_k}\big) \div\uvec \, \dd x \, \dd t = 0.
        \end{equation*}
        By similar arguments, we have that
        \begin{equation*}
        \begin{aligned}
            &\lim\limits_{k\to\infty}\int_\Omega 
            \overline{L_k}(\tau)-L_k(\rho(\tau))
            \, \dd x 
            = 
            \int_\Omega 
            \overline{\rho\log\rho}(\tau) - (\rho\log\rho)(\tau)\, \dd x,
            \\
            &\lim\limits_{k\to\infty}\int_0^\tau \int_\Omega 
            \overline{T_k}\, \overline{q} - \overline{T_k q}
            \, \dd x \, \dd t
            =
            \int_0^\tau \int_\Omega 
            \rho\overline{q} - \overline{\rho q}
            \, \dd x \, \dd t.
        \end{aligned}
        \end{equation*}
        Thus, passing to the limit $k\to \infty$ in \eqref{ineq diff Lk} leads to
        \begin{equation*}
            \begin{aligned}
                \int_\Omega 
                \overline{\rho\log\rho}(\tau) - (\rho\log\rho)(\tau)
                \, \dd x 
                \leq 
                C\int_0^\tau \int_\Omega 
                \rho \overline{q} - \overline{\rho q}
                \, \dd x \, \dd t 
                \qquad \forall\, \tau \in [0,T].
            \end{aligned}
        \end{equation*}
        From now on, we can proceed precisely as in the proof of Lemma~\ref{lem:str cv density(eps,delt)} (see \eqref{wk-lim-q-str-conv}) to conclude that $\rhod \to \rho$ in $L^1(\OmegaT)$.
        Using \Cref{lem:impr prs} and interpolation eventually leads to 
        \begin{equation*}
            \rhod \to \rho \quad \text{in } L^{\big(\frac{5\tgamma-3}{3}\big)^-}(\OmegaT).
        \end{equation*}
        \end{proof}
    Finally, we use the convergences derived in this section so far to pass to the limit $\artprs\to 0$ in the energy inequality completing the proof of 
    Theorem~\ref{main result:ex wkSol}.
    \begin{proof}[Proof of \Cref{main result:ex wkSol}]
        In view of \Cref{lem:wk cv(delt)} and \Cref{lem:str cv dens}, we only have to show that the triplet $(\rho,\uvec,c)$ satisfies the energy inequality \eqref{def wsol Eineq}.
        For $\artprs\in(0,1)$, we have that the triplet $(\rhod,\uvecd,\cd)$ satisfies the energy inequality \eqref{wk sol rNSK energy(delt)}, that is,
        \begin{equation}\label{energy ineq(delta)}
        \begin{aligned}
            &-\int_0^T  \, E_\artprs[\rhod,\rhod\uvecd,\cd] \partial_t \psi \, \dt 
            +
            \int_0^T \psi \int_\Omega 
            \Svisc(\nablax\uvecd):\nablax\uvecd + \paracoup|\partial_t \cd|^2
            \, \dx \, \dt
            \\
            &\leq 
            E_{\artprs}[\rho_{0,\artprs},(\rho\uvec)_{0,\artprs},c_0] \psi(0)
        \end{aligned}
        \end{equation}
        for any $\psi \in C^\infty_c([0,T))$ satisfying $\psi \geq 0$, where $E_\artprs[\cdot,\cdot,\cdot]$ is defined in Proposition~\ref{prop:wkSol rNSK(delta)}.
        By virtue of \Cref{lem:impr prs}, \Cref{lem:wk cv(delt)}, and \Cref{lem:str cv dens}, we have that
        \begin{equation*}
            E_{\artprs}[\rhod,\rhod\uvecd,\cd] \weak E[\rho,\rho\uvec,c] \quad \text{in } L^1((0,T)).
        \end{equation*}
        By using the latter convergence, the weak lower semi-continuity of convex functionals, and \eqref{mollify III}, we deduce from \eqref{energy ineq(delta)} after passing to the limit $\artprs\to 0$ that the inequality
        \begin{equation*}
            \begin{aligned}
            -\int_0^T  E[\rho,\rho\uvec,c] \partial_t \psi \, \dt 
            +
            \int_0^T \psi \int_\Omega 
            \Svisc(\nablax\uvec):\nablax\uvec + \paracoup|\partial_t c|^2
            \, \dx \, \dt
            \leq 
            E[\rho_{0},(\rho\uvec)_{0},c_0] \psi(0)
            \end{aligned}
        \end{equation*}
        holds for any test function $\psi \in C^\infty_c([0,T))$ satisfying $\psi \geq 0$, where $E[\cdot,\cdot,\cdot]$ is defined in \eqref{def E(t)}.
        This is precisely \eqref{def wsol Eineq}.
    \end{proof}

%% file: conclusions.tex
In this work, we have investigated the IBVP of the isothermal rNSKE from a mathematical point of view.
This system can be seen as an approximate system of the isothermal NSKE and, in particular, as a model for a compressible viscous two-phase fluid if the underlying pressure function is of Van-der-Waals type.
Based on the approximation scheme in \cite{FeireislNovotnyPetzeltova2001} we have shown the global-in-time existence of finite energy weak-solutions (see \Cref{main result:ex wkSol}) for a broad variety of pressure functions generalizing the results in \cite{FeireislNovotnyPetzeltova2001,Feireisl2002} to the rNSKE.
This result includes non-monotone pressure functions of Van-der-Waals type and complements the results in \cite{RohdeWendt2025a,ChaudhuriRohdeWendt2025}, where the global-in-time existence of finite energy weak solutions to the IBVP associated with the 3D rNSKE has been postulated.
\par
As an application, for the single-phase compressible NSE in a porous domain, several homogenization results have been derived depending on the size of the obstacles and their mutual distances.
Without being exhaustive we refer the interested reader to \cite{BasaricChaudhuri2024, BellaOschmann2023, BLMO2025, DieningFeireislLu2017, FeireislLu2015, FeireislNovotnyTakahashi2010, HoeferKowalczykSchwarzacher2021, LuPokorny2021, LuSchwarzacher2018, Masmoudi2002, NecasovaOschmann2023, NecasovaPan2022, Oschmann2022, PokornySkrisovsky2021a} and the references contained therein.
Concerning the compressible NSKE, up to our knowledge, the only rigorous homogenization result is due to \cite{RohdeWolff2020} focusing on the case where the size of the obstacles is comparable to their mutual distances.
There, the authors replace the third-order differential operator in the momentum equation $\eqref{NSK eqs}_2$ with a smooth convolution operator.
With the global-in-time existence result \Cref{main result:ex wkSol} at hand we plan to derive corresponding homogenization results for the rNSKE \cite{OschmannWendt2026-hom}.
\par
Recently, a parabolic relaxation formulation of the non-isothermal NSK system has been proposed in \cite{KeimMunzRohde2023}. 
It would be very interesting to develop a weak solution framework for this non-isothermal relaxation system including a corresponding global-in-time existence result.

%% file: appendix.tex
Here we state a weak compactness result concerning the effective viscous flux in its general form, which we will use frequently in this paper (see \Cref{sec:VanArtVisc} and \Cref{sec:VanArtPres}).
This result extends the general version from \cite[Proposition~7.36]{NovotnyStraskraba2004} in the sense that we allow for one more term in the continuity and in the momentum equation, respectively, as well as include the 2D case.
The full extension is not needed for the analysis in this work, however, we think that this result is of independent interest in view of future applications.
Although most of the proof follows the same lines as in \cite[Proposition~7.36]{NovotnyStraskraba2004}, we decided to give a full proof for the sake of completeness.
\begin{proposition}\label{prop:EVF(general)}
    Let $\Omega \subseteq \RR^d$, $d\in\{2,3\}$, be a bounded domain and let $T,\mu>0$, $\eta \geq 0$.
    Suppose that\footnote{In 3D, an explicit choice of exponents could be $q=\frac52$, $r=s=\sigma=w=2$, and $z=\frac{12}{5}$.}
    \begin{equation}\label{EVF(general) conditions exponents}
        \begin{aligned}
            &1<q,r,s,z,\sigma<\infty,
            \quad 
            \max\{2,r^\prime\}
            \leq 
            w\leq 
            \infty,
            \quad 
            z^\prime <\sigma^\ast,
            \\
            &\frac{1}{q} + \frac{1}{2^\ast} < \frac{1}{z^\prime},
            \quad 
            \frac{1}{z} + \frac{1}{q^\ast} < 1,
            \quad 
            \frac{1}{s} + \frac{1}{q^\ast} < 1,
        \end{aligned}
    \end{equation}
    and assume that
    \begin{equation}\label{EVF(general) cv}
    \begin{aligned}
        &\qvec_n \to \qvec \quad \text{in } \Cw([0,T];L^z(\Omega;\RR^d)),
        \quad
        \uvec_n \weak \uvec \quad \text{in } L^2(0,T;W^{1,2}_0(\Omega;\RR^d)),
        \\
        &p_n \weak p \quad \text{in } L^r(\OmegaT),
        \quad
        \Fvec_n \weak \Fvec \quad \text{in } L^s(\OmegaT;\RR^d),
        \\
        &\Gvec_n \to \Gvec \quad \text{in } L^{w^\prime}(0,T;W^{-1,w^\prime}(\Omega;\RR^d))
        \quad
        g_n \to g \quad \text{in } \Cw([0,T];L^q(\Omega)),
        \\
        &g_n \weakstar g \quad \text{in } L^w(\OmegaT),
        \quad
        \xi f_n \weak \xi f \quad \text{in } L^2(0,T;W^{-1,2}(\Omega)) \quad \forall\, \xi \in C^\infty_c(\Omega),
        \\
        &\nablax\Deltax^{-1}\big[\xi f_n\big] \to \nablax\Deltax^{-1}\big[\xi f\big] \quad \text{in } L^2(0,T;L^{z^\prime}(\Omega;\RR^d)) \quad \forall\, \xi \in C^\infty_c(\Omega),
        \\
        &h_n \weak h \quad \text{in } L^2(0,T;L^\sigma(\Omega)),
    \end{aligned}
    \end{equation}
    where $\Delta^{-1}$ denotes the inverse Laplace operator on $\RR^d$.
    For $n\in\NN$, suppose that 
    \begin{equation}\label{EVF(general) cont}
        \begin{aligned}
            &\partial_t g_n + \div(g_n\uvec_n) = f_n + h_n \quad \text{in } \mathcal{D}^\prime(\OmegaT)
        \end{aligned}
    \end{equation}
    and
    \begin{equation}\label{EVF(general) mom}
        \begin{aligned}
            \partial_t \qvec_n 
            +
            \div(\qvec_n \otimes \uvec_n)
            +
            \nablax p_n
            -
            \div\Svisc(\nablax\uvec_n) 
            = 
            \Fvec_n
            +
            \Gvec_n
            \quad \text{in } \mathcal{D}^\prime(\OmegaT;\RR^d),
        \end{aligned}
    \end{equation}
    where 
    \begin{equation*}
        \Svisc(A) 
        := 
        \mu \left(A + A^{\mathrm{T}} - \frac{2}{d} \mathrm{Tr}A\right) + \eta \,\mathrm{Tr} A
        \quad \text{for } A\in \RR^{d\times d}.
    \end{equation*}
    Then the relation
    \begin{equation}\label{EVF(general) relation}
        \begin{aligned}
            &\lim\limits_{n\to\infty}
            \int_0^T \int_\Omega 
            \varphi g_n \left( p_n - \left(\eta + \frac{2(d-1)}{d}\mu\right)\div\uvec_n\right)
            \, \dd x \, \dd t
            \\
            &\quad =
            \int_0^T \int_\Omega \varphi g \left( p - \left(\eta + \frac{2(d-1)}{d}\mu\right)\div\uvec \right)\, \dd x \, \dd t
        \end{aligned}
    \end{equation}
    holds for any test function $\varphi \in C^\infty_c(\OmegaT)$.
\end{proposition}
\begin{proof}
    We assume without loss of generality that the boundary of $\Omega$ is smooth. 
    This is indeed no restriction, since the arguments in the sequel are all local.
    We fix $\xi \in C^\infty_c(\Omega)$.
    Due to \eqref{EVF(general) cont}, we have that
    \begin{equation}\label{EVF cont eq}
    \begin{aligned}
        \partial_t \nablax \Deltax^{-1}\big[g_n\xi\big]
        &=
        -
        \nablax \div\Deltax^{-1}\big[g_n \uvec_n \xi\big]
        +
        \nablax \Deltax^{-1}\big[g_n\uvec_n \cdot \nablax \xi\big]
        \\
        &\quad \quad 
        +
        \nablax \Deltax^{-1}\big[f_n\xi\big]
        +
        \nablax \Deltax^{-1}\big[h_n\xi\big]
        \qquad \text{in } \mathcal{D}^\prime(\OmegaT;\RR^d).
    \end{aligned}
    \end{equation}
    By virtue of the Mikhlin multiplier theorem (see e.g.~\cite[Theorem~9]{FeireislNovotny2017singlim}), we have that
    \begin{equation}\label{EVF embed Riesz}
        \begin{aligned}
            \nablax \Deltax^{-1}\colon L^s(\RR^d) \to W^{1,s}(\RR^d;\RR^d),
            \quad 
            \nablax \otimes\nablax \Deltax^{-1}\colon L^s(\RR^d) \to L^s(\RR^d;\RR^{d\times d})
        \end{aligned}
    \end{equation}
    define both continuous linear operators for any $s \in (1,\infty)$.
    In particular, we have due to \eqref{EVF(general) conditions exponents}, \eqref{EVF(general) cv}, and \eqref{EVF cont eq} that $\nablax \Deltax^{-1}\big[g_n \xi\big]\in L^w(0,T;W^{1,w}(\Omega;\RR^d))\cap L^\infty(0,T;W^{1,q}(\Omega;\RR^d))$ as well as $\partial_t \nablax \Deltax^{-1}\big[g_n\xi\big] \in L^2(0,T;L^{z^\prime}(\Omega;\RR^d))$.
    For fixed functions $\psi \in C^\infty_c((0,T))$, $\phi \in C^\infty_c(\Omega)$ we thus have by a density argument that the function $\psi \phi \nablax\Deltax^{-1}\big[g_n\xi\big]$ is a valid test function for \eqref{EVF(general) mom}.
    Testing \eqref{EVF(general) mom} with the function $\psi \phi \nablax\Deltax^{-1}\big[g_n\xi\big]$, we obtain after a long calculation using integration by parts several times, as well as \eqref{EVF cont eq}, that
    \begin{equation}\label{EVF mom(n)}
        \begin{aligned}
            &\int_0^T \psi \int_\Omega 
            \phi\xi  g_n \left(p_n - \left(\eta + \frac{2(d-1)}{d}\shearvisc\right) \div\uvec_n\right)
            \, \dd x \, \dd t
            \\
            &=
            -\int_0^T \psi \int_\Omega 
            \nablax \phi \cdot p_n \nablax \Deltax^{-1}\big[ g_n \xi \big] 
            \, \dd x \, \dd t
            \\&
            \quad
            +
            \left(\eta + \frac{d-2}{d}\shearvisc\right) \int_0^T \psi \int_\Omega 
            \nablax \phi \cdot \div\uvec_n \nablax \Deltax^{-1}\big[g_n \xi\big]
            \, \dd x \, \dd t
            \\
            &\quad 
            +
            \mu \int_0^T \psi \int_\Omega 
            \nablax \phi \cdot \nablax \uvec_n \cdot \nablax \Delta_x^{-1}\big[ g_n \xi \big]
            \, \dd x \, \dd t
            \\
            &\quad
            -
            \mu \int_0^T \psi \int_\Omega 
            \nablax \phi \cdot \nablax \otimes\nablax \Deltax^{-1} \big[g_n \xi\big] \cdot \uvec_n
            \, \dd x \, \dd t
            \\
            &\quad
            +
            \mu \int_0^T \psi \int_\Omega 
            \nablax \phi \cdot \uvec_n g_n \xi 
            \, \dd x \, \dd t
            -
            \int_0^T \psi \int_\Omega 
            \phi \Fvec_n \cdot \nablax \Deltax^{-1}\big[ g_n \xi \big]
            \, \dd x \, \dd t
            \\
            &\quad
            -
            \big\langle \Gvec_n, \psi\phi \nablax\Deltax^{-1}\big[g_n \xi\big] \big\rangle
            -
            \int_0^T \psi \int_\Omega 
            \nablax \phi \cdot \qvec_n \otimes \uvec_n \cdot \nablax \Deltax^{-1} \big[g_n \xi\big]
            \, \dd x \, \dd t
            \\
            &\quad
            -
            \int_0^T \psi \int_\Omega 
            \phi \qvec_n \cdot \nablax \Deltax^{-1}\big[f_n \xi \big]
            \, \dd x \, \dd t
            -
            \int_0^T \psi \int_\Omega 
            \qvec_n \cdot \nablax\Deltax^{-1}\big[h_n \xi\big]
            \, \dd x \, \dd t
            \\
            &\quad
            -
            \int_0^T \partial_t \psi \int_\Omega 
            \phi \qvec_n \cdot \nablax \Deltax^{-1}\big[g_n \xi\big]
            \, \dd x \, \dd t
            +
            \int_0^T \psi \int_\Omega 
            \div \Deltax^{-1}\big[\phi \qvec_n\big] g_n \uvec_n \cdot \nablax \xi
            \, \dd x \, \dd t
            \\
            &\quad
            +
            \int_0^T \psi \int_\Omega 
            \uvec_n \cdot
            \bigg(
            g_n\xi \nablax \div \Deltax^{-1}\big[\qvec_n \phi\big] - \nablax \otimes\nablax \Deltax^{-1}\big[g_n \xi\big] \cdot \qvec_n \phi
            \bigg)
            \, \dd x \, \dd t.
        \end{aligned}
    \end{equation}
    We wish to pass to the limit $n\to \infty$ in \eqref{EVF mom(n)}.
    To do so, we deduce from \eqref{EVF(general) cv} some further convergences by exploiting compactness arguments.
    \\
    By virtue of the third and fourth relation in \eqref{EVF(general) conditions exponents}, we have that the Sobolev embeddings $L^q(\Omega)\hookrightarrow\hookrightarrow W^{-1,2}(\Omega)$, $L^z(\Omega)\hookrightarrow\hookrightarrow W^{-1,2}(\Omega)$, and $L^z(\Omega) \hookrightarrow\hookrightarrow W^{-1,\sigma^\prime}(\Omega)$ are all compact.
    In view of the first and sixth convergences in \eqref{EVF(general) cv}, this leads to
    \begin{equation}\label{EVF str cv qvec(n) + g(n)}
        \begin{aligned}
            &\qvec_n \to \qvec \quad \text{in } L^{\infty^-}(0,T;W^{-1,2}(\Omega;\RR^d)),
            \quad 
            \qvec_n \to \qvec \quad \text{in } L^{\infty^-}(0,T;W^{-1,\sigma^\prime}(\Omega;\RR^d)),
            \\
            &g_n \to g \quad \text{in } L^{\infty^-}(0,T;W^{-1,2}(\Omega)),
        \end{aligned}
    \end{equation}
    and moreover, by virtue of the continuity of the second operator in \eqref{EVF embed Riesz},
    \begin{equation}\label{EVF str cv Riesz(g_n xi)}
        \begin{aligned}
            \nablax \otimes \nablax \Deltax^{-1}\big[g_n \xi\big] \to \nablax \otimes \nablax \Deltax^{-1}\big[g \xi\big]
            \quad \text{in } L^{\infty^-}(0,T;W^{-1,2}(\Omega;\RR^{d\times d})).
        \end{aligned}
    \end{equation}
    From the continuity of the second operator in \eqref{EVF embed Riesz} we further conclude from the seventh and tenth convergences in \eqref{EVF(general) cv} that
    \begin{equation}\label{EVF wk cv L^w(W^1,w)}
    \begin{aligned}
        &\psi \phi \nablax\Deltax^{-1}\big[g_n \xi\big] \weak \psi \phi \nablax \Deltax^{-1}\big[g\xi\big] \quad \text{in } L^w(0,T;W^{1,w}_0(\Omega;\RR^d)),
        \\
        &\psi\phi \nablax\Deltax^{-1}\big[h_n\xi\big] \weak \psi\phi\nablax\Deltax^{-1}\big[h\xi\big] \quad \text{in } L^2(0,T;W^{1,\sigma}_0(\Omega;\RR^d)).
    \end{aligned}
    \end{equation}
    Combining the second convergence in \eqref{EVF(general) cv} with \eqref{EVF str cv qvec(n) + g(n)} and using the Sobolev embedding $W^{1,2}(\Omega) \hookrightarrow L^{\overline{2}}(\Omega)$ leads to
    \begin{equation}\label{EVF wk cv qvec otimes uvec(n) + g uvec(n)}
        \begin{aligned}
            \qvec_n \otimes \uvec_n \weak \qvec\otimes\uvec \quad \text{in } L^{2}(0,T;L^{\frac{\overline{2}z}{\overline{2}+z}}(\Omega;\RR^{d\times d}),
            \quad 
            g_n \uvec_n \weak g \uvec \quad \text{in } L^2(0,T;L^{\frac{\overline{2}q}{\overline{2}+q}}(\Omega;\RR^d)).
        \end{aligned}
    \end{equation}
    In view of the compact Sobolev embeddings $W^{1,z}(\Omega)\hookrightarrow\hookrightarrow L^{(z^\ast)^-}(\Omega)$ and $W^{1,q}(\Omega)\hookrightarrow\hookrightarrow L^{(q^\ast)^-}(\Omega)$ and the continuity of the first operator in \eqref{EVF embed Riesz}, we conclude from the first and sixth convergence in \eqref{EVF(general) cv} that
    \begin{equation}\label{EVF str cv inv divergences}
        \begin{aligned}
            &\nablax \Deltax^{-1}\big[\qvec_n \xi\big] \to \nablax \Deltax^{-1}\big[\qvec \xi\big]
            \quad \text{in } L^{\infty^-}(0,T;L^{(z^\ast)^-}(\Omega;\RR^{d\times d})),
            \\ 
            &\nablax \Deltax^{-1}\big[g_n \xi\big] \to \nablax \Deltax^{-1}\big[g \xi\big]
            \quad \text{in } L^{\infty^-}(0,T;L^{(q^\ast)^-}(\Omega;\RR^d)).
        \end{aligned}
    \end{equation}
    Due to the fourth relation in \eqref{EVF(general) conditions exponents}, we conclude further from the first and sixth convergences in \eqref{EVF(general) cv} by using \cite[Theorem~11.34]{FeireislNovotny2017singlim}, that
    \begin{equation*}
    \begin{aligned}
        &g_n(t)\xi\nablax\div\Deltax^{-1}\big[\qvec_n(t)\phi\big] - \nablax \otimes\nablax \Deltax^{-1} \big[g_n(t)\xi\big]\cdot \qvec_n(t) \phi 
        \\
        &\weak
        g(t) \xi \nablax \div\Deltax^{-1}\big[\qvec(t)\phi\big] - \nablax \otimes \nablax \Deltax^{-1}\big[g(t) \xi\big] \cdot \qvec(t) \phi
        \quad \text{in } L^{\frac{zq}{z+q}}(\Omega;\RR^d)
    \end{aligned}
    \end{equation*}
    for any $t \in [0,T]$.
    This implies in view of the uniform bounds implied by the first and sixth convergences in \eqref{EVF(general) cv}, the continuity of the second operator in \eqref{EVF embed Riesz}, as well as the compactness of the embedding $L^{\frac{zq}{z+q}}(\Omega) \hookrightarrow\hookrightarrow W^{-1,2}(\Omega)$, which holds due to the fourth relation in \eqref{EVF(general) conditions exponents}, that
    \begin{equation}\label{EVF str cv commutator}
        \begin{aligned}
            &g_n\xi\nablax\div\Deltax^{-1}\big[\qvec_n\phi\big] - \nablax \otimes\nablax \Deltax^{-1} \big[g_n\xi\big]\cdot \qvec_n \phi 
            \\
            &\to 
            g \xi \nablax \div\Deltax^{-1}\big[\qvec\phi\big] - \nablax \otimes \nablax \Deltax^{-1}\big[g \xi\big] \cdot \qvec \phi
            \quad \text{in } L^{\infty^-}(0,T;W^{-1,2}(\Omega;\RR^d).
        \end{aligned}
    \end{equation}
    In view of \eqref{EVF(general) cv} and \eqref{EVF str cv qvec(n) + g(n)}--\eqref{EVF str cv commutator}, we may pass to the limit $n\to \infty$ in \eqref{EVF mom(n)} and obtain
    \begin{equation}\label{EVF I}
        \begin{aligned}
            &\lim\limits_{n\to\infty}\int_0^T \psi \int_\Omega 
            \phi\xi  g_n \left(p_n - \left(\eta + \frac{2(d-1)}{d}\shearvisc\right) \div\uvec_n\right)
            \, \dd x \, \dd t
            \\
            &=
            -\int_0^T \psi \int_\Omega 
            \nablax \phi \cdot p \nablax \Deltax^{-1}\big[ g \xi \big] 
            \, \dd x \, \dd t
            \\
            &\quad
            +
            \left(\eta + \frac{d-2}{d}\shearvisc\right) \int_0^T \psi \int_\Omega 
            \nablax \phi \cdot \div\uvec \nablax \Deltax^{-1}\big[g \xi\big]
            \, \dd x \, \dd t
            \\
            &\quad
            +
            \mu \int_0^T \psi \int_\Omega 
            \nablax \phi \cdot \nablax \uvec \cdot \nablax \Delta_x^{-1}\big[ g \xi \big]
            \, \dd x \, \dd t
            \\
            &\quad
            -
            \mu \int_0^T \psi \int_\Omega 
            \nablax \phi \cdot \nablax \otimes\nablax \Deltax^{-1} \big[g \xi\big] \cdot \uvec
            \, \dd x \, \dd t
            \\
            &\quad
            +
            \mu \int_0^T \psi \int_\Omega 
            \nablax \phi \cdot \uvec g \xi 
            \, \dd x \, \dd t
            -
            \int_0^T \psi \int_\Omega 
            \phi \Fvec \cdot \nablax \Deltax^{-1}\big[ g \xi \big]
            \, \dd x \, \dd t
            \\
            &\quad
            - 
            \big\langle \Gvec, \psi \phi \nablax \Deltax^{-1}\big[g \xi\big]\big \rangle
            -
            \int_0^T \psi \int_\Omega 
            \nablax \phi \cdot \qvec \otimes \uvec \cdot \nablax \Deltax^{-1} \big[g \xi\big]
            \, \dd x \, \dd t
            \\
            &\quad 
            -
            \int_0^T \psi \int_\Omega 
            \phi \qvec \cdot \nablax \Deltax^{-1}\big[f \xi \big]
            \, \dd x \, \dd t
            -
            \int_0^T \psi \int_\Omega 
            \phi \qvec \cdot \nablax\Deltax^{-1}\big[h \xi\big] 
            \, \dd x \, \dd t
            \\
            &\quad 
            -
            \int_0^T \partial_t \psi \int_\Omega 
            \phi \qvec \cdot \nablax \Deltax^{-1}\big[g \xi\big]
            \, \dd x \, \dd t
            +
            \int_0^T \psi \int_\Omega 
            \div \Deltax^{-1}\big[\phi \qvec\big] g \uvec \cdot \nablax \xi
            \, \dd x \, \dd t
            \\
            &\quad
            +
            \int_0^T \psi \int_\Omega 
            \uvec \cdot
            \bigg(
            g\xi \nablax \div \Deltax^{-1}\big[\qvec \phi\big] - \nablax \otimes\nablax \Deltax^{-1}\big[g \xi\big] \cdot \qvec \phi
            \bigg)
            \, \dd x \, \dd t.
        \end{aligned}
    \end{equation}
    By the same token, we may pass to the limit $n\to\infty$ in \eqref{EVF(general) cont} and \eqref{EVF(general) mom} to obtain
    \begin{equation}\label{EVF limit cont}
        \begin{aligned}
            \partial_t g
            +
            \div(g\uvec)
            =
            f + h
            \quad \text{in } \mathcal{D}^\prime(\OmegaT),
        \end{aligned}
    \end{equation}
    and
    \begin{equation}\label{EVF limit mom}
        \begin{aligned}
            \partial_t \qvec 
            +
            \div(\qvec\otimes\uvec)
            +
            \nablax p
            -
            \div\Svisc(\nablax \uvec)
            =
            \Fvec
            +
            \Gvec
            \quad \text{in } \mathcal{D}^\prime(\OmegaT;\RR^d).
        \end{aligned}
    \end{equation}
    From \eqref{EVF limit cont}, we deduce
    \begin{equation}\label{EVF limit time deriv test function}
        \begin{aligned}
            \partial_t\nablax\Deltax^{-1}\big[g\xi\big]
            &=
            -\nablax\div\Deltax^{-1}[g\uvec\xi]
            +
            \nablax \Deltax^{-1}\big[g\uvec\cdot \nablax \xi\big]
            \\
            &\quad\quad +
            \nablax \Deltax^{-1}\big[f\xi\big]
            +
            \nablax \Deltax^{-1}\big[h \xi\big]
            \qquad \text{in } \mathcal{D}^\prime(\OmegaT;\RR^d).
        \end{aligned}
    \end{equation}
    Due to \eqref{EVF(general) cv}, \eqref{EVF limit time deriv test function}, as well as the continuity of the operators in \eqref{EVF embed Riesz}, we have that $\nablax \Deltax^{-1}\big[g\xi\big]\in L^w(0,T;W^{1,w}(\Omega;\RR^d))\cap L^\infty(0,T;W^{1,q}(\Omega;\RR^d))$ and $\partial_t \nablax \Deltax^{-1}\big[g\xi\big]\in L^2(0,T;L^{z^\prime}(\Omega;\RR^d))$.
    In particular, we have by a density argument that the function $\psi\phi\nablax \Deltax^{-1}\big[g\xi\big]$ is a valid test function for \eqref{EVF limit mom}.
    Testing \eqref{EVF limit mom} with the function $\psi\phi\nablax \Deltax^{-1}\big[g\xi\big]$, we obtain after a long calculation using integration by parts several times, as well as \eqref{EVF limit cont}, that
    \begin{equation}\label{EVF II}
        \begin{aligned}
            &\int_0^T \psi \int_\Omega 
            \phi\xi  g \left(p - \left(\eta + \frac{2(d-1)}{d}\shearvisc\right) \div\uvec\right)
            \, \dd x \, \dd t
            \\
            &=
            -\int_0^T \psi \int_\Omega 
            \nablax \phi \cdot p \nablax \Deltax^{-1}\big[ g \xi \big] 
            \, \dd x \, \dd t
            \\
            &\quad
            +
            \left(\eta + \frac{d-2}{d}\shearvisc\right) \int_0^T \psi \int_\Omega 
            \nablax \phi \cdot \div\uvec_n \nablax \Deltax^{-1}\big[g \xi\big]
            \, \dd x \, \dd t
            \\
            &\quad
            +
            \mu \int_0^T \psi \int_\Omega 
            \nablax \phi \cdot \nablax \uvec \cdot \nablax \Delta_x^{-1}\big[ g \xi \big]
            \, \dd x \, \dd t
            \\
            &\quad
            -
            \mu \int_0^T \psi \int_\Omega 
            \nablax \phi \cdot \nablax \otimes\nablax \Deltax^{-1} \big[g \xi\big] \cdot \uvec
            \, \dd x \, \dd t
            \\
            &\quad
            +
            \mu \int_0^T \psi \int_\Omega 
            \nablax \phi \cdot \uvec g \xi 
            \, \dd x \, \dd t
            -
            \int_0^T \psi \int_\Omega 
            \phi \Fvec \cdot \nablax \Deltax^{-1}\big[ g \xi \big]
            \, \dd x \, \dd t
            \\
            &\quad
            -
            \big\langle \Gvec, \psi \phi \nablax \Deltax^{-1}\big[g \xi\big] \big\rangle
            -
            \int_0^T \psi \int_\Omega 
            \nablax \phi \cdot \qvec \otimes \uvec \cdot \nablax \Deltax^{-1} \big[g \xi\big]
            \, \dd x \, \dd t
            \\
            &\quad
            -
            \int_0^T \psi \int_\Omega 
            \phi \qvec \cdot \nablax \Deltax^{-1}\big[f \xi\big]
            \, \dd x \, \dd t
            -
            \int_0^T \psi \int_\Omega 
            \phi \qvec \cdot \nablax \Deltax^{-1} \big[h \xi\big]
            \, \dd x \, \dd t
            \\
            &\quad
            -
            \int_0^T \partial_t \psi \int_\Omega 
            \phi \qvec \cdot \nablax \Deltax^{-1}\big[g \xi\big]
            \, \dd x \, \dd t
            +
            \int_0^T \psi \int_\Omega 
            \div \Deltax^{-1}\big[\phi \qvec\big] g \uvec \cdot \nablax \xi
            \, \dd x \, \dd t
            \\
            &\quad
            +
            \int_0^T \psi \int_\Omega 
            \uvec \cdot
            \bigg(
            g\xi \nablax \div \Deltax^{-1}\big[\qvec \phi\big] - \nablax \otimes\nablax \Deltax^{-1}\big[g \xi\big] \cdot \qvec \phi
            \bigg)
            \, \dd x \, \dd t.
        \end{aligned}
    \end{equation}
    Comparing \eqref{EVF I} and \eqref{EVF II} leads to \eqref{EVF(general) relation}.
\end{proof}